\documentclass[a4paper,oneside,11pt]{article}

\usepackage{geometry}                
\usepackage{graphicx}
\usepackage{amsmath,amsfonts,amssymb,amsthm}
\usepackage{epstopdf}
\usepackage{pstricks}
\usepackage[utf8]{inputenc} 
\usepackage[T1]{fontenc} 
 
\usepackage[english]{babel} 

\title{Spectral rigidity for primitive elements of $F_N$}
\author{Camille Horbez}

\usepackage[all]{xy}
\usepackage{bbm}
\begin{document}
\maketitle
\newtheorem{de}{Definition} [section]
\newtheorem{theo}[de]{Theorem} 
\newtheorem{prop}[de]{Proposition}
\newtheorem{lemma}[de]{Lemma}
\newtheorem{cor}[de]{Corollary}
\newtheorem{propd}[de]{Proposition-Definition}

\theoremstyle{remark}
\newtheorem{rk}[de]{Remark}
\newtheorem{ex}[de]{Example}
\newtheorem{question}[de]{Question}

\normalsize

\addtolength\topmargin{-.5in}
\addtolength\textheight{1.in}
\addtolength\oddsidemargin{-.045\textwidth}
\addtolength\textwidth{.09\textwidth}

\begin{abstract}
Two trees in the boundary of outer space are said to be \emph{primitive-equivalent} whenever their translation length functions are equal in restriction to the set of primitive elements of $F_N$. We give an explicit description of this equivalence relation, showing in particular that it is nontrivial. This question is motivated by our description of the horoboundary of outer space for the Lipschitz metric in \cite{Hor14-2}. Along the proof, we extend a theorem due to White about the Lipschitz metric on outer space to trees in the boundary, showing that the infimal Lipschitz constant of an $F_N$-equivariant map between the metric completion of any two minimal, very small $F_N$-trees is equal to the supremal ratio between the translation lengths of the elements of $F_N$ in these trees. We also provide approximation results for trees in the boundary of outer space.
\end{abstract}

\section*{Introduction}
\emph{Outer space} $CV_N$ was introduced by Culler and Vogtmann in \cite{CV86} with a view to studying the group $\text{Out}(F_N)$ of outer automorphisms of a finitely generated free group. The space $CV_N$ (or its unprojectivized version $cv_N$) is the space of equivariant homothety (isometry) classes of simplicial free, minimal, isometric actions of $F_N$ on simplicial metric trees. The \emph{translation length} of an element $g$ of a group $G$ acting on an $\mathbb{R}$-tree $T$ is defined as $||g||_T:=\inf_{x\in T}d_T(x,gx)$. Translation lengths provide an embedding of $cv_N$ into $\mathbb{R}^{F_N}$, whose image has projectively compact closure, as was proved by Culler and Morgan \cite{CM87}. This compactification $\overline{CV_N}$ of outer space was described by Cohen and Lustig \cite{CL95} and Bestvina and Feighn \cite{BF94} as the space of homothety classes of minimal, very small actions of $F_N$ on $\mathbb{R}$-trees.

Instead of considering the translation lengths of all elements of $F_N$, one might only look at the subset $\mathcal{P}_N$ of primitive elements of $F_N$, i.e. those elements that belong to some free basis of $F_N$, and get another compactification of $CV_N$ as a subspace of $\mathbb{PR}^{\mathcal{P}_N}$, which we call the \emph{primitive compactification} of outer space. Our original motivation for describing this alternative compactification comes from the desire to get a description of the horoboundary of outer space with respect to the so-called Lipschitz metric on $CV_N$, whose systematic study was initiated by Francaviglia and Martino in \cite{FM11} (the distance between $T,T'\in CV_N$ is defined as the logarithm of the infimal Lipschitz constant of an $F_N$-equivariant map from the covolume $1$ representative of $T$ to the covolume $1$ representative of $T'$). This in turn is motivated by the question of describing the behaviour of random walks on $\text{Out}(F_N)$: in \cite{Hor14-2}, we derive an Oseledets-like theorem about possible growth rates of elements of $F_N$ under iteration of random automorphisms of the free group from a description of the horoboundary of outer space. It turns out that the horocompactification of outer space is isomorphic to the primitive compactification \cite{Hor14-2}.

Describing the primitive compactification of outer space requires understanding the lack of rigidity of the set $\mathcal{P}_N$ in $\overline{cv_N}$, i.e. giving a description of the equivalence relation that identifies $T,T'\in\overline{cv_N}$ whenever their translation length functions are equal in restriction to $\mathcal{P}_N$. This question may also be considered of independent interest, as part of a much wider class of problems arising in several contexts. The \emph{marked length spectrum rigidity conjecture} is still open for Riemannian manifolds : it is not known whether the isometry type of a negatively curved closed Riemannian manifold $M$ is determined by the length of the geodesic representatives of the free homotopy classes of curves in $M$ (this was proven to be true in the case of surfaces by both Croke \cite{Cro90} and Otal \cite{Ota90}). Culler and Morgan's result states that an analogue of the marked length spectrum rigidity conjecture holds for free groups. It is then natural to ask, given a subset $S\subseteq F_N$, whether it is spectrally rigid in $cv_N$ (or in the closure $\overline{cv_N}$), i.e. whether the restriction to $S$ of the marked length spectrum of a tree $T\in cv_N$ (or $T\in\overline{cv_N}$) determines $T$ up to $F_N$-equivariant isometry. Several results have already been obtained for some classes of subsets of $F_N$. Smillie and Vogtmann have shown that no finite subset of $F_N$ is spectrally rigid in $\overline{cv_N}$ for $N\ge 3$ \cite{SV92}. Kapovich has proved that almost every trajectory of the simple nonbacktracking random walk on $F_N$ with respect to any free basis yields a subset of $F_N$ that is spectrally rigid in $cv_N$ \cite{Kap12}. Ray has proved that for all $\phi\in\text{Aut}(F_N)$ and all $g\in F_N$, the $\phi$-orbit of $g$ is not spectrally rigid in $cv_N$ \cite{Ray11}. Finally, Carette, Francaviglia, Kapovich and Martino have shown that the set $\mathcal{P}_N$ (and more generally the $H$-orbit of any $g\in F_N$ for $N\ge 3$, where $H$ is any subgroup of $\text{Aut}(F_N)$ that projects to a nontrivial normal subgroup of $\text{Out}(F_N)$) is spectrally rigid in $cv_N$ \cite{CFKM12}, and they raise the question of whether $\mathcal{P}_N$ is spectrally rigid in $\overline{cv_N}$ for $N\ge 3$ (for $N=2$, they provide a counterexample, attributed to Tao). 

An element of $F_N$ is \emph{simple} if it belongs to some proper free factor of $F_N$. One can define another equivalence relation on $\overline{cv_N}$, by saying that two trees are \emph{simple-equivalent} if they have the same translation length functions in restriction to the set of \emph{simple} elements of $F_N$. One easily checks that this equivalence relation is the same as the one define above using primitive elements (this is the content of Proposition \ref{primitive-separable} of the present paper), and it turns out that it is sometimes easier to work with the collection of simple elements of $F_N$ rather than primitive elements in our arguments.

Generalizing Tao's counterexample to higher ranks, we show that the set $\mathcal{P}_N$ is not spectrally rigid in $\overline{cv_N}$, yet we also show that this class of examples is the only obstruction to spectral rigidity of $\mathcal{P}_N$ in $\overline{cv_N}$. Our construction is the following. Let $T_0$ be a (non necessarily minimal) $F_{N-1}$-tree in which some point is fixed by an element of $F_{N-1}$ not contained in any proper free factor (the simplest example is to consider a tree $T_0$ reduced to a point, but one can also find more elaborated simplicial examples, as well as nonsimplicial examples by considering trees dual to a measured foliation on a surface with one single boundary component). Let $T\in \overline{cv_N}$ be a tree given by a graph of actions having $T_0$ as its only vertex tree, and having a single edge $e$ with trivial edge group. Two trees are said to be \emph{special-pull-equivalent} if they are both obtained from such a tree by partially or totally folding the edge $e$ from one or both of its extremities along translates of the form $ge$, where $g\in F_{N-1}$ does not belong to any proper free factor of $F_{N-1}$, see Figure \ref{fig-slide}, see also Section \ref{sec-slide-equiv} for precise definitions. The lack of spectral rigidity of the set $\mathcal{P}_N$ in $\overline{cv_N}$ is precisely given by this equivalence relation on trees. 

\begin{figure}
\begin{center}
\input{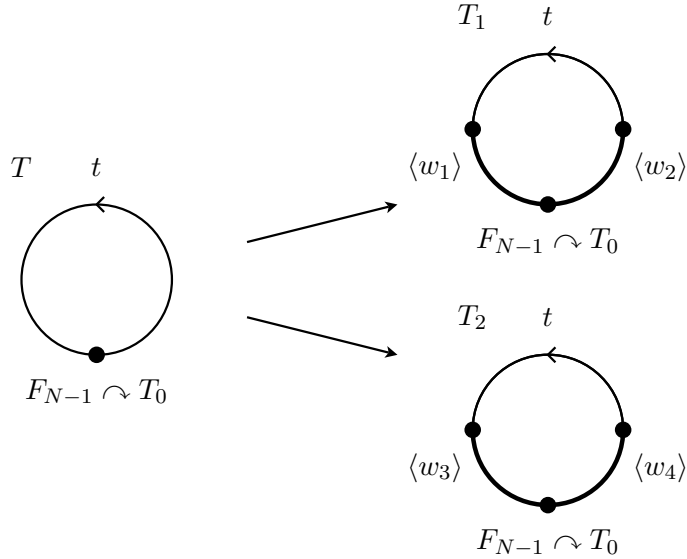}
\caption{The trees $T_1$ and $T_2$ are \emph{special-pull-equivalent} if $w_1,w_2,w_3,w_4$ do not belong to any proper free factor of $F_{N-1}$.}
\label{fig-slide}
\end{center}
\end{figure}

\begin{theo} \label{intro-equivalences}
For all $T,T'\in\overline{cv_N}$, the following assertions are equivalent.
\begin{itemize}
\item For all $g\in\mathcal{P}_N$, we have $||g||_T=||g||_{T'}$.
\item For all simple elements $g\in F_N$, we have $||g||_T=||g||_{T'}$.
\item The trees $T$ and $T'$ are special-pull-equivalent.
\end{itemize}
\end{theo}

Carette, Francaviglia, Kapovich and Martino's result, which states that the set of primitive elements is spectrally rigid in $cv_N$, is derived from Francaviglia and Martino's work \cite{FM11} about extremal Lipschitz distortion between trees $T,T'\in cv_N$. The key point, due to White, is that the minimal Lipschitz constant of an $F_N$-equivariant map from $T$ to $T'$ is also equal to $\Lambda(T,T'):=\sup_{g\in F_N}\frac{||g||_{T'}}{||g||_T}$, and this supremum is achieved on a finite set of primitive elements called \emph{candidates}, represented in the quotient graph $T/F_N$ by a special class of loops.

In order to study rigidity of $\mathcal{P}_N$ in $\overline{cv_N}$, we extend White's result to trees in the boundary of the unprojectivized outer space. Interested in the metric completion of outer space, Algom-Kfir extended it to the case where $T$ is a simplicial tree (possibly with nontrivial stabilizers) \cite[Proposition 4.5]{AK12}. We generalize it to arbitrary trees in $\overline{cv_N}$, thus answering Algom-Kfir's question \cite[Question 4.6]{AK12}. Given $T,T'\in\overline{cv_N}$, we define $\text{Lip}(T,T')$ to be the infimum of a Lipschitz constant of an $F_N$-equivariant map $f:T\to \overline{T'}$, where $\overline{T'}$ denotes the metric completion of $T'$, if such a map exists, and $\text{Lip}(T,T')=+\infty$ otherwise. We define $\Lambda(T,T'):=\sup_{g\in F_N}\frac{||g||_{T'}}{||g||_T}$ (where we take the conventions $\frac{0}{0}=0$ and $\frac{1}{0}=+\infty$). We prove equality between these two notions of stretching.

\begin{theo} \label{intro-White}
For all $T,T'\in\overline{cv_N}$, we have $\text{Lip}(T,T')={\Lambda}(T,T')$.
\end{theo}

Our proof of Theorem \ref{intro-White} relies on a structure theory of trees in the boundary of outer space. Levitt showed in \cite{Lev94} that any tree $T\in\overline{cv_N}$ splits as a graph of actions whose vertex trees have dense orbits (they can be reduced to a point). The case of trees with dense orbits, considered in Section \ref{sec-dense-case}, relies on two side results that provide us some approximation techniques. The first of these results extends work by Bestvina and Feighn \cite{BF94} and Guirardel \cite{Gui98}, and gives a way of approximating trees with dense orbits by free actions on simplicial metric trees.

\begin{theo}\label{intro-approx}
For all $T\in\overline{cv_N}$, the following assertions are equivalent.
\begin{itemize}
\item There exists a sequence $(T_n)_{n\in\mathbb{N}}$ of trees in $cv_N$ converging to $T$, such that for all $n\in\mathbb{N}$, there exists a $1$-Lipschitz $F_N$-equivariant map $f_n:T_n\to T$.
\item All arc stabilizers in $T$ are trivial. 
\end{itemize}
\end{theo}

\noindent The second side result we use enables us to build Lipschitz $F_N$-equivariant maps between trees in $\overline{cv_N}$ by a limiting process.

\begin{theo} \label{intro-limits}
Let $T$ and $T'$ be two very small $F_N$-trees, let $(T_n)_{n\in\mathbb{N}}$ (resp. $(T'_n)_{n\in\mathbb{N}}$) be a sequence of trees converging to $T$ (resp. $T'$) in the equivariant Gromov-Hausdorff topology, and let $(M_n)_{n\in\mathbb{N}}$ be a sequence of real numbers, satisfying $M:=\liminf_{n\to +\infty} M_n<+\infty$. Assume that for all $n\in\mathbb{N}$, there exists an $M_n$-Lipschitz $F_N$-equivariant map $f_n:T_n\to T'_n$. Then there exists an $M$-Lipschitz $F_N$-equivariant map $f:T\to \overline{T'}$.
\end{theo}

The rest of our proof of Theorem \ref{intro-White}, which is carried out in Section \ref{sec-White}, consists in extending the standard techniques in the proof of White's theorem to be able to deal with graphs of actions.

We also extend the notion of candidates to trees in the boundary of outer space. However, the supremum in the definition of ${\Lambda}(T,T')$ can no longer be taken over the set $\mathcal{P}_N$ in general. This property nevertheless holds true for trees that contain at least two orbits of edges with trivial stabilizer in their simplicial part. Further building on the work of Bestvina and Feighn \cite{BF94} and Guirardel \cite{Gui98}, we obtain an approximation result by such trees, which is used in Section \ref{sec-rigidity} to derive Theorem \ref{intro-equivalences}.
\\
\\
\indent The paper is organized as follows. In Section \ref{sec-CVn}, we review basic facts about outer space and $\mathbb{R}$-trees. In Section \ref{sec-beginning}, we prove that two trees in $\overline{cv_N}$ are primitive-equivalent if and only if they are simple-equivalent (Section \ref{sec-equivalences}). We also give (and discuss) the precise definition of the special-pull-equivalence relation (Section \ref{sec-slide-equiv}), and we show that special-pull-equivalent trees are simple-equivalent, using the Whitehead algorithm (Section \ref{sec-Whitehead}). We also define the primitive compactification of outer space (Section \ref{sec-primitive-compactification}). Section \ref{sec-approximations} is dedicated to the proof of Theorem \ref{intro-approx}, by using techniques of approximations by geometric trees. We also prove an approximation result by trees having at least two orbits of edges with trivial stabilizers (Theorem \ref{approximation-by-separable}). Section \ref{sec-limits} is devoted to the proof of Theorem \ref{intro-limits}. In the next two sections, we prove our extension of White's theorem to trees in $\overline{cv_N}$. The case of trees with dense orbits is treated in Section \ref{sec-dense-case}, where we also prove that simple-equivalent trees with dense orbits are equal. We complete the proof of Theorem \ref{intro-White} in Section \ref{sec-White}. We also generalize the notion of candidates (Section \ref{sec-candidates}), and give more precise statements in the case of trees that have two distinct orbits of edges with trivial stabilizers in their simplicial parts (Section \ref{sec-separable-candidates}). In Section \ref{sec-rigidity}, we complete the proof of Theorem \ref{intro-equivalences}, by proving that simple-equivalent trees are special-pull-equivalent.

\section*{Acknowledgments}

It is a pleasure to thank my advisor Vincent Guirardel for his helpful advice, and his patience in reading through first drafts of the present paper and suggesting many improvements and simplifications.

\section{Outer space and its closure} \label{sec-CVn}

We start by fixing a few notations and recalling standard facts about outer space and $F_N$-actions on $\mathbb{R}$-trees in its closure.

\subsection{Outer space and its closure}\label{sec-cvn}

\emph{Outer space} $CV_N$ was defined by Culler and Vogtmann in \cite{CV86} to be the space of simplicial, free, minimal, isometric actions of $F_N$ on simplicial metric trees, up to equivariant homothety (an action of $F_N$ on a tree is said to be \emph{minimal} if there is no proper invariant subtree). We denote by $cv_N$ the \emph{unprojectivized outer space}, in which trees are considered up to isometry, instead of homothety. The reader is referred to \cite{Vog02} for an excellent survey and reference article about outer space. 

An \emph{$\mathbb{R}$-tree} is a metric space $(T,d_T)$ in which any two points $x$ and $y$ are joined by a unique arc, which is isometric to a segment of length $d_T(x,y)$ (the reader is referred to \cite{CM87} for an introduction to $\mathbb{R}$-trees). Let $T$ be an \emph{$F_N$-tree}, i.e. an $\mathbb{R}$-tree equipped with an isometric action of $F_N$. For $g\in F_N$, the \emph{translation length} of $g$ in $T$ is defined to be

\begin{displaymath}
||g||_T:=\inf_{x\in T}d_T(x,gx).
\end{displaymath}

\noindent Culler and Morgan have shown in \cite[Theorem 3.7]{CM87} that the map

\begin{displaymath}
\begin{array}{cccc}
i:&cv_N&\to &\mathbb{R}^{F_N}\\
&T&\mapsto & (||g||_T)_{g\in F_N}
\end{array}
\end{displaymath}

\noindent is injective, and actually a homeomorphism onto its image. More precisely, the following holds.

\begin{theo} \label{rigidity} (Culler-Morgan \cite[Theorem 3.7]{CM87})
Let $T,T'$ be two minimal $F_N$-trees. If $||g||_T=||g||_{T'}$ for all $g\in F_N$, then there is a unique $F_N$-equivariant isometry from $T$ to $T'$.
\end{theo}

Taking the quotient by equivariant homotheties, we get an embedding of $CV_N$ into the projective space $\mathbb{PR}^{F_N}$, whose image has compact closure $\overline{CV_N}$ \cite[Theorem 4.5]{CM87}. Hence $\overline{CV_N}$ is a compactification of $CV_N$. Bestvina and Feighn \cite{BF94}, extending results by Cohen and Lustig \cite{CL95}, have identified the compactification $\overline{CV_N}$ as the space of homothety classes of minimal, \emph{very small} $F_N$-trees, i.e. trees with trivial or maximally cyclic arc stabilizers and trivial tripod stabilizers. We also denote by $\overline{cv_N}$ the lift of $\overline{CV_N}$ to $\mathbb{R}^{F_N}$. We call the topology induced by this embedding on each of the spaces $CV_N, \overline{CV_N}, cv_N$ and $\overline{cv_N}$ the \emph{axes topology}, it is equivalent to the \emph{weak topology} on $CV_N$ introduced by Culler and Vogtmann in \cite{CV86}.

\subsection{A metric on outer space} \label{sec-metric}

There is a natural asymmetric metric on outer space, whose systematic study was initiated by Francaviglia and Martino in \cite{FM11} : given $T,T'\in cv_N$, the distance $d(T,T')$ is defined as the logarithm of the infimal Lipschitz constant of an $F_N$-equivariant map from $T$ to $T'$ (see also \cite[Section 2.4]{AK12}). An easy Arzelà-Ascoli argument shows that this infimal Lipschitz constant is actually achieved \cite[Lemma 3.4]{FM11}. This defines a topology on outer space, which is equivalent to the usual one \cite[Theorems 4.11 and 4.18]{FM11}. An element $g\in F_N$ is a \emph{candidate} in $T$ if it is represented in the quotient graph $X:=T/F_N$ by a loop which is either
\begin{itemize}
\item an embedded circle in $X$, or

\item an embedded bouquet of two circles in $X$, i.e. $\gamma=\gamma_1\gamma_2$, where $\gamma_1$ and $\gamma_2$ are embedded circles in $X$ which meet in a single point, or

\item a barbell graph, i.e. $\gamma=\gamma_1\eta\gamma_2\overline{\eta}$, where $\gamma_1$ and $\gamma_2$ are embedded circles in $X$ that do not meet, and $\eta$ is an embedded path in $X$ that meets $\gamma_1$ and $\gamma_2$ only at their origin (and $\overline{\eta}$ denotes the path $\eta$ crossed in the opposite direction). We call $\eta$ the \emph{central path} of $\gamma$.
\end{itemize}

\noindent The following result, due to White, gives an alternative description of the metric on outer space. A proof can be found in \cite[Proposition 3.15]{FM11}, it was simplified by Algom-Kfir in \cite[Proposition 2.3]{AK11}.

\begin{theo} \label{White} (White, see \cite[Proposition 3.15]{FM11} or \cite[Proposition 2.3]{AK11}) 
For all $F_N$-trees $T$,$T'\in CV_N$, we have
\begin{displaymath}
d(T,T')=\log \sup_{g\in F_N\smallsetminus\{e\}}\frac{||g||_{T'}}{||g||_T}.
\end{displaymath}
\noindent Furthermore, the supremum is achieved for an element $g\in F_N$ which is a candidate in $X:=T/F_N$.
\end{theo} 

Notice in particular that candidates in $X$ are \emph{primitive} elements of $F_N$, i.e. they belong to some free basis of $F_N$ (see Lemma \ref{detect-separability}, for instance). White's theorem has been extended by Algom-Kfir to the case of two trees $T,T'\in\overline{cv_N}$ when $T$ is assumed to be simplicial (in \cite[Proposition 4.5]{AK12}, Algom-Kfir states her result when $T'$ is a tree in the metric completion of outer space, but it actually holds true with the same proof for all trees $T'\in\overline{cv_N}$). We denote by  $\text{Lip}(T,T')$ the infimal Lipschitz constant of an $F_N$-equivariant map from $T$ to $T'$.

\begin{theo}\label{Algom-Kfir} (Algom-Kfir \cite[Proposition 4.5]{AK12})
Let $T,T'\in\overline{cv_N}$. If $T$ is simplicial, then 

\begin{displaymath}
\text{Lip}(T,T')=\sup_{g\in F_N\smallsetminus\{e\}}\frac{||g||_{T'}}{||g||_T}.
\end{displaymath}
Furthermore, the supremum is achieved for an element $g\in F_N$ which is a candidate in $X:=T/F_N$.
\end{theo}

\subsection{Decomposing actions in $\overline{CV_N}$} 
We now recall a result due to Levitt \cite{Lev94} which allows to decompose any $F_N$-tree into simpler actions. The reader is referred to \cite{Ser77} for an introduction to graphs of groups and related terminology. An \emph{$F_N$-graph of actions} consists of

\begin{itemize}
\item a marked metric graph of groups, whose edges all have positive length, with fundamental group $F_N$, with vertex groups $G_v$, edge groups $G_e$, and for every oriented edge $e$ with terminal vertex $t(e)$, an injective morphism $i_e:G_e\to G_{t(e)}$, and

\item an isometric action of every vertex group $G_v$ on an $\mathbb{R}$-tree $T_v$ (possibly reduced to a point), and

\item a point $p_e\in T_{t(e)}$ fixed by $i_e(G_e)\subseteq G_{t(e)}$ for every oriented edge $e$.
\end{itemize}

Associated to any $F_N$-graph of actions $\mathcal{G}$ is an $F_N$-tree $T(\mathcal{G})$. Informally, the tree $T(\mathcal{G})$ is obtained from the Bass-Serre tree of the underlying graph of groups by equivariantly attaching the vertex trees $T_v$ at the vertices $v$, an incoming edge being attached to $T_v$ at the prescribed attaching point. The reader is referred to \cite[Proposition 3.1]{Gui98} for a precise description of the tree $T(\mathcal{G})$. We say that an $F_N$-tree $T$ \emph{splits as a graph of actions} if there exists a graph of actions $\mathcal{G}$ such that $T=T({\mathcal{G}})$. An $F_N$-tree $T$ has \emph{dense orbits} if the $F_N$-orbit of one (and hence every) point of $T$ is dense in $T$. 

\begin{theo} \label{graph-of-actions} (Levitt \cite[Theorem 5]{Lev94})
Every $T\in\overline{cv_N}$ splits uniquely as a graph of actions with vertex trees having dense orbits (possibly reduced to a point).
\end{theo}

We denote by $T^{simpl}$ the corresponding simplicial tree, obtained by collapsing all the vertex trees to points. An \emph{edge} in $T$ is a segment in the simplicial part of $T$ that projects to an edge in $T^{simpl}$. 

\subsection{Trees with dense orbits}

In this head, we collect a few facts about $F_N$-trees with dense orbits.

\begin{lemma} \label{dense-arcs} (Bestvina-Feighn \cite[Remark 1.9]{BF94}, Gaboriau-Levitt \cite[Proposition I.10]{GL95}, Sela \cite[Proposition 1.4]{Sel96}, Levitt-Lustig \cite[Lemma 4.2]{LL03})
Every very small $F_N$-tree with dense orbits has trivial arc stabilizers.
\end{lemma}

Given a tree $T\in\overline{cv_N}$, a subset $X\subseteq T$, and $M\in\mathbb{R}$, we denote by $\mathcal{N}_M(X)$ the $M$-neighborhood of $X$ in $T$. The \emph{bridge} between two closed subtrees $X,Y\subseteq T$ which do not intersect is the unique segment in $T$ which meets $X\cup Y$ only at its endpoints. Given a closed subtree $X\subseteq T$ and $x\in T$, we denote by $\pi_X(x)$ the closest point projection of $x$ to the subtree $X$.

\begin{lemma} \label{intersection-nbd}
Let $T$ be an $\mathbb{R}$-tree, let $X,Y\subseteq T$ be closed subtrees, and let $M\in\mathbb{R}$. If $X\cap Y\neq\emptyset$, then $\mathcal{N}_{M}(X)\cap\mathcal{N}_{M}(Y)=\mathcal{N}_M(X\cap Y)$. If $X\cap Y=\emptyset$, then for any point $y$ in the bridge between $X$ and $Y$ in $T$, we have $\mathcal{N}_{M}(X)\cap\mathcal{N}_{M}(Y)\subseteq\mathcal{N}_M(\{y\})$.
\end{lemma}

\begin{proof}
Let $x\in\mathcal{N}_{M}(X)\cap\mathcal{N}_{M}(Y)$. Let $J$ denote the subtree $X\cap Y$, or the bridge between $X$ and $Y$ in case $X\cap Y=\emptyset$. Assume that $\pi_X(x)\notin J$. Then one checks that $x$ and $\pi_X(x)$ must belong to the same component of $T\smallsetminus J$, and hence that $\pi_Y(x)\in J$. So either $\pi_X(x)\in J$, or $\pi_Y(x)\in J$, and the claim follows. 
\end{proof}

In the following statement, notice that whenever $T\in\overline{cv_N}$ is a tree with dense orbits, then the $F_N$-action on $T$ uniquely extends to an isometric action on its metric completion $\overline{T}$, and $\overline{T}$ again has dense orbits. Recall that an $F_N$-tree is \emph{minimal} if it contains no proper $F_N$-invariant subtree. An $F_N$-tree $T$ which is not minimal has a unique minimal proper $F_N$-invariant subtree $T^{\min}$, which is also the union of all axes of hyperbolic elements in $T$. In particular, for all $g\in F_N$, we have $||g||_T=||g||_{T^{\min}}$. When $T$ has dense orbits, we have $T^{\min}\subseteq T\subseteq \overline{T^{\min}}$, since the orbit of any point $x\in T^{\min}$ is dense in $T$ and contained in $T^{\min}$. For all $F_N$-trees $T$ and all $g\in F_N$, either $||g||_T=0$ (we say that $g$ is \emph{elliptic} in $T$), and in this case $g$ has a fixed point in $T$, or $||g||_T>0$ (then $g$ is said to be \emph{hyperbolic} in $T$), and in this case $g$ has an axis in $T$, i.e. there exists a subspace of $T$ homeomorphic to the real line on which $g$ acts by translation, with translation length $||g||_T$. In both cases, we define the \emph{characteristic set} of $g$ to be $C_T(g):=\{x\in T|d(x,gx)=||g||_T\}$ (see \cite[1.3]{CM87} for a description of the action of elements of $F_N$ on $F_N$-trees). 

\begin{prop} \label{dense-rigidity}
Let $T,T'$ be very small $F_N$-trees with dense orbits. Then there exists at most one Lipschitz $F_N$-equivariant map from $T$ to $\overline{T'}$.
\end{prop}

\begin{proof}
It is enough to show that for all $M\in\mathbb{R}$, there exists at most one $M$-Lipschitz $F_N$-equivariant map from $T$ to $\overline{T'}$, so we fix $M\in\mathbb{R}$. Let $x\in T$, and let $\epsilon>0$. We claim that we can find a subset $X_{\epsilon}\subset \overline{T'}$ whose diameter is bounded above by $3M\epsilon$, with the property that $f(x)\in X_{\epsilon}$ for all $M$-Lipschitz $F_N$-equivariant maps $f:T\to\overline{T'}$. Indeed, let $f:T\to\overline{T'}$ be $M$-Lipschitz and $F_N$-equivariant. As $T$ has dense orbits, there exists $g\in F_N$ such that $d_T(x,gx)\le\epsilon$, and the set of all such elements of $F_N$ is not contained in any cyclic subgroup of $F_N$. As $f$ is $M$-Lipschitz and $F_N$-equivariant, any such $g\in F_N$ satisfies $d_{\overline{T'}}(f(x),gf(x))\le M\epsilon$. We also have $d_{\overline{T'}}(f(x),gf(x))=||g||_{\overline{T'}}+2d_{\overline{T'}}(f(x),C_{\overline{T'}}(g))$ (see \cite[1.3]{CM87}), so $f(x)\in\mathcal{N}_{M\epsilon}(C_{\overline{T'}}(g))$ and $||g||_{T'}\le M\epsilon$. Let $g,g'\in F_N$ be two elements satisfying $d_T(x,gx)\le\epsilon$ and $d_T(x,g'x)\le\epsilon$, which do not generate a cyclic subgroup of $F_N$ (in particular, the commutator $[g,g']$ is nontrivial). We have $f(x)\in\mathcal{N}_{M\epsilon}(C_{\overline{T'}}(g))\cap\mathcal{N}_{M\epsilon}(C_{\overline{T'}}(g'))$. As $\overline{T'}$ is very small and has dense orbits, it follows from Lemma \ref{dense-arcs} that the commutator $[g,g']$ does not fix any arc in $\overline{T'}$, so $C_{\overline{T'}}(g)\cap C_{\overline{T'}}(g')$ is a (possibly empty) segment of length at most $2M\epsilon$ (see \cite[1.10]{CM87}). By Lemma \ref{intersection-nbd}, the set $\mathcal{N}_{M\epsilon}(C_{\overline{T'}}(g))\cap\mathcal{N}_{M\epsilon}(C_{\overline{T'}}(g'))=\mathcal{N}_{M\epsilon}(C_{\overline{T'}}(g)\cap C_{\overline{T'}}(g'))$ has diameter at most $3M\epsilon$. We set $X_{\epsilon}:=\mathcal{N}_{M\epsilon}(C_{\overline{T'}}(g))\cap\mathcal{N}_{M\epsilon}(C_{\overline{T'}}(g'))$.

If $f,f':T\to \overline{T'}$ are two $M$-Lipschitz, $F_N$-equivariant maps, then for all $\epsilon>0$, we have $f(x),f'(x)\in X_{\epsilon}$, hence $d_{\overline{T'}}(f(x),f'(x))\le 3M\epsilon$. This implies that $f(x)=f'(x)$. As this is true for all $x\in T$, we get that $f=f'$.
\end{proof}

\subsection{Morphisms between $F_N$-trees}\label{sec-morphism}

A \emph{morphism} between two $\mathbb{R}$-trees $T$ and $T'$ is a map $f:T\to T'$ such that every segment $J\subset T$ can be subdivided into finitely many subsegments, in restriction to which $f$ is an isometry (in particular, any morphism between two $\mathbb{R}$-trees is $1$-Lipschitz). We say that two arcs in $T$ are \emph{folded} by $f$ if they have initial subsegments whose $f$-images are equal. 

Let $T$ be an $F_N$-tree containing an edge $e$ with trivial stabilizer. A \emph{U-turn} over $e$ is a pair of distinct adjacent edges in $T$ of the form $(e,ge)$, where $g$ belongs to the stabilizer of one of the extremities $v$ of $e$ and is not a proper power, such that either the stabilizer of the image of $v$ in $T^{simpl}$ has rank at least $2$, or $v$ does not project to a valence one vertex of the quotient graph of actions. In the following lemmas, we collect a few facts about $F_N$-equivariant morphisms between $F_N$-trees.

\begin{lemma}\label{morphism-1}
Let $T$ and $T'$ be two very small $F_N$-trees. If $T'$ has trivial arc stabilizers, then an $F_N$-equivariant morphism from $T$ to $T'$ cannot fold any U-turn in $T$.
\end{lemma}

\begin{proof}
Assume that an $F_N$-equivariant morphism $f:T\to T'$ folds a U-turn $(e,e')$ in $T$. Then there exists an initial segment $I$ of $e$ such that $f(I)$ has nontrivial stabilizer. The hypothesis made on $T'$ implies that $f(I)$ is a point, contradicting the definition of a morphism.
\end{proof}

\begin{lemma}\label{morphism-2}
Let $T$ and $T'$ be two very small $F_N$-trees. An $F_N$-equivariant morphism from $T$ to $T'$ cannot identify nontrivial initial segments of edges in $T$ having distinct nontrivial stabilizers. It cannot either identify a nontrivial initial segment of an edge in $T$ with nontrivial stabilizer with one of its translates.
\end{lemma}

\begin{proof}
Otherwise, as edge stabilizers in $T$ are maximally cyclic, the stabilizer of the image of these segments would have rank at least $2$. As $T'$ is very small, this image would be a point, contradicting the definition of a morphism.
\end{proof}

\begin{lemma}\label{morphism-3}
Let $T$ and $T'$ be two very small $F_N$-trees. An $F_N$-equivariant morphism from $T$ to $T'$ cannot identify a nontrivial initial segment of an edge in $T$ with nontrivial stabilizer with an arc lying in a vertex tree of $T$ with dense orbits.
\end{lemma}

\begin{proof}
Otherwise, the image of this segment would be an arc with nontrivial stabilizer lying in a vertex tree of $T'$ with dense orbits. By Lemma \ref{dense-arcs}, it would thus be reduced to a point, contradicting the definition of a morphism.
\end{proof}

\subsection{The quotient volume of $F_N$-trees}

Let $T\in\overline{cv_N}$. The \emph{volume} of a finite subtree $K\subset T$ (i.e. the convex hull of a finite number of points, which is a finite union of segments) is the sum of the lengths of the segments in $K$. The \emph{quotient volume} of $T$ is defined to be the infimal volume of a finite subtree of $T$ whose $F_N$-translates cover $T$. We collect a few facts which were observed by Algom-Kfir in \cite[Section 3.3]{AK12}.

\begin{prop} \label{qvol} (Algom-Kfir \cite{AK12})
\begin{itemize}
\item For all $F_N$-trees $T$, all minimal $F_N$-trees $T'$ and all $L\in\mathbb{R}$, if there exists an $L$-Lipschitz $F_N$-equivariant map from $T$ to $T'$, then $qvol(T')\le L qvol(T)$.
\item Let $T\in\overline{cv_N}$, and let $(T_n)_{n\in\mathbb{N}}$ be a sequence of trees in $cv_N$ converging to $T$. Then $qvol(T)\ge\limsup_{n\to +\infty}qvol(T_n)$. If in addition $T^{simpl}$ contains an orbit of edges with nontrivial stabilizers, then $qvol(T)>\limsup_{n\to +\infty} qvol(T_n)$. 
\end{itemize}
\end{prop}

\subsection{Finding primitive elements in graphs of groups decompositions of $F_N$}

We now state a lemma that will be useful for detecting primitive elements in graphs of groups decompositions of $F_N$.

\begin{lemma} \label{detect-separability}
Let $X$ be a minimal graph of groups decomposition of $F_N$ containing an edge $e$ with trivial stabilizer. Let $g\in F_N$. If $g$ is elliptic in the Bass-Serre tree of $X$, or if any fundamental domain of its axis crosses the orbit of a lift of $e$ at most once, then $g$ is simple. For all vertex groups $G$ in $X$, there exists a proper free factor of $F_N$ that contains $G$. 
\end{lemma}

\begin{proof}
Let $e$ be an edge in $X$ with trivial stabilizer. Collapsing each component of the complement of $e$ in $X$ to a point yields a free splitting of $F_N$ either of the form $F_N=F_{N-1}\ast$ if $e$ is nonseparating (in which case we denote by $t$ a stable letter), or of the form $F_N=F_k\ast F_{N-k}$ if $e$ is separating. All vertex groups in $X$ are contained in a free factor provided by the splitting. If the axis of a hyperbolic element $g\in F_N$ does not cross any lift of $e$, then $g$ belongs to a proper free factor provided by the splitting. If any fundamental domain of the axis of $g$ in the Bass-Serre tree $T$ of $X$ crosses the orbit of a lift of $e$ exactly once, then $e$ is nonseparating and $g$ is conjugated to an element of the form $tg'$ with $g'\in F_{N-1}$, so $g$ is primitive (because if $\{x_1,\dots,x_{N-1}\}$ is a free basis of $F_{N-1}$, then $\{x_1,\dots,x_{N-1},tg'\}$ is a free basis of $F_N$). 
\end{proof}

\section{Some equivalence relations on $\overline{cv_N}$}\label{sec-beginning}

\subsection{Primitive-equivalence versus simple-equivalence} \label{sec-equivalences}

Recall that an element $g\in F_N$ is \emph{primitive} if it belongs to some free basis of $F_N$ (we denote by $\mathcal{P}_N$ the set of primitive elements of $F_N$). It is \emph{simple} if it belongs to some proper free factor of $F_N$. Two trees $T,T'\in\overline{cv_N}$ are \emph{primitive-equivalent} if for all $g\in \mathcal{P}_N$, we have $||g||_T=||g||_{T'}$. They are \emph{simple-equivalent} if for all simple elements $g\in F_N$, we have $||g||_T=||g||_{T'}$. 

\begin{prop} \label{primitive-separable}
Two elements $T,T'\in\overline{cv_N}$ are primitive-equivalent if and only if they are simple-equivalent.
\end{prop}

\begin{proof}
Simple-equivalent trees are obviously primitive-equivalent, as primitive elements are simple. Assume that $||g||_{T'}=||g||_T$ for all $g\in \mathcal{P}_N$. Let $w\in F_N$ be simple, i.e there exists a free basis $\{a_1,\dots, a_N\}$ of $F_N$ such that $w$ belongs to the free factor of $F_N$ generated by $a_1,\dots,a_{N-1}$. Then for all $k\in\mathbb{N}$, we have $a_Nw^k\in\mathcal{P}_N$, because $\{a_1,\dots,a_{N-1},a_Nw^k\}$ is again a free basis of $F_N$. So 

\begin{displaymath}
\begin{array}{rl}
||w||_{T'}&=\lim_{k\to +\infty} \frac{||a_Nw^k||_{T'}}{k}\\
&=\lim_{k\to +\infty} \frac{||a_Nw^k||_{T}}{k}\\
&=||w||_T.
\end{array}
\end{displaymath}

\noindent This shows that $T$ and $T'$ are simple-equivalent.
\end{proof}

\subsection{Special-pull-equivalent trees}\label{sec-slide-equiv}

The following notion is illustrated in Figure \ref{fig-turns}. The \emph{corank} of a free factor $F$ of $F_N$ is the rank of any complementary free factor of $F_N$, i.e. it is equal to $N$ minus the rank of $F$.

\begin{de} \label{de-pull}
Let $T,\widehat{T}\in\overline{cv_N}$. The tree $T$ is a \emph{pull} of $\widehat{T}$ if there exist 

\begin{itemize}
\item an edge $e$ in $\widehat{T}$ with trivial stabilizer, and extremal vertices $v_1$ and $v_2$, and 
\item for each $i\in\{1,2\}$, a (possibly degenerate) subsegment $J_i\subseteq e$ that contains $v_i$, such that $J_1\cap J_2$ contains at most one point, and
\item for each $i\in\{1,2\}$ such that $J_i$ is nondegenerate, an element $g_i$ in the stabilizer of $v_i$, which is not a proper power,
\end{itemize}

\noindent so that $T$ is obtained as a quotient of $\widehat{T}$ by equivariantly identifying $J_i$ with $g_iJ_i$ for each $i\in\{1,2\}$. 
\end{de}

We will also say that $T$ is obtained from $\widehat{T}$ by pulling the edge $e$, with pulling elements $g_1$ and $g_2$. 

\begin{figure}
\begin{center}
\input{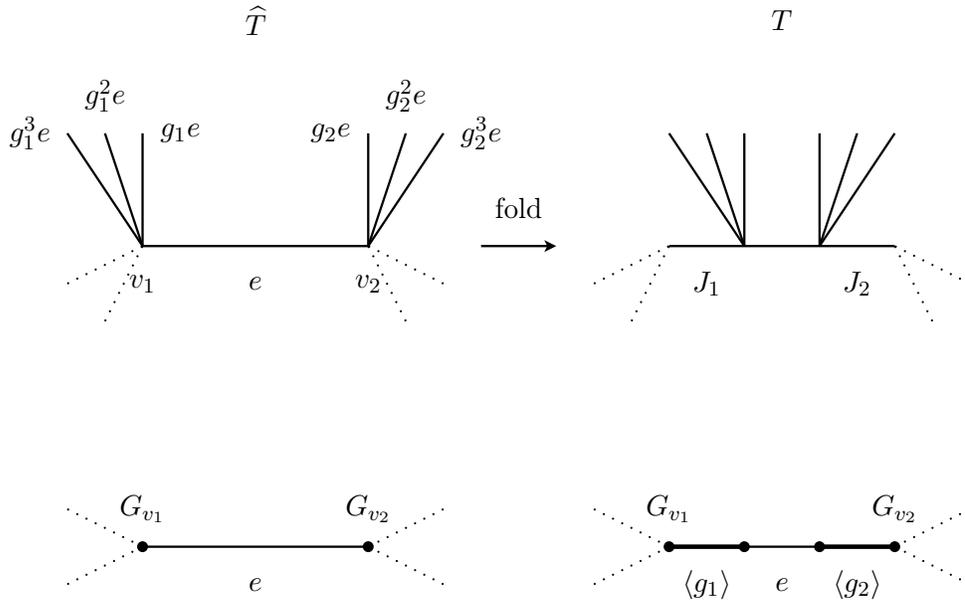}
\caption{The tree $T$ is obtained by pulling $\widehat{T}$.}
\label{fig-turns}
\end{center}
\end{figure}

\begin{de} \label{de-NS-pull}
Let $T,\widehat{T}\in\overline{cv_N}$. The tree $T$ is a \emph{special pull} of $\widehat{T}$ if $T$ is a pull of $\widehat{T}$ and, with the above notations, the edge $e$ projects to a nonseparating edge in the underlying graph of the decomposition of $\widehat{T}$ as a graph of actions with dense orbits, and neither $g_1$ nor $g_2$ belongs to a corank $2$ free factor of $F_N$.
\end{de}

\begin{de}\label{de-NS}
Two trees $T,T'\in\overline{cv_N}$ are \emph{special-pull-equivalent} if either $T=T'$, or there exists a tree $\widehat{T}\in\overline{cv_N}$ such that $T$ and $T'$ are special pulls of $\widehat{T}$.
\end{de}

In the second case of Definition \ref{de-NS}, the tree $\widehat{T}$ has a single orbit of edges with trivial stabilizer, otherwise the pulling element would be contained in a corank $2$ free factor of $F_N$. That special-pull-equivalence is indeed an equivalence relation will be proved in Lemma \ref{NS-pull-equivalence}.

\paragraph*{Description of special-pull-equivalence in the case $N=2$.}
If $N=2$, then any pull is special. Let $T,\widehat{T}\in\overline{cv_2}$ be such that $T$ is a special pull of $\widehat{T}$. We assume that there does not exist any tree $\widetilde{T}\neq\widehat{T}\in \overline{cv_2}$ such that both $T$ and $\widehat{T}$ are special pulls of $\widetilde{T}$ with same pulling elements. As displayed on Figure \ref{fig-cv2-1}, the tree $\widehat{T}$ is then the Bass-Serre tree of a splitting of the form $F_2=\langle a\rangle\ast$, where $a$ is a primitive element of $F_2$, and $T$ is obtained by (partially or totally) pulling the edge of $\widehat{T}$, with $a$ as a pulling element.

\begin{figure}
\begin{center}
\input{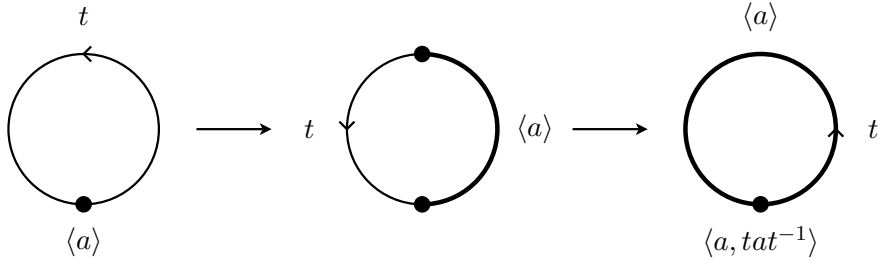}
\caption{Pulls in $\overline{cv_2}$.}
\label{fig-cv2-1}
\end{center}
\end{figure}

An exhaustive description of the boundary of $CV_2$ was given by Culler and Vogtmann in \cite{CV91}. With their terminology, the quotient $\overline{CV_2}/\sim$ is given by collapsing all spikes in $\overline{CV_2}$, and is thus homeomorphic to a disk with "fins" attached on top, see Figure \ref{fig-cv2}. 

\begin{figure}
\begin{center}
\def\JPicScale{.8}
\input{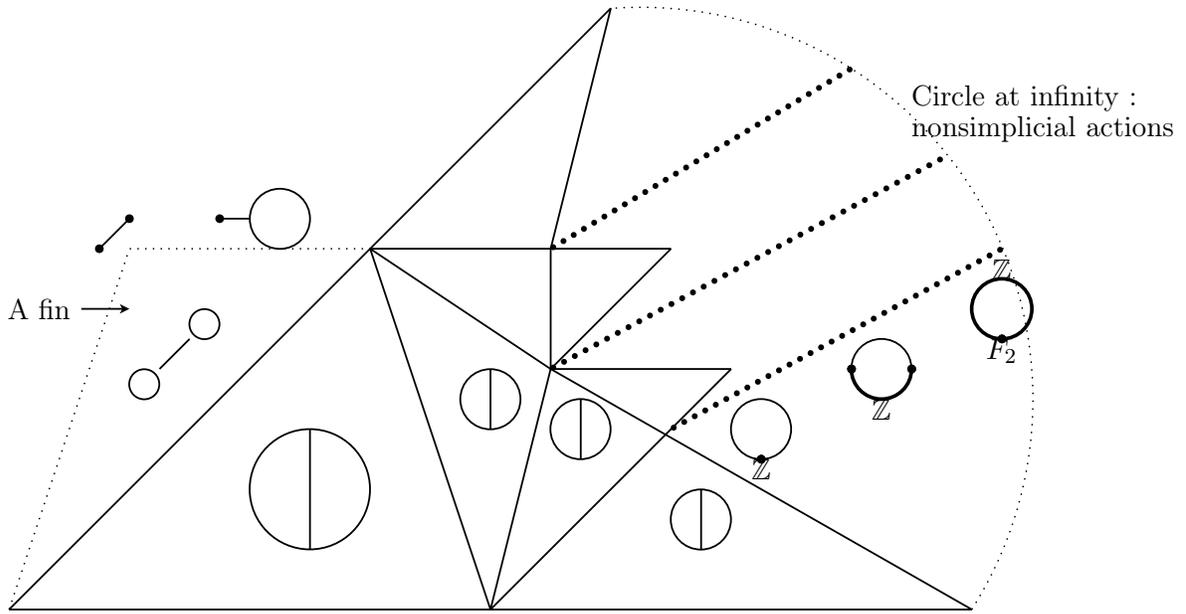}
\caption{The quotient space $\overline{CV_2}/\sim$ is obtained by collapsing all peaks in bold dotted lines to points, and hence is homeomorphic to a disk with "fins" attached on top.}
\label{fig-cv2}
\end{center}
\end{figure}

\paragraph*{Description of special-pull-equivalence in the case $N\ge 3$.}

We now assume that $N\ge 3$, and we give a description of special-pull-equivalence in terms of graphs of actions. The discussion below is illustrated in Figures \ref{fig-case-1} to \ref{fig-case-3}. Let $\widehat{T}\in\overline{cv_N}$ be a tree with exactly one orbit of edges with trivial stabilizer $e$, and let $T\neq\widehat{T}$ be a special pull of $\widehat{T}$. We assume that there does not exist any tree $\widetilde{T}\neq\widehat{T}\in cv_N$ such that both $T$ and $\widehat{T}$ are special pulls of $\widetilde{T}$ with same pulling elements. The tree $\widehat{T}$ splits as a graph of actions having 

\begin{itemize}
\item a single vertex, whose corresponding vertex tree is a (non necessarily minimal) $A$-tree $T_0$, where $A$ is a corank one free factor of $F_N$, and 
\item a single loop-edge with trivial stabilizer. 
\end{itemize}

As $N\ge 3$, the group $A$ is not cyclic, so it has at most one fixed point in $\widehat{T}$, and the $A$-minimal subtree $T_0^{min}$ of $T_0$ is well-defined. Minimality of $T$ implies that $T_0$ is obtained from $T_0^{min}$ by possibly adding some completion points, and attaching at most two $A$-orbits of edges (the discussion below will show that we can actually attach at most one $A$-orbit of edges when passing from $T_0^{min}$ to $T_0$). The valence one extremities of these edges are attaching points for $e$ in $\widehat{T}$. One of the following situations occurs.
\\
\\
\textit{Case 1} (see Figure \ref{fig-case-1}): The tree $T_0$ is minimal (or more generally, we have $T_0$ is the closure of ${T_0^{min}}$, i.e. $T_0$ is obtained from $T_0^{min}$ by adding completion points, or in other words $T_0\smallsetminus T_0^{min}$ does not contain any simplicial edge).
\\
Then $T$ is obtained from $\widehat{T}$ by pulling $e$, either at one of its extremities or at both of its extremities. (Notice that we cannot perform any pull from a completion point). 
\\
\begin{figure}
\begin{center}
\input{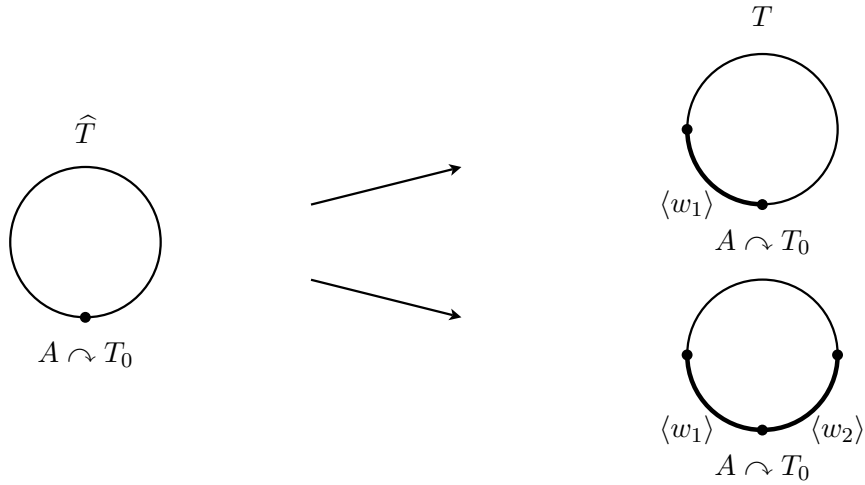}
\caption{The situation in Case 1.}
\label{fig-case-1}
\end{center}
\end{figure}

\noindent \textit{Case 2} (see Figure \ref{fig-case-2}): The tree $T_0$ is not minimal, and $T_0\smallsetminus T_0^{min}$ contains a simplicial edge $e'$ whose stabilizer $\langle w\rangle$ is not contained in any proper free factor of $A$.
\\
Then the valence one extremity of $e'$ in the decomposition of $T_0$ as a graph of actions has valence at least $3$ in the decomposition of $\widehat{T}$ as a graph of actions. Otherwise, the tree $\widehat{T}$ would be obtained from a tree $\widetilde{T}$ by pulling this edge, contradicting the assumption made on $\widehat{T}$. This implies in particular that $T_0$ is obtained from $T_0^{min}$ by attaching a single orbit of edges. When passing from $\widehat{T}$ to $T$, the edge $e$ is pulled at only one of its extremity, otherwise this would create a tripod stabilizer.
\\
\begin{figure}
\begin{center}
\input{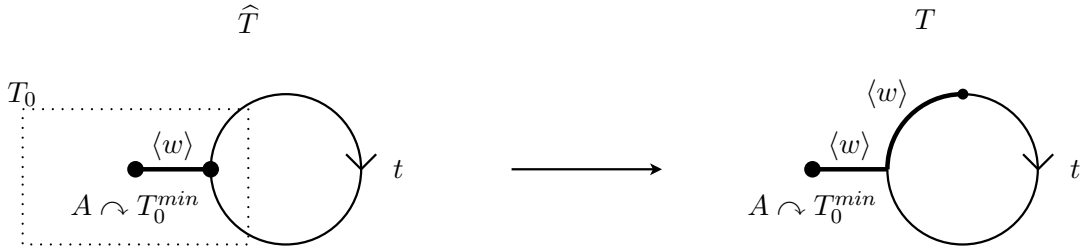}
\caption{The situation in Case 2.}
\label{fig-case-2}
\end{center}
\end{figure}

\noindent \textit{Case 3} (see Figure \ref{fig-case-3}): The tree $T_0$ is not minimal, and $T_0\smallsetminus T_0^{min}$ contains a simplicial edge whose stabilizer is nontrivial, and contained in some proper free factor of $A$.
\\
If there were two such edges, or if $e$ projected to a loop-edge in the decomposition of $\widehat{T}$ as a graph of action, then no special pulling operation could be performed on $\widehat{T}$, so we are in the situation displayed on Figure \ref{fig-case-3}.
\\
\begin{figure}
\begin{center}
\input{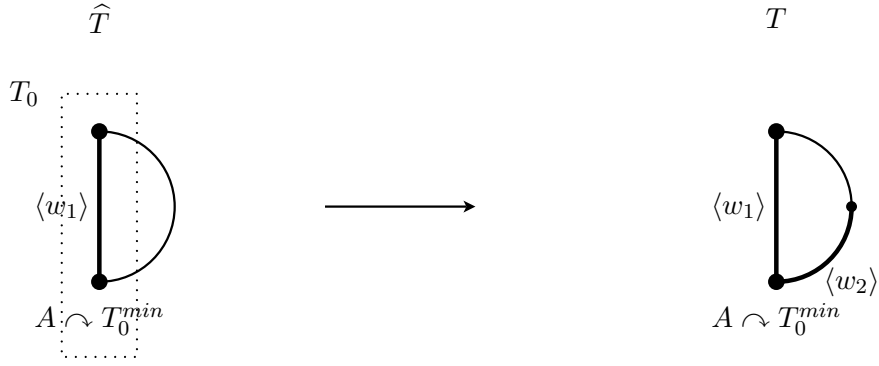}
\caption{The situation in Case 3.}
\label{fig-case-3}
\end{center}
\end{figure}

\noindent \textit{Case 4} : The tree $T_0$ is not minimal, and $T_0\smallsetminus T_0^{min}$ contains a simplicial edge with trivial stabilizer. 
\\
We will show that this case never happens. Indeed, the valence one extremity of this edge in the decomposition of $T_0$ as a graph of actions has valence at least $3$ in the decomposition of $\widehat{T}$ as a graph of actions, and it has trivial stabilizer. In this situation, no special pulling operation can be performed on $\widehat{T}$, a contradiction.
\\
\\
Notice that in all cases, the simplicial part of $T_0\smallsetminus T_0^{min}$ contains at most one orbit of edges.

\paragraph*{Special-pull-equivalence is an equivalence relation.}

\begin{lemma}\label{NS-pull-equivalence}
Special-pull-equivalence is an equivalence relation on $\overline{cv_N}$.
\end{lemma}

\begin{proof}
Let $T\in\overline{cv_N}$, and assume that there exists a tree $\widehat{T}\in\overline{cv_N}$ so that $T$ is a special pull of $\widehat{T}$. Then the corank one free factor $A$ (with the notations from the above paragraph) is uniquely determined as being the smallest free factor of $F_N$ containing all arc stabilizers in $T$, and the minimal $A$-tree $T_0^{min}$ is determined by the restriction to $A$ of the translation length function of $T$. It then follows from the description given in the previous paragraph of the relationship between $T_0^{min}$ and $\widehat{T}$ that $\widehat{T}$ is uniquely determined. Lemma \ref{NS-pull-equivalence} follows from this observation.
\end{proof}

We denote by $\sim$ the special-pull-equivalence relation on $\overline{cv_N}$. The \emph{standard element} of a nontrivial class of special-pull-equivalence is the unique element of the class in which the length of the edge with trivial stabilizer (if any) is maximal. Each equivalence class is star-shaped and contractible, and consists of a union of simplices of dimension at most $2$. 

\paragraph*{A few examples.}

The simplest examples of special-pull-equivalent trees arise by pulling the Bass-Serre tree of a splitting of $F_N$ of the form $F_N=A\ast$ (where $A$ is a corank one free factor of $F_N$), with any nonsimple elements of $F_{N-1}$ as pulling elements, see Figure \ref{fig-examples}.

A more elaborated class of examples arises by letting $T_{N-1}$ be any geometric $F_{N-1}$-tree dual to a foliation on a surface $S$ with a single boundary component, and forming a graph of actions whose underlying graph of groups represents the splitting $F_N=F_{N-1}\ast$, with attaching point $x$ given by the boundary curve of $S$, and pulling elements stabilizing $x$ in $T_{N-1}$ (see Figure \ref{fig-examples}, see also Section \ref{sec-approximations} for a more detailed account of this construction). 

There is a way of building new examples by an iterative process. Start from a minimal $F_{N-1}$-tree $Y$ that contains a point $x$ whose stabilizer is not contained in any proper free factor of $F_{N-1}$, form an HNN-extension $F_{N-1}\ast$, and the corresponding graph of actions with attaching point $x$, and fold the corresponding edge $e$ totally over a translate $ge$, where $g$ does not belong to any proper free factor of $F_{N-1}$. This gives a tree $T'$ having a point stabilizer which is not contained in any proper free factor of $F_N$. Hence it can serve as the vertex tree of an $F_{N+1}$-tree whose NS-pull-equivalence class is nontrivial. Iterating this process creates a class of NS-pull-equivalent trees. 

\begin{figure}
\begin{center}
\input{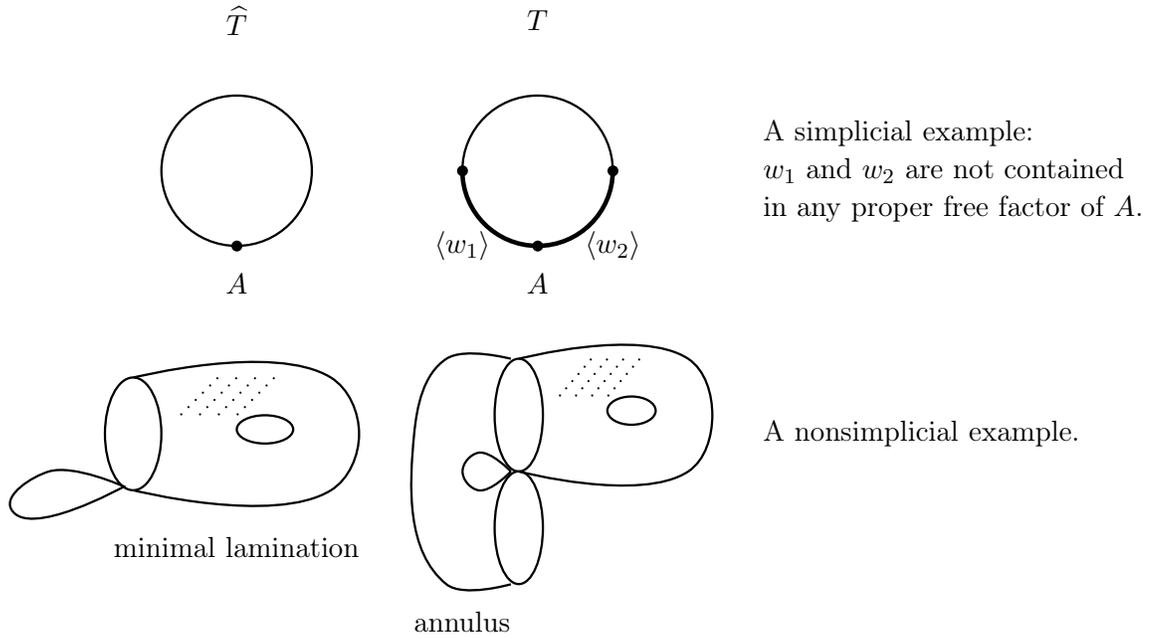}
\caption{Examples of NS-pull-equivalent trees.}
\label{fig-examples}
\end{center}
\end{figure}

\subsection{Special-pull-equivalent trees are simple-equivalent.}\label{sec-Whitehead}

This section is devoted to the proof of the following implication.

\begin{prop} \label{slide-separable}
Any two special-pull-equivalent trees $T,T'\in\overline{cv_N}$ are simple-equivalent.
\end{prop} 

In order to prove Proposition \ref{slide-separable}, we start by checking that a certain class of elements of $F_N$ are nonsimple, using methods due to Whitehead \cite{Whi36} and further developed by Stallings \cite{Sta99}. The \emph{Whitehead graph} of an element $w\in F_N$ with respect to a free basis $B$ of $F_N$, denoted by $Wh_B(w)$, is the graph whose vertices are the elements of $B^{\pm 1}$, two vertices $a$ and $b$ being joined by an edge if $ab^{-1}$ occurs as a subword of the cyclic word that represents $w$ in the basis $B$. A \emph{cutpoint} in a connected graph $X$ is a point $p\in X$ such that $X\smallsetminus\{p\}$ is disconnected.

\begin{prop} \label{Whitehead} (Whitehead \cite{Whi36}, Stallings \cite{Sta99})
An element $w\in F_N$ is simple if and only if its Whitehead graph with respect to any free basis of $F_N$ is either disconnected or contains a cutpoint.
\end{prop}

\begin{prop} \label{class-separable}
Let $A$ be a corank $1$ free factor of $F_N$. For all $w\in A$, the following assertions are equivalent.
\begin{itemize}
\item The element $w$ is contained in some proper free factor of $A$.
\item There exist a basis $B=\{x_1,\dots,x_N\}$ of $F_N$, such that $A$ is the subgroup generated by $x_1,\dots,x_{N-1}$, and an element $v\in F_N$, such that $vx_Nwx_N^{-1}$ is primitive, and the product $vx_Nwx_N^{-1}$ is cyclically reduced when $v$ and $w$ are written as reduced words in the basis $B$.
\end{itemize}
\end{prop}

\begin{proof}
First assume that $w$ is contained in a proper free factor of $A$, and let $\{x_1,\dots,x_k\}$ denote a basis of this free factor. Let $\{x_{k+1},\dots,x_{N-1}\}$ be a basis of a complementary free factor of $A$, let $x_N\in F_N$ be such that $F_N=A\ast\langle x_N\rangle$, and let $v:=x_{k+1}$. Then $\{vx_Nwx_N^{-1},x_1,\dots,x_k,x_{k+2},\dots,x_{N-1},x_N\}$ is a free basis of $F_N$, so $vx_Nwx_N^{-1}$ is primitive. In addition, the product $vx_Nwx_N^{-1}$ is cyclically reduced when written as a reduced word in the basis $\{x_1,\dots,x_N\}$ of $F_N$.

Assume now that $w$ is not contained in any proper free factor of $A$. Assume by contradiction that there exists a basis $B=\{x_1,\dots,x_N\}$ of $F_N$ such that $A=\langle x_1,\dots, x_{N-1}\rangle$, and an element $v\in F_N$ such that the product $vx_Nwx_N^{-1}$ is cyclically reduced when $v$ and $w$ are written as reduced words in the basis $B$, and $vx_Nwx_N^{-1}$ is primitive. By Proposition \ref{Whitehead}, we can choose $x_1,\dots, x_{N-1}$ such that the Whitehead graph of $w$ is connected without cutpoint in the basis $\{x_1,\dots,x_{N-1}\}$ of $A$. We denote by $a$ the first letter of $w$ in $B$, by $b$ its last letter, and by $c_1$ the last letter of $v$. The Whitehead graph of $W:=vx_Nwx_N^{-1}$ in $B$ contains $Wh_B(w)$, in which an edge joining $b$ to $a^{-1}$ is replaced by an edge joining $b$ to $x_N$ and an edge joining $a^{-1}$ to $x_N$, and $Wh_B(W)$ also contains an edge joining $x_N^{-1}$ to $c_1$, see Figure \ref{fig-Whitehead}. In particular, it is connected, and its only possible cutpoint is $c_1$, provided there is no edge joining $x_N^{-1}$ to a vertex different from $c_1$. This implies that $c_1^{-1}$ (resp. $c_1$) is the first (resp. last) letter of the reduced word that represents $v$ in the basis $B$, i.e. there exists a subword $\widetilde{v}$ of $v$ so that $vx_Nwx_N^{-1}=(c_1^{-1}\widetilde{v}c_1)x_Nwx_N^{-1}$.

\begin{figure}
\begin{center}
\input{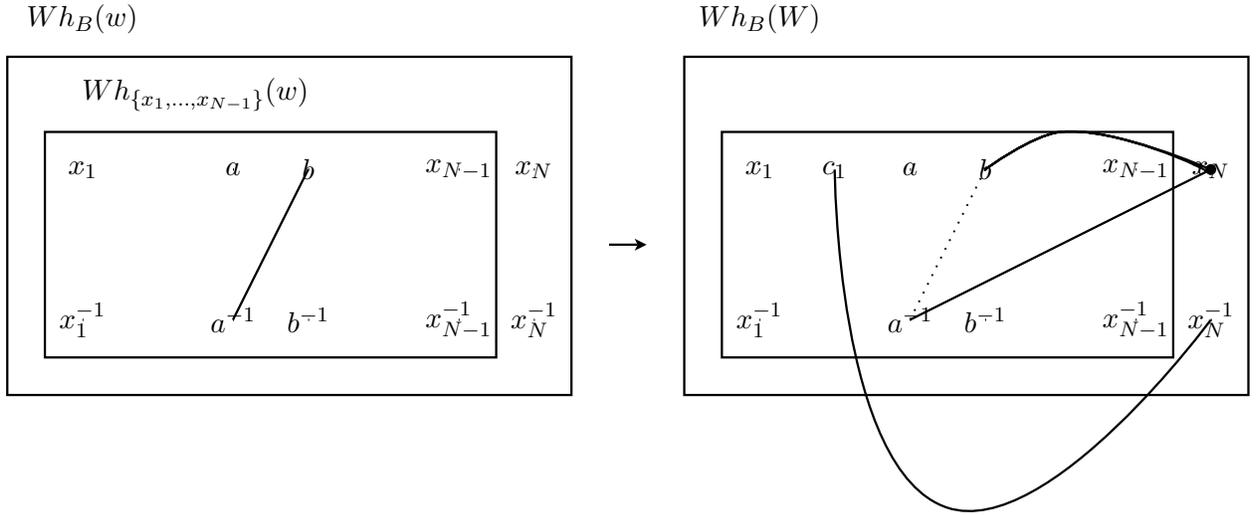}
\caption{The Whitehead graphs $Wh_B(w)$ and $Wh_B(W)$ in the proof of Proposition \ref{class-separable}.}
\label{fig-Whitehead}
\end{center}
\end{figure}

First observe that $c_1\neq x_N$, otherwise all occurrences of $x_N^{-1}$ in the cyclic word that represents $W$ in the basis $B$ should be followed by another occurrence of $x_N^{-1}$, and $W$ would be a power of $x_N$, a contradiction. As $c_1$ is a cutpoint of $Wh_B(w)$, all occurrences of the letter $x_N$ in the reduced word representing $v$ in $B$ are preceded by an occurrence of $c_1$, and all occurrences of $x_N^{-1}$ are followed by an occurrence of $c_1^{-1}$. Let $x^{(1)}_N:=c_1x_N$. In the basis $B_1:=\{x_1,\dots,x_{N-1},x^{(1)}_N\}$, the element $W$ is represented by a reduced cyclic word of the form $v_1x^{(1)}_Nw{x^{(1)}_N}^{-1}$, and the length of $v_1$ in $B_1$ is strictly smaller than the length of $v$ in $B$. In addition, the element $w\in F_N$ is represented by the same reduced word in $B$ and in $B'$. Repeating the above argument shows that there exists $c_2\in B_1$ such that the first letter of $v_1$ is $c_2^{-1}$ and its last letter is $c_2$. Letting $x_N^{(2)}:=c_2x_N^{(1)}$, the element $W$ is represented by a reduced cyclic word of the form $v_2x_N^{(2)}w{x_N^{(2)}}^{-1}$ in the basis $B_2:=\{x_1,\dots,x_{N-1},x_N^{(2)}\}$, and the length of $v_2$ in $B_2$ is strictly smaller than the length of $v_1$ in $B_1$. One can then repeat this process infinitely often, contradicting the fact that the lengths of the words representing $W$ in the bases we get along the process form a strictly decreasing sequence of positive integers.
\end{proof}

\begin{proof} [Proof of Proposition \ref{slide-separable}]
Let $T,T'\in\overline{cv_N}$ be special-pull-equivalent. Assume that $T\neq T'$, and let $\widehat{T}\in\overline{cv_N}$ be a tree with a single orbit of edges with trivial stabilizer $e$, such that $T$ and $T'$ are both pulls of $\widehat{T}$, with pulling elements $g_1, g_2, g'_1$ and $g'_2$. Equivariantly collapsing the complement of $e$ to a point in $\widehat{T}$ yields a splitting $F_N=A\ast$ (we denote by $t$ a stable letter). Any element $w\in F_N$ either belongs to a conjugate of $A$, or of the cyclic subgroup of $F_N$ generated by $t$, or is conjugated to an element of the form $w_1t^{\alpha_1}w_2t^{\alpha_2}\dots w_kt^{\alpha_k}$, with $\alpha_i\in\mathbb{Z}\smallsetminus\{0\}$ and $w_i\in A\smallsetminus\{e\}$ for all $i\in\{1,\dots,k\}$. Such an element has the same translation length in $T$ and $T'$, unless it is of the form $t^{\epsilon}g_i^kt^{-\epsilon}w$ or $t^{\epsilon}{g'_i}^kt^{-\epsilon}w$ for some $\epsilon=\pm 1$, some $k\in\mathbb{Z}\smallsetminus\{0\}$, and some $i\in\{1,2\}$ with $g_i$ (or $g'_i$) nonsimple in $A\smallsetminus\{e\}$. As any element of $F_N$ of this form is nonsimple by Proposition \ref{class-separable}, all simple elements of $F_N$ have the same translation length in $T$ and $T'$.
\end{proof}

\subsection{The primitive compactification of outer space}\label{sec-primitive-compactification}

Our main result gives a description of a new compactification of outer space, which we call the \emph{primitive compactification}, defined by restricting translation lengths functions to the set $\mathcal{P}_N$ of primitive elements of $F_N$ in Culler and Morgan's construction. Our motivation for introducing this compactification comes from our description of the compactification of outer space by horofunctions, which is itself motivated by the desire to study random walks on $\text{Out}(F_N)$. In \cite{Hor14-2}, we will prove that the compactification of outer space by horofunctions is isomorphic to the primitive compactification. Let 

\begin{displaymath}
i_{prim}:CV_N\to\mathbb{PR}^{\mathcal{P}_N}
\end{displaymath} 

\noindent be the map obtained from the map $i$ defined in Section \ref{sec-cvn} by only considering translation lengths of primitive elements of $F_N$. The relation $\sim$ again denotes the primitive-equivalence relation defined above.

\begin{theo}
The map $i_{prim}$ is a homeomorphism onto its image. The closure $\overline{i_{prim}(CV_N)}$ is compact, and homeomorphic to $\overline{CV_N}/{\sim}$.
\end{theo}

This means that $\overline{i_{prim}(CV_N)}$ is indeed a compactification of $CV_N$.

\begin{proof}
Continuity of $i_{prim}$ follows from the continuity of $i$, and injectivity of $i_{prim}$ was proved in \cite[Theorem 3.4]{CFKM12} as a consequence of White's theorem (this is a particular case of our main result). To show that $i_{prim}$ is an embedding, we let $(T_n)_{n\in\mathbb{N}}\in CV_N^{\mathbb{N}}$ be a sequence that leaves every compact subspace of $CV_N$, and let $T\in CV_N$. By compactness of $\overline{CV_N}$, some subsequence of $(T_n)_{n\in\mathbb{N}}$ converges to a tree $T_{\infty}\in\overline{CV_N}\smallsetminus CV_N$, and Theorem \ref{equivalences} implies that $T_{\infty}$ is not primitive-equivalent to $T$ (this actually only uses the particular case of Theorem \ref{equivalences} where one of the trees belongs to $CV_N$, which can easily be deduced from Theorem \ref{Algom-Kfir}). Therefore, the sequence $(i_{prim}(T_n))_{n\in\mathbb{N}}$ does not converge to $i_{prim}(T)$, showing that $i_{prim}$ is an embedding. Compactness of $\overline{i_{prim}(CV_N)}$ follows from compactness of $\overline{i(CV_N)}$. By definition, the map $i_{prim}$ extends to a bijective continuous map, and hence a homeomorphism, from the compact space $\overline{CV_N}/{\sim}$ to the Hausdorff space $\overline{i_{prim}(CV_N)}$.
\end{proof}

\section{Approximations of trees} \label{sec-approximations}

The next two sections aim at developing techniques that will turn out to be useful for extending White's theorem to trees in the boundary of outer space, and describing the lack of spectral rigidity of the set $\mathcal{P}_N$ of primitive elements of $F_N$ in $\overline{cv_N}$. In the present section, building on ideas of Bestvina and Feighn \cite{BF94} and Guirardel \cite{Gui98}, we provide nice approximations for a wide class of trees in $\overline{cv_N}$ by "nicer" trees, see Theorems \ref{approximation-by-simplicial} and \ref{approximation-by-separable}. 

\subsection{Geometric trees} \label{sec-geometric}

Of particular interest are trees in $\overline{cv_N}$ which are dual to measured foliations on some $2$-complexes, which are called \emph{geometric} trees. Geometric trees can be decomposed in a nice and controlled way, and can be used to approximate every tree $T\in\overline{cv_N}$. We recall a few facts about this class of trees, and refer the reader to \cite{BF95} or \cite{GLP94} for details.

A \emph{system of partial isometries} $X$ of a finite tree or multi-interval $K$ is a finite collection of isometries $\phi_j:A_j\to B_j$ between nonempty finite subtrees of $K$. The subtrees $A_j$ and $B_j$ are called the \emph{bases} of $X$, and $\phi_j$ is called a \emph{singleton} if its bases are reduced to points. The \emph{suspension} of $X$ is the foliated $2$-complex $\Sigma$ built in the following way. Start with the union of $K$ (foliated by points) and \emph{bands} $A_j\times[0,1]$ (foliated by $\{\ast\}\times [0,1]$). For all $t\in A_j$, glue $(t,0)\in A_j\times\{0\}$ with $t\in A_j$ and $(t,1)\in A_j\times\{1\}$ with $\phi_j(t)\in B_j$. There is a natural transverse measure on each band given by the metric on the base $A_j$. This induces a transverse measure on $\Sigma$. We will denote by $(\Sigma,\mathcal{F},\mu)$ (or simply by $\Sigma$ if the context is clear) the band complex $\Sigma$ equipped with its foliation $\mathcal{F}$ and its transverse measure $\mu$.

\indent Associated to a system of $k$ partial isometries $X$ (or its corresponding measured foliated band complex $(\Sigma,\mathcal{F},\mu)$), together with a set $\mathcal{C}$ of curves contained in leaves of $\Sigma$, is a dual $\mathbb{R}$-tree, constructed as follows. Choose a basepoint $\ast$ on $\Sigma$. When $K$ is a finite tree (or a multi-interval in which the extremities of the intervals are joined by singletons, in such a way that collapsing the  subsegments of the leaves determined by these singletons to points yields a finite tree), the fundamental group of $\Sigma$ is naturally identified with the free group $F_k$ having one generator for each partial isometry in $X$. Let $N$ denote the subgroup of $F_k$ normally generated by the free homotopy classes of the curves in $\mathcal{C}$, and let $G(X):=F_k/N$. There is a canonical epimorphism $\rho:F_k\to G(X)$. We denote by $\overline{\Sigma}$ the covering space of $\Sigma$ corresponding to $\rho$. The measured foliation on $\Sigma$ lifts to a measured foliation on $\overline{\Sigma}$, we denote by $\overline{\mu}$ the transverse measure on $\overline{\Sigma}$. Define a pseudo-metric on $\overline{\Sigma}$ by $\delta(x,y):=\inf_{\gamma}\overline{\mu}(\gamma)$, where the infimum is taken over all paths joining $x$ to $y$ in $\overline{\Sigma}$ (and $\overline{\mu}(\gamma)$ is obtained by integrating the measure $\overline{\mu}$ along the path $\gamma$). The metric space obtained by making this pseudo-distance Hausdorff (sometimes called the \emph{leaf space made Hausdorff}) is an $\mathbb{R}$-tree \cite[Proposition 1.7]{LP97}, which we denote by $T(X,\mathcal{C})$ (or equivalently $T(\Sigma,\mathcal{C})$). It is naturally equipped with an isometric action of $G(X)$. An $\mathbb{R}$-tree equipped with an action of a finitely presented group $G$ is called \emph{geometric} if there exists a system of partial isometries $X$, and a set of curves $\mathcal{C}$ contained in leaves of the associated measured foliated band complex, such that $G=G(X)$ and $T=T(X,\mathcal{C})$. Otherwise it is called \emph{nongeometric}. Let $\Sigma$, $\Sigma'$ be two measured foliated band complexes, together with sets of curves $\mathcal{C}$ and $\mathcal{C}'$. We call $(\Sigma,\mathcal{C})$ and $(\Sigma',\mathcal{C}')$ \emph{equivalent} if $T(\Sigma,\mathcal{C})=T(\Sigma',\mathcal{C}')$. Let $\Sigma^*$ denote $\Sigma$ minus its singletons. We say that $\Sigma$ has \emph{pure components} if $K$ is a multi-interval, and in each component of $\Sigma^*$, each finite singular $\mathring{X}$-orbit (i.e. the orbit of each point under the restrictions of the partial isometries in $X$, or their inverses, to the interior of their bases) is reduced to one point in $\partial K$ (an orbit is \emph{singular} if it contains a point in the boundary of some base). 

Given a geometric $F_N$-tree $T$, there is a way of producing a system of isometries $X$ on a finite tree $K$, so that $T=T(X,\emptyset)$. Fix a free basis $\{g_1,\dots,g_N\}$ of $F_N$, and let $K$ be a finite subtree of $T$. For all $i\in\{1,\dots,N\}$, the generator $g_i$ defines a partial isometry of $K$, with domain $g_i^{-1}(K)\cap K$ and image $K\cap g_i(K)$, and we may assume $K$ to be sufficiently big, so that these bases are nondegenerate. If $T$ is geometric, then $K$ can be chosen so that the associated geometric tree is equal to $T$ \cite[Proposition II.1]{GL95}. The following theorem provides a "normal form" for systems of partial isometries dual to a given geometric $F_N$-tree.

\begin{theo} (Imanishi \cite{Ima79}, Gaboriau-Levitt-Paulin \cite{GLP94}) \label{Imanishi}
Let $T$ be a geometric $F_N$-tree. Then there exist a system of partial isometries $X$ having pure components, and a set of curves $\mathcal{C}$ contained in leaves of $\Sigma$, such that $T=T(X,\mathcal{C})$. The subcomplex $\Sigma^*$ is a disjoint union of finitely many open $\mathring{X}$-invariant sets, and if $U$ is one of these sets, then either every leaf contained in $U$ is compact (in which case $U$ is called a \emph{family of finite orbits}), or else every leaf contained in $U$ is dense in $U$ (in which case $U$ is called \emph{minimal}). Furthermore, the system $X$ may be chosen in such a way that all families of finite orbits are orientable (i.e. no $\mathring{X}$-word fixes a point in an orbit and reverses orientation).
\end{theo}

One can give the following classification of minimal components. Starting from a foliated band complex $\Sigma_0$ associated to a minimal system of partial isometries $X_0$ on a finite tree or multi-interval $K_0$, we define a new band complex $\Sigma_1$ in the following way. Let $K_1$ denote the set of points in $K_0$ which belong to at least two bases of $\Sigma_0$. Let $X_1$ be the system of partial isometries of $K_1$ obtained by restricting the elements of $X_0$ to $K_1$. We define $\Sigma_1$ to be the suspension of $X_1$. Starting from $\Sigma_0$ and iterating this process, we build a sequence of foliated band complexes $\Sigma_i$. If for all $i\in\mathbb{N}$ we have $\Sigma_{i+1}\neq\Sigma_i$, we say that $\Sigma_0$ is \emph{exotic} (or \emph{Levitt}, or \emph{thin}), otherwise $\Sigma_0$ is a measured foliation on a compact surface \cite{GLP94}. (In the case of $F_N$-trees, the homogeneous case described in \cite[Section 4]{GLP94} cannot occur, see \cite[Proposition 1.8]{BF94}). A band $B=b\times [-1,1]$ of a band complex is \emph{very naked} if $b\times (-1,1)$ does not meet the curves in $\mathcal{C}$. Exotic components have the following property.

\begin{prop} \label{isolated-exotic} (Bestvina-Feighn \cite{BF95}, Gaboriau-Levitt-Paulin \cite{GLP94}, see also \cite[Section 7.1]{Gui98})
If $T\in\overline{cv_N}$ contains an exotic minimal component, then there exist a band complex $X$ satisfying the conclusions of Theorem \ref{Imanishi} and a collection of curves $\mathcal{C}$ in $X$ such that $T=T(X,\mathcal{C})$, and $X$ contains a very naked band (contained in an exotic component of $X$).
\end{prop}

The structure of band complexes which only have simplicial and surface components is also well-understood, thanks to the following results of Bestvina and Feighn. 

\begin{prop} \label{outer-limits-structure} (Bestvina-Feighn \cite[Proposition 5.1]{BF94})
Let $X$ be a band complex with only simplicial and surface components dual to an $F_N$-tree $T\in\overline{cv_N}$. Then there exists another band complex $X'$ dual to $T$ of the form $X'=(S\cup A\cup\Gamma)\cup_fG$ such that
\begin{itemize}
\item $S$ is a (possibly disconnected) compact surface, none of whose components is homeomorphic to an annulus or a Möbius band, and each connected component of $S$ is equipped with a minimal foliation, and

\item $\Gamma$ is a finite metric graph, and

\item $G$ is a finite graph with no valence $1$ vertices and empty foliation, and

\item $A$ is a finite disjoint union of annuli foliated by essential loops, and 

\item $f:\partial S\cup\partial A\cup F\to G$, where $F$ is a finite subset of $S\cup A\cup \Gamma$ and $f$ is essential on each component of $\partial S\cup\partial A$.
\end{itemize}
\end{prop}

\begin{prop} \label{outer-limits} (Bestvina-Feighn \cite[Lemma 4.1]{BF94})
Let $Y$ be a finite graph, and $S$ a compact (possibly disconnected) surface. Let $f:\partial S\to Y$ be a map that is essential on each boundary component. Assume that $X:=S\cup_f Y$ has free fundamental group. Then there exist a finite graph $Y'$ and a homotopy equivalence $\psi:Y\to S^1\vee Y'$ such that $\psi\circ f:\partial S\to S^1\vee Y'$ is homotopic to a map that sends one boundary component of $S$ homeomorphically onto $S^1$, and sends all other boundary components of $S$ into $Y'$.
\end{prop} 

We call the boundary component of $S$ that is sent homeomorphically to $S^1$ a \emph{distinguished circle}. 
\\
\\
\indent We finish this section by explaining how geometric trees can be used to approximate all actions in $\overline{cv_N}$, and give a characterization of geometric trees due to Levitt and Paulin \cite{LP97}. Let $T\in\overline{cv_N}$. Following \cite{GS90}, we say that a sequence $(T_n)_{n\in\mathbb{N}}$ of trees in $\overline{cv_N}$ \emph{converges strongly} to $T$ if there exist surjective $F_N$-equivariant morphisms $f_{np}:T_n\to T_p$ for all $n<p$, and $f_n:T_n\to T$ for all $n\in\mathbb{N}$ such that 
\begin{itemize}
\item for all $n<p$, we have $f_p\circ f_{np}=f_n$, and 
\item for all $n\in\mathbb{N}$ and all $x,y\in T_n$, there exists $p\ge n$ such that $d_{T_p}(f_{np}(x),f_{np}(y))=d_{T}(f_n(x),f_n(y))$. 
\end{itemize}
\noindent The following result is due to Levitt and Paulin \cite[Theorem 2.6]{LP97}, see also \cite[Proposition II.1]{GL95} where the minimality statement appears. The fact that the trees $T_n$ can be chosen to belong to $\overline{cv_N}$ follows from \cite[Proposition 1.8]{BF94} and \cite[Corollary I.6]{GL95}.

\begin{prop}\label{non-geometric} (Levitt-Paulin \cite[Theorem 2.6]{LP97}, Gaboriau-Levitt \cite[Proposition II.1]{GL95}) 
For all $T\in \overline{cv_N}$, there exists a sequence $(T_n)_{n\in\mathbb{N}}\in\overline{cv_N}$ of minimal geometric $F_N$-trees which converges strongly to $T$. A tree is geometric if and only if it cannot occur as such a strong limit in a nonstationary way.
\end{prop}

The trees $T_n$ can be constructed from $T$ by applying the construction preceding Theorem \ref{Imanishi} to a well-chosen exhaustion of $T$ by finite trees $K_n$. In particular, the morphisms $f_{np}$ and $f_n$ can be chosen to be injective in restriction to every segment of $K_n$ which has a translate in $K_n$. By choosing $K_n$ to contain an edge in each orbit of edges in $T$, we can thus assume $f_{np}$ and $f_n$ to be injective on segments with nontrivial stabilizers of $T_n$.

\subsection{Approximations of $F_N$-trees with dense orbits by free and simplicial actions} \label{sec-approx-by-simplicial}

A \emph{Lipschitz approximation} of a tree $T\in\overline{cv_N}$ is a sequence of trees $(T_n)_{n\in\mathbb{N}}\in\overline{cv_N}^{\mathbb{N}}$ converging (non-projectively) to $T$, together with $1$-Lipschitz $F_N$-equivariant maps $f_n:T_n\to T$ for all $n\in\mathbb{N}$. We give a characterization of trees in $\overline{cv_N}$ that admit a Lipschitz approximation by free, simplicial actions.

\begin{theo} \label{approximation-by-simplicial}
A tree $T\in\overline{cv_N}$ admits a Lipschitz approximation by elements of $cv_N$ if and only if all arc stabilizers in $T$ are trivial. 
\end{theo}

Let $T\in\overline{cv_N}$ be a geometric tree, and let $X$ be a system of partial isometries associated to $T$ given by Theorem \ref{Imanishi}, together with a set of curves $\mathcal{C}$ contained in the leaves of $\Sigma$. Assume that some band $B$ of $\Sigma$ is very naked. For small $\delta>0$, let $\Sigma_{\delta}$ be a band complex obtained by narrowing $B$ of width $\delta$ from one of its boundary leaves, see Figure \ref{fig-narrowing}. The inclusion $\Sigma_{\delta}\subset\Sigma$ is a homotopy equivalence, so there is an epimorphism $\rho_{\delta}:\pi_1(\Sigma_{\delta})\to F_N$, whose kernel is normally generated by the free homotopy classes of the curves in $\mathcal{C}$, which are still contained in leaves of $\Sigma_{\delta}$. Denote by $\overline{\Sigma}$ and $\overline{\Sigma_{\delta}}$ the covering spaces corresponding to $\rho$ and $\rho_{\delta}$, respectively. Let $T_{\delta}$ be the minimal subtree of the $F_N$-tree obtained by making the leaf space of $\overline{\Sigma_{\delta}}$ Hausdorff. There is a natural $F_N$-equivariant morphism of $\mathbb{R}$-trees from $T_{\delta}$ to $T$ induced by the inclusion $\overline{\Sigma_{\delta}}\subset\overline{\Sigma}$. 

\begin{figure}
\begin{center}
\input{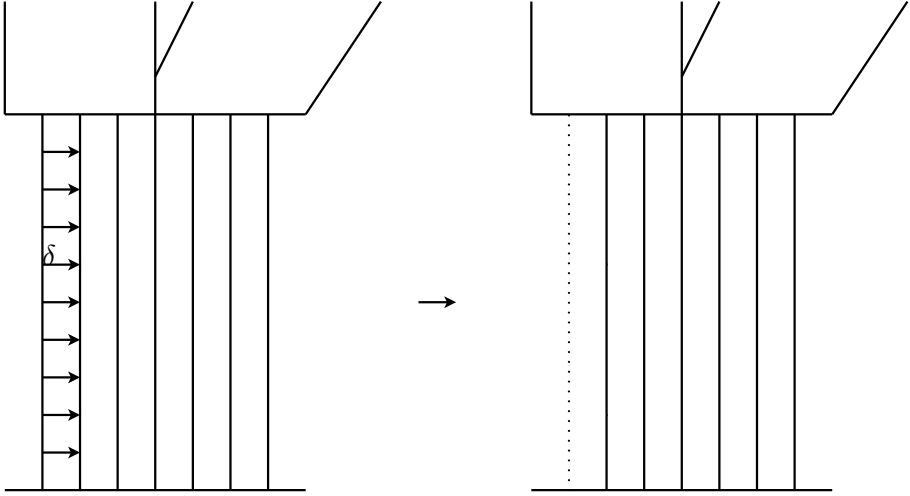}
\caption{Narrowing a band.}
\label{fig-narrowing}
\end{center}
\end{figure}

\begin{lemma}\label{narrowing} (Guirardel \cite[Section 7.2]{Gui98})
The trees $T_{\delta}$ converge to $T$ as $\delta$ goes to $0$.
\end{lemma}

\begin{proof} [Proof of Theorem \ref{approximation-by-simplicial}]
First assume that $T$ admits a Lipschitz approximation by a sequence $(T_n)_{n\in\mathbb{N}}$ of trees in $cv_N$. As there exist $1$-Lipschitz $F_N$-equivariant maps $f_n:T_n\to T$, by Proposition \ref{qvol}, we have $qvol(T)\le qvol(T_n)$ for all $n\in\mathbb{N}$. However, if $T$ has a nontrivial arc stabilizer, then $qvol(T)>\limsup_{n\to +\infty} qvol(T_n)$. Hence $T$ has trivial arc stabilizers.

Conversely, let $T$ be a tree in $\overline{cv_N}$ with trivial arc stabilizers. First assume that $T$ contains an exotic component. Then Proposition \ref{isolated-exotic} yields an equivalent band complex which contains a very naked band $B$, to which we can apply the narrowing process. Guirardel shows in \cite[Section 7]{Gui98} that we can choose $\delta>0$ arbitrarily small and get a tree $T_{\delta}$, in which the exotic component of $T$ has been replaced by new simplicial and exotic components, and the number $E(\Sigma_{\delta})$ of ends of singular leaves satisfies $E(\Sigma_{\delta})<E(\Sigma)$. Iterating the construction a finite number of times yields an approximation of $T$ in which the minimal component $T_v$ has been replaced by a simplicial part with trivial edge stabilizers. Iterating this process, we can approximate all exotic components in $T$ without creating arc stabilizers. Hence we are left with a band complex which can be assumed to have the form prescribed by Proposition \ref{outer-limits-structure}.

As $T$ has trivial arc stabilizers, this band complex contains no annulus. Assume that it contains some surface component, and let $C$ be a distinguished circle provided by Proposition \ref{outer-limits}. One can "narrow" the surface that contains $C$ from its boundary along width $\delta>0$ to either create compact leaves, or leaves having a single end (except for at most finitely many of them), see Figure \ref{fig-narrowing-surface}. However, in a minimal surface component, all half-leaves are dense, so in the new band complex $\Sigma_{\delta}$ created in this way, the surface containing $C$ has been replaced by a simplicial component, with trivial arc stabilizers. As in Lemma \ref{narrowing}, the trees $T_{\delta}$ dual to the band complex $\Sigma_{\delta}$ converge to $T$ as $\delta$ tends to $0$, and they come with $F_N$-equivariant morphisms from $T_{\delta}$ to $T$.  Iterating this process, we successively approximate all the surface components by simplicial components with trivial edge stabilizers. Finally, we can approximate all vertices with nontrivial stabilizer in the quotient graph by roses having arbitrarily small petals to get a Lipschitz approximation of $T$ by elements of $cv_N$.
\end{proof}

\begin{figure}
\begin{center}
\input{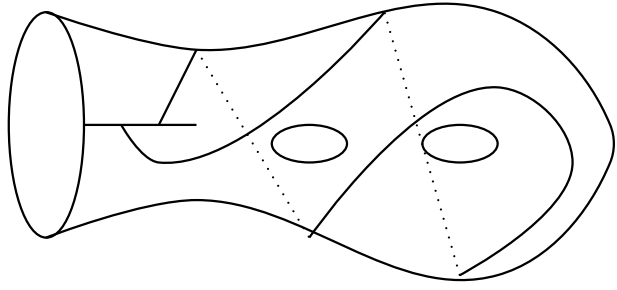}
\caption{Narrowing a surface component creates compact leaves.}
\label{fig-narrowing-surface}
\end{center}
\end{figure}

As a consequence of Theorem \ref{approximation-by-simplicial}, we show that any Lipschitz $F_N$-equivariant map between $F_N$-trees with dense orbits preserves alignment. In particular, any $F_N$-equivariant morphism between minimal $F_N$-trees with dense orbits is an isometry. Let $T,T'\in\overline{cv_N}$, and $f:T\to T'$ be an $F_N$-equivariant map. The \emph{bounded cancellation constant} of $f$, denoted by $BCC(f)$, is defined to be the supremum of all real numbers $B$ with the property that there exist $a,b,c\in T$ with $b\in [a,c]$, such that $d_{T'}(f(b),[f(a),f(c)])=B$. Notice that an $F_N$-equivariant map $f:T\to T'$ preserves alignment if and only if $BCC(f)=0$. We denote by $\text{Lip}(f)$ the Lipschitz constant of $f$.

\begin{prop} \label{bcc} (Bestvina-Feighn-Handel \cite[Lemma 3.1]{BFH97})
Let $T\in cv_N$ and $T'\in \overline{cv_N}$, and let $f:T\to T'$ be an $F_N$-equivariant map. Then $BCC(f)\le Lip(f) qvol(T)$.
\end{prop}

\begin{cor} \label{alignment-preserving}
Let $T,T'\in\overline{cv_N}$ have dense orbits, and let $f:T\to \overline{T'}$ be a Lipschitz $F_N$-equivariant map. Then $f$ preserves alignment. 
\end{cor}

\begin{proof}
Let $a,b,c\in T$ with $b\in [a,c]$, and let $C:=d_{\overline{T'}}(f(b),[f(a),f(c)])$. Assume by contradiction that $C>0$. As $T$ has dense orbits, all arc stabilizers in $T$ are trivial (Lemma \ref{dense-arcs}), so Theorem \ref{approximation-by-simplicial} provides a Lipschitz approximation $(T_n)_{n\in\mathbb{N}}$ of $T$ by free and simplicial $F_N$-trees. By Proposition \ref{qvol}, the quotient volume of $T_n$ converges to $0$ as $n$ goes to infinity. By definition of a Lipschitz approximation, for all $n\in\mathbb{N}$, there exists a $1$-Lipschitz $F_N$-equivariant map $f_n:T_n\to T$. Minimality of $T$ implies that $f_n$ is surjective for all $n\in\mathbb{N}$. Composing $f_n$ with $f$ yields a $\text{Lip}(f)$-Lipschitz $F_N$-equivariant map $f'_n:T_n\to \overline{T'}$. Tightening $f'_n$ on edges if necessary (which does not increase its Lipschitz constant), we can assume that $f'_n$ is linear on edges. Slightly perturbing $f'_n$ on the vertices of $T_n$, and extending it linearly on the edges of $T_n$ again if necessary, we get the existence of a $\text{Lip}(f)$-Lipschitz $F_N$-equivariant map $f''_n:T_n\to T'$, with $d_{\overline{T'}}(f'_n(x),f''_n(x))\le\frac{C}{4}$ for all $x\in T_n$. By Proposition \ref{bcc}, the bounded cancellation constant $BCC(f''_n)$ tends to $0$ as $n$ goes to infinity. For all $n\in\mathbb{N}$, let $a_n$ (resp. $c_n$) be a preimage of $a$ (resp. $c$) by $f_n$ in $T_n$. Then there exists $b_n\in [a_n,c_n]$ such that $f_n(b_n)=b$. We have $d_{\overline{T'}}(f'_n(b_n),[f'_n(a_n),f'_n(c_n)])=C$, so $d_{T'}(f''_n(b_n),[f''_n(a_n),f''_n(c_n)])\ge\frac{C}{2}$. This implies that $BCC(f''_n)\ge \frac{C}{2}$ for all $n\in\mathbb{N}$, a contradiction. Hence $f$ preserves alignment.
\end{proof} 

\begin{cor}\label{morphism-isometry}
Let $T,T'\in\overline{cv_N}$ have dense orbits. Then any $F_N$-equivariant morphism from $T$ to $T'$ is an isometry.
\qed
\end{cor}

\subsection{Approximations by trees having two edges with trivial stabilizers} \label{sec-approx-by-separable}

An $F_N$-tree $T\in\overline{cv_N}$ is \emph{good} if there exists a Lipschitz approximation $(T_n)_{n\in\mathbb{N}}\in\overline{cv_N}^{\mathbb{N}}$ of $T$ such that for all $n\in\mathbb{N}$, the tree $T_n^{simpl}$ contains at least two $F_N$-orbits of edges with trivial stabilizers. The following statement will be used in Section \ref{sec-rigidity} to describe the lack of rigidity of the set $\mathcal{P}_N$ in $\overline{cv_N}$. We recall the definition of a pull from Section \ref{sec-slide-equiv}.

\begin{theo} \label{approximation-by-separable}
Every tree $T\in\overline{cv_N}$ is a pull of a good tree. More precisely, for all $T\in\overline{cv_N}$, either $T$ is good, or there exists a good tree $T'\in\overline{cv_N}$ which has exactly one orbit of edges with trivial stabilizer, such that $T$ is a pull of $T'$.
\end{theo}

\begin{proof}
We argue differently depending on whether $T$ is geometric or not.
\\
\\
\noindent\textit{Case 1} : The tree $T$ is geometric.
\\
\\
\textit{Case 1.1} : The tree $T$ contains an exotic component.
\\ Applying the same narrowing process as in the proof of Theorem \ref{approximation-by-simplicial} to this exotic component yields a Lipschitz approximation $(T_n)_{n\in\mathbb{N}}$ of $T$, in which the exotic component -- dual to some subtree $T_v$ of $T$ -- is replaced by a family of finite orbits, dual to some tree $T^1$ with trivial edge stabilizers. If for some $n\in\mathbb{N}$, the tree $T_n^{simpl}$ contains at most one orbit of edges with trivial stabilizer, then $T^1$ is the Bass-Serre tree of a one-edge free splitting, and Lemmas \ref{dense-arcs} and \ref{morphism-1} imply that a morphism $f:T^1\to T_v$ cannot fold any U-turn. If $T^1$ is the Bass-Serre tree of a splitting of the form $F_i\ast F_{k-i}$, then $f$ might only reduce the length of the unique orbit of edges in $T^1$, and $T_n$ cannot converge to $T$. If $T^1$ the Bass-Serre tree of a splitting of the form $F_{k-1}\ast$, then $f$ can either reduce the length of the unique orbit of edges in $T^1$, or create a second orbit of edges with trivial stabilizers, in which case we can assume $T_n^{simpl}$ to contain two $F_N$-orbits of edges with trivial stabilizers. 
\\
\\
\textit{Case 1.2} : The tree $T$ is dual to a band complex $\Sigma$ which has the structure prescribed by Proposition \ref{outer-limits-structure}. 
\\ If $\Sigma$ contains no surface component and no annulus, then $T$ is simplicial and has trivial edge stabilizers. So either $T$ contains two $F_N$-orbits of edges with trivial stabilizers, or $T$ is the Bass-Serre tree of a one-edge free splitting of $F_N$, in which case $T$ can be approximated by blowing up its vertex groups, adding a small loop with trivial stabilizer of length going to $0$. Otherwise, let $C$ be a distinguished circle provided by Proposition \ref{outer-limits}. If $C$ belongs to a surface component, then as in Case 1.1 we get a Lipschitz approximation of $T$ by trees having at least two orbits of edges with trivial stabilizers. We now assume that $C$ belongs to an annulus $A$. In this case, narrowing a band corresponds to unfolding an edge in the dual tree, see Figure \ref{fig-narrow-simplicial}, and this operation creates an orbit of edges $e$ with trivial stabilizer. This operation does not affect minimality of the dual tree. If there is another simplicial orbit of edges with trivial stabilizer in the tree dual to $\Sigma$, then $T$ is good. It may happen that some extremity of $e$ has cyclic stabilizer, and is such that there are exactly two $F_N$-orbits of edges coming out of it, one of which has nontrivial stabilizer. We let $T'$ be the tree obtained from $T$ by totally unfolding the edges with nontrivial stabilizers coming out of such extremities of $e$. This operation does not create obtrusive powers or tripod stabilizers, so the tree $T'$ is again very small, and by definition $T$ is a pull of $T'$. In addition, if we equivariantly remove the edge with trivial stabilizer of $T'$ we have just constructed, we get (at least) one tree, to which we can apply the above argument (this tree might not be minimal for the action of its stabilizer, if $e$ projects to a loop-edge in the associated graph of actions, but the above argument still works in this case). If all distinguished circles of $T'$ are contained in surface components, then the above argument shows that $T'$ is good. Otherwise, one can again unfold an annulus. This operation creates a second edge with trivial stabilizer, again showing that $T'$ is good.
\\
\\
\textit{Case 2} : The tree $T$ is nongeometric.\\
\\
Let $(T_n)_{n\in\mathbb{N}}\in\overline{cv_N}^{\mathbb{N}}$ be a sequence of minimal geometric $F_N$-trees converging strongly to $T$, given by Proposition \ref{non-geometric}. Denote by $f_n:T_n\to T$ and $f_{n,p}:T_n\to T_p$ the corresponding morphisms, which might be assumed to be injective on the edges in $T_n$ with nontrivial stabilizers. In particular, the sequence $(T_n)_{n\in\mathbb{N}}$ is a Lipschitz approximation of $T$, so it is enough to show that $T_n$ can be assumed to contain two $F_N$-orbits of edges with trivial stabilizers for all $n\in\mathbb{N}$. Assume by contradiction that for some $n\in\mathbb{N}$, the tree $T_n$ contains at most one edge $e_n$ with trivial stabilizer in its simplicial part. The morphism $f_{n,n+1}$ cannot identify 

\begin{itemize}
\item two initial subsegments of edges in the simplicial part of $T_n$ with distinct nontrivial stabilizer (Lemma \ref{morphism-2}), nor
\item a nontrivial subsegment of an edge with nontrivial stabilizer with a path lying in a vertex tree of $T_n$ with dense orbits (Lemma \ref{morphism-3}), nor
\item two arcs lying in a vertex tree of $T_n$ with dense orbits (Corollary \ref{morphism-isometry}).
\end{itemize}

\noindent If $f_{n,n+1}$ identifies a subsegment $J$ of $e_n$ with a subsegment $J'$ of one of the translates of $e_n$, such that $J'$ meets the $F_N$-orbit of $J$ in a single point (this might happen if $e_n$ projects to a loop-edge in the associated graph of actions), then we can replace $T_n$ by a tree with two $F_N$-orbits of edges with trivial stabilizers. So up to reducing the length of $e_n$ in $T_n$, we might assume that to pass from $T_n$ to $T_{n+1}$, we need only fold a subsegment of the edge $e_n$ along some path in $T_n$, at each of its extremities, and iterating the same argument shows that for all $k\ge n$, the tree $T_k$ has a unique edge with trivial stabilizer, and in order to pass from $T_k$ to $T_{k+1}$, one has to fold a subsegment of $e_k$ equivariantly along some path in $T_k$. Assume that the sequence $(T_k)_{k\in\mathbb{N}}$ is nonstationary. Then we can find $x\in e_n$ and a nonstationary sequence $(x_k)_{k\in\mathbb{N}}$ of elements of $e_n$ converging to $x$ such that the subsegment $[a,x_k]$ of $e_n$ is folded when passing from $T_n$ to $T_k$ (where $a$ denotes one of the extremities of $e_n$). For all $k\in\mathbb{N}$, let $y_k$ be a point in $T_n$ that is identified with $x_k$ during the folding process. In particular, the sequence $(y_k)_{k\in\mathbb{N}}$ is bounded, and the segments $[x,y_k]$ form an increasing sequence of segments in $T_n$. So $(y_k)_{k\in\mathbb{N}}$ converges to some point $y\in T_n$, and $f_n(x)=f_n(y)$. However, for all $k\in\mathbb{N}$, we have $d(f_{n,k}(x),f_{n,k}(y))>0$, contradicting strong convergence of the sequence $(T_k)_{k\in\mathbb{N}}$ to $T$. 
\end{proof}

\begin{figure}
\begin{center}
\def\JPicScale{.5}
\input{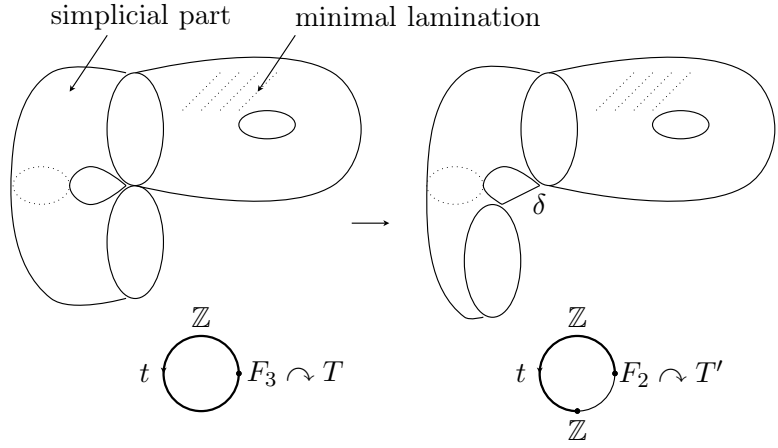}
\caption{Narrowing a band in a simplicial component.}
\label{fig-narrow-simplicial}
\end{center}
\end{figure}

\begin{rk}\label{rk-surface}
The proof of Theorem \ref{approximation-by-separable} shows that if $T$ is either nongeometric, or contains an exotic component, then $T$ is good. The constructions made in Section \ref{sec-slide-equiv} provide examples of trees with simplicial and surface components which are obtained as pulls of good trees but are not good, as will follow from Proposition \ref{case-good-approx}.
\end{rk}

\section{Limits of Lipschitz maps between very small $F_N$-trees} \label{sec-limits}

The goal of this section is to explain how to construct Lipschitz $F_N$-equivariant maps between very small $F_N$-trees, by using a limiting process (Theorem \ref{limits}). We start by recalling some facts about the equivariant Gromov-Hausdorff topology on the space of $F_N$-trees.

\subsection{Equivariant Gromov-Hausdorff topology} \label{sec-Gromov}

In \cite{Pau88}, Paulin introduced yet another topology on $\overline{cv_N}$. Let $T$ and $T'$ be two $F_N$-trees, let $K\subset T$ and $K'\subset T'$ be finite subsets, let $P\subset F_N$ be a finite subset of $F_N$, and let $\epsilon>0$. A \emph{$P$-equivariant $\epsilon$-relation} between $K$ and $K'$ is a subset $R\subseteq K\times K'$ whose projection to each factor is surjective, such that for all $(x,x'),(y,y')\in R$ and all $g,h\in P$, we have $|d_T(gx,hy)-d_{T'}(gx',hy')|<\epsilon$. We denote by $O(T,K,P,\epsilon)$ the set of $F_N$-trees $T'$ for which there exists a finite subset $K'\subset T'$ and a $P$-equivariant $\epsilon$-relation $R\subseteq K\times K'$. Paulin showed that these sets define a basis of open sets for a topology on the set of $F_N$-trees, called the \emph{equivariant Gromov-Hausdorff topology} \cite{Pau88}. This topology is equivalent to the axes topology on $\overline{cv_N}$ \cite{Pau89}.

Let $T$ be an $F_N$-tree, and let $(T_n)_{n\in\mathbb{N}}$ be a sequence of $F_N$-trees that converges to $T$ in the equivariant Gromov-Hausdorff topology. Let $x\in T$. Let $(K^k)_{k\in\mathbb{N}}$ be an increasing sequence of finite subsets of $T$ containing $x$, such that the finite trees spanned by the subsets $K^k$ yield an exhaustion of $T$, and let $F_N=\bigcup_{k\in\mathbb{N}}P^k$ be an exhaustion of $F_N$ by finite subsets. For all $k\in\mathbb{N}$, let $n_k$ be the smallest integer such that $T_{n_k}\in O(T,K^k,P^k,\frac{1}{k})$. For all $n\in \{n_k,\dots,n_{k+1}-1\}$, we can find a finite subset $K_n\subset T_n$ and a $P^{k}$-equivariant $\frac{1}{k}$-relation $R_n\subseteq K^k\times K_n$. Choose $x_n\in K_n$ such that $(x,x_n)\in R_n$. We say that the sequence $(x_n)_{n\in\mathbb{N}}\in\prod_{n\in\mathbb{N}} T_n$ is an \emph{approximation} of $x$ in the trees $T_n$, relative to the exhaustions determined by $K^k$ and $P^k$. 

\begin{lemma} \label{Gromov}
Let $T$ be an $F_N$-tree, and let $(T_n)_{n\in\mathbb{N}}$ be a sequence of $F_N$-trees that converges to $T$ in the equivariant Gromov-Hausdorff topology. Let $x,y\in T$, let $g\in F_N$, let $M\in\mathbb{R}$. Let $(x_n)_{n\in\mathbb{N}},(y_n)_{n\in\mathbb{N}},(z_n)_{n\in\mathbb{N}}\in\prod_{n\in\mathbb{N}} T_n$ be approximations of $x$, $y$ and $gx$ relative to the same exhaustions. Then
\begin{itemize}
\item the distance $d_{T_n}(x_n,y_n)$ converges to $d_T(x,y)$, and  
\item the distance $d_{T_n}(gx_n,z_n)$ converges to $0$, and
\item if $x\in\mathcal{N}_M(C_T(g))$, then for sufficiently large $n\in\mathbb{N}$, we have $x_n\in\mathcal{N}_{M+1}(C_{T_n}(g))$.
\end{itemize}
\end{lemma}

\begin{proof}
The first two assertions follow from the definition of the equivariant Gromov-Hausdorff topology. To prove the third assertion, one uses the fact that in an $F_N$-tree $T$, we have $d_T(x,gx)=2d_T(x,C_T(g))+||g||_T$ (see \cite[1.3]{CM87}), and the continuity of translation lengths in the equivariant Gromov-Hausdorff topology.
\end{proof}

Let $(T_n)_{n\in\mathbb{N}}$ be a sequence of $F_N$-trees. A sequence $(x_n)_{n\in\mathbb{N}}\in\prod_{n\in\mathbb{N}}T_n$ is \emph{bounded} if for all $g\in F_N$, the distance $d_{T_n}(x_n,gx_n)$ is bounded.

\begin{prop}\label{bounded-sequence}
Let $T$ be a very small $F_N$-tree, and let $(T_n)_{n\in\mathbb{N}}$ be a sequence of $F_N$-trees that converges to $T$ in the equivariant Gromov-Hausdorff topology. A sequence $(x_n)_{n\in\mathbb{N}}\in\prod_{n\in\mathbb{N}}T_n$ is bounded if and only if there exist $x\in T$, exhaustions of $T$ and $F_N$, and an approximation $(x'_n)_{n\in\mathbb{N}}\in\prod_{n\in\mathbb{N}}T_n$ of $x$ relative to these exhaustions, such that $d_{T_n}(x_n,x'_n)$ is bounded.
\end{prop}

In particular, Proposition \ref{bounded-sequence} shows the existence of bounded sequences in any converging sequence of very small $F_N$-trees. Note that its proof is not specific to the case of $F_N$-trees, and only requires the tree $T$ to be \emph{irreducible}, i.e. there exist two hyperbolic isometries in $T$ whose commutator is also hyperbolic in $T$.

\begin{proof}
First assume that there exists $x\in T$, exhaustions of $T$ and $F_N$, and an approximation $(x'_n)_{n\in\mathbb{N}}\in\prod_{n\in\mathbb{N}}T_n$ of $x$ relative to these exhaustions, such that $d_{T_n}(x_n,x'_n)$ is bounded. It follows from the first two assertions of Lemma \ref{Gromov} that for all $g\in F_N$, the distance $d_{T_n}(x'_n,gx'_n)$ is bounded. The triangular inequality, together with the fact that the $F_N$-action on $T_n$ is isometric for all $n\in\mathbb{N}$, implies that $d_{T_n}(x_n,gx_n)$ is bounded.

Conversely, assume that for all $g\in F_N$, the distance $d_{T_n}(x_n,gx_n)$ is bounded. Let $x\in T$, and let $(x'_n)_{n\in\mathbb{N}}\in\prod_{n\in\mathbb{N}}T_n$ be an approximation of $x$ in the trees $T_n$ relative to some exhaustions (without loss of generality, we can assume that for all $g\in F_N$, there exists $k\in\mathbb{N}$ such that $gx\in K^k$). Let $a,b\in F_N$ be such that the commutator $[a,b]$ is hyperbolic in $T$. Using \cite[1.3]{CM87}, we can find $M\in\mathbb{R}$ such that for all $n\in\mathbb{N}$, we have $x_n\in\mathcal{N}_M(C_{T_n}(a))\cap\mathcal{N}_M(C_{T_n}(b))$. Let $M'\in\mathbb{R}$ be such that $x\in\mathcal{N}_{M'}(C_T(a))\cap\mathcal{N}_{M'}(C_T(b))$. For $n\in\mathbb{N}$ sufficiently large, we have $x'_n\in \mathcal{N}_{M'+1}(C_{T_n}(a))\cap\mathcal{N}_{M'+1}(C_{T_n}(b))$ (Lemma \ref{Gromov}). As $[a,b]$ is hyperbolic in $T$, it is also hyperbolic in $T_n$ for all $n$ sufficiently large, and as $||a||_{T_n}$ and $||b||_{T_n}$ are bounded, this implies that the intersection $C_{T_n}(a)\cap C_{T_n}(b)$ has bounded length. By Proposition \ref{intersection-nbd}, both $x_n$ and $x'_n$ lie in a neighborhood of $C_{T_n}(a)\cap C_{T_n}(b)$ (or of any point in the bridge between $C_{T_n}(a)$ and $C_{T_n}(b)$) in $T_n$ of bounded diameter, so $d_{T_n}(x_n,x'_n)$ is bounded. 
\end{proof}

\subsection{Limits of Lipschitz $F_N$-equivariant maps between very small $F_N$-trees}

Given an $\mathbb{R}$-tree $T$, recall that $\overline{T}$ denotes the metric completion of $T$. We aim at showing the following result. 

\begin{theo} \label{limits}
Let $T$ and $T'$ be two very small $F_N$-trees, let $(T_n)_{n\in\mathbb{N}}$ (resp. $(T'_n)_{n\in\mathbb{N}}$) be a sequence of $F_N$-trees converging to $T$ (resp. $T'$) in the equivariant Gromov-Hausdorff topology, and let $(M_n)_{n\in\mathbb{N}}$ be a sequence of real numbers, satisfying $M:=\liminf_{n\to +\infty} M_n<+\infty$. Assume that for all $n\in\mathbb{N}$, there exists an $M_n$-Lipschitz $F_N$-equivariant map $f_n:T_n\to T'_n$. Then there exists an $M$-Lipschitz $F_N$-equivariant map $f:T\to \overline{T'}$.
\end{theo}

Again, Theorem \ref{limits} can be generalized to more general contexts. We only need to require the existence of hyperbolic isometries whose commutator is again hyperbolic in the trees $T$ and $T'$.

\begin{rk} \label{rk-completion}
It is not true in general that we can find an $M$-Lipschitz $F_N$-equivariant map $f:T\to T'$ without passing to the completion, see Example \ref{completion-needed}. However, this is possible in some particular cases, for example if the tree $T$ is simplicial. Indeed, in this case, one can always slightly move the $f$-image of a vertex in $T$ to make it lie in $T'$ without increasing the Lipschitz constant of $f$ (no element of $F_N$ fixes a point in $\overline{T'}\smallsetminus T'$), and tighten $f$ on the edges of $T$ to make the image $f(T)$ entirely lie in $T'$ (which again does not increase the Lipschitz constant of $f$). 
\end{rk}

Our proof of Theorem \ref{limits} uses the theory of ultralimits of metric spaces. Given a nonprincipal ultrafilter $\omega$ on $\mathbb{N}$, we first show that if $(T_n)_{n\in\mathbb{N}}$ is a sequence of very small $F_N$-trees converging to $T$, and $(p_n)_{n\in\mathbb{N}}\in\prod_{n\in\mathbb{N}} T_n$ is a bounded sequence, then the $\omega$-limit of the pointed metric spaces $(T_n,p_n)$ is a complete $F_N$-tree $T_{\omega}$, which contains an isometrically embedded copy of $T$ (hence of $\overline{T}$) as an $F_N$-invariant subtree (Proposition \ref{GHvsUL}). Taking the ultralimit of the maps $f_n$ provides a Lipschitz $F_N$-equivariant map $f_{\omega}:T_{\omega}\to T'_{\omega}$. We get the desired map $f:T\to\overline{T'}$ by precomposing $f_{\omega}$ with the embedding $T\hookrightarrow T_{\omega}$, and postcomposing it with the projection $T'_{\omega}\to \overline{T'}$.
\\
\\
\indent We start by recalling the construction of ultralimits of metric spaces and maps between them. We refer the reader to \cite[Chapter 9]{Kap09} for an introduction to this topic. A \emph{nonprincipal ultrafilter} on the set $\mathbb{N}$ of natural numbers is a map $\omega:2^{\mathbb{N}}\to\{0,1\}$ such that
\begin{itemize}
\item for all $A,B\subseteq\mathbb{N}$, we have $\omega(A\cup B)=\omega(A)+\omega(B)-\omega(A\cap B)$, and
\item we have $\omega(\emptyset)=0$ and $\omega(\mathbb{N})=1$, and
\item for all finite sets $A\subseteq\mathbb{N}$, we have $\omega(A)=0$.
\end{itemize}

\noindent The existence of nonprincipal ultrafilters follows from the axiom of choice. We fix once and for all such a nonprincipal ultrafilter $\omega$ on $\mathbb{N}$. Given a sequence $(x_n)_{n\in\mathbb{N}}\in\mathbb{R}^{\mathbb{N}}$, there exists a unique $x_{\omega}\in\overline{\mathbb{R}}:=\mathbb{R}\cup\{\pm\infty\}$ such that for every neighborhood $U$ of $x_{\omega}$ in $\overline{\mathbb{R}}$, we have $\omega(\{n\in\mathbb{N}|x_n\in U\})=1$. We call $x_{\omega}$ the \emph{$\omega$-limit} of the sequence $(x_n)_{n\in\mathbb{N}}$, and denote it by $\lim_{\omega}x_n$.

Let $((X_n,d_n,p_n))_{n\in\mathbb{N}}$ be a sequence of pointed metric spaces, and let 

\begin{displaymath}
X:=\{(x_n)_{n\in\mathbb{N}}\in\prod_{n\in\mathbb{N}} X_n|\lim_{\omega}d_n(x_n,p_n)<+\infty\}.
\end{displaymath}

\noindent Define a pseudo-metric on $X$ by $d_{\omega}((x_n)_{n\in\mathbb{N}},(y_n)_{n\in\mathbb{N}}):=\lim_{\omega}d_n(x_n,y_n)\in [0,+\infty)$. The \emph{$\omega$-limit} of the pointed metric spaces $(X_n,p_n)$, denoted by $\lim_{\omega}(X_n,p_n)$, is defined to be the Hausdorff quotient of $X$ for this pseudo-metric. It is a well-known fact that the $\omega$-limit of any sequence of pointed metric spaces is complete.

\indent The class of $\mathbb{R}$-trees is closed under taking ultralimits (see \cite[Lemma 4.6]{Sta07}, for instance). Let $(T_n)_{n\in\mathbb{N}}$ be a sequence of $F_N$-trees converging to a very small $F_N$-tree $T$ in the equivariant Gromov-Hausdorff topology, and let $(p_n)_{n\in\mathbb{N}}$ be a bounded sequence. Whenever a sequence $(x_n)_{n\in\mathbb{N}}\in\prod_{n\in\mathbb{N}} T_n$ is such that the distance $d_{T_n}(x_n,p_n)$ is bounded, then the distance $d_{T_n}(gx_n,p_n)\le d_{T_n}(gx_n,gp_n)+d_{T_n}(gp_n,p_n)$ is also bounded. Hence there is a natural isometric $F_N$-action on $T_{\omega}$ defined by $g(x_n)_{n\in\mathbb{N}}=(gx_n)_{n\in\mathbb{N}}$. From now on, whenever an $\mathbb{R}$-tree $T_{\omega}$ is obtained as an ultralimit of a converging sequence of $F_N$-trees (in the equivariant Gromov-Hausdorff topology) with respect to a bounded sequence, we will equip it with the $F_N$-action described above. 

\begin{prop} \label{GHvsUL}
Let $(T_n)_{n\in\mathbb{N}}$ be a sequence of $F_N$-trees, converging in the equivariant Gromov-Hausdorff topology to a very small $F_N$-tree $T$. Let $(p_n)_{n\in\mathbb{N}}\in\prod_{n\in\mathbb{N}} T_n$ be a bounded sequence, and denote by $T_{\omega}$ the $\omega$-limit of $(T_n,p_n)_{n\in\mathbb{N}}$. Then $\overline{T}$ isometrically embeds into $T_{\omega}$ as a closed $F_N$-invariant subtree.
\end{prop} 

\begin{proof}
Using Proposition \ref{bounded-sequence}, we can find an approximation $(p'_n)_{n\in\mathbb{N}}\in\prod_{n\in\mathbb{N}} T_n$ of some point $p\in T$ such that $d_{T_n}(p_n,p'_n)$ is bounded. For $x\in T$, let $(x_n)_{n\in\mathbb{N}}\in\prod_{n\in\mathbb{N}} T_n$ be an approximation of $x$ with respect to the same exhaustions as those used to define the approximation $(p'_n)_{n\in\mathbb{N}}$ of $p$ (we can assume that $x\in K^k$ for all $k\in\mathbb{N}$). By Lemma \ref{Gromov}, the distance $d_{T_n}(x_n,p'_n)$ is bounded. The triangle inequality then implies that $d_{T_n}(x_n,p_n)$ is bounded, so we get a map   

\begin{center}
$\begin{array}{cccc}
\psi:&T&\to&T_{\omega}\\
&x&\mapsto & (x_n)_{n\in\mathbb{N}}
\end{array}$.
\end{center}

The first assertion of Lemma \ref{Gromov} shows the map $\psi$ is an isometric embedding, and the second shows that $\psi$ is $F_N$-equivariant. In particular, the tree $T$ isometrically embeds as an $F_N$-invariant subtree in $T_{\omega}$. The $\mathbb{R}$-tree $T_{\omega}$ is complete, so the completion $\overline{T}$ also isometrically embeds as a (closed) $F_N$-invariant subtree of $T_{\omega}$.
\end{proof}

Let $M\in\mathbb{R}$, and let $(X_n,d_n,p_n)$ and $(X'_n,d'_n,p'_n)$ be two sequences of pointed metric spaces, together with $M$-Lipschitz maps $f_n:X_n\to X'_n$. Assume that for all sequences $(x_n)_{n\in\mathbb{N}}\in\prod_{n\in\mathbb{N}} X_n$ such that $d_n(p_n,x_n)$ is bounded, we have $\lim_{\omega}d'_n(p'_n,f_n(x_n))<+\infty$. Then we can define a map $f_{\omega}:X_{\omega}\to X'_{\omega}$ by setting $f_{\omega}((x_n)_{n\in\mathbb{N}}):=(f_n(x_n))_{n\in\mathbb{N}}$. This applies for example to the case where $p'_n=f_n(p_n)$ for all $n\in\mathbb{N}$. The map $f_{\omega}$ is also $M$-Lipschitz (this applies more generally to the case where the maps $f_n$ are $M_n$-Lipschitz with $\lim_{\omega} M_n=M$). 

\begin{proof}[Proof of Theorem \ref{limits}]
Up to passing to a subsequence, we may assume that the sequence $(M_n)_{n\in\mathbb{N}}$ converges to $M$. Let $(p_n)_{n\in\mathbb{N}}\in\prod_{n\in\mathbb{N}} T_n$ be a bounded sequence, and for all $n\in\mathbb{N}$, let $q_n:=f_n(p_n)$, then $(q_n)_{n\in\mathbb{N}}$ is bounded. Letting $T_{\omega}:=\lim_{\omega}(T_n,p_n)$ and $T'_{\omega}:=\lim_{\omega}(T'_n,q_n)$, we thus get an $F_N$-equivariant $M$-Lipschitz map $f_{\omega}:T_{\omega}\to T'_{\omega}$ by setting $f_{\omega}((x_n)_{n\in\mathbb{N}}):=(f_n(x_n))_{n\in\mathbb{N}}$ for all $(x_n)_{n\in\mathbb{N}}\in T_{\omega}$. By Proposition \ref{GHvsUL}, the tree $T$ (resp. $\overline{T'}$) isometrically embeds in $T_{\omega}$ (resp. $T'_{\omega}$) as an $F_N$-invariant subtree. Denote by $i:T\hookrightarrow T_{\omega}$ the inclusion map (which is obviously $F_N$-equivariant and $1$-Lipschitz), and by $\pi:T'_{\omega}\to \overline{T'}$ the closest point projection, which is also easily seen to be $F_N$-equivariant and $1$-Lipschitz. The map $\pi\circ f_{\omega}\circ i:T\to \overline{T'}$ is the desired $M$-Lipschitz $F_N$-equivariant map from $T$ to $\overline{T'}$.
\end{proof}

\section{The case of trees with dense orbits} \label{sec-dense-case}

In this section, we prove our two main results (Theorems \ref{intro-equivalences} and \ref{intro-White}) in the case of trees with dense orbits.

\subsection{An easy inequality in the extension of White's theorem}

Given $T,T'\in\overline{cv_N}$, we define $\text{Lip}(T,T')$ to be the infimum of a Lipschitz constant of an $F_N$-equivariant map $f:T\to \overline{T'}$ if such a map exists, and $\text{Lip}(T,T')=+\infty$ otherwise. We define ${\Lambda}(T,T'):=\sup_{g\in F_N}\frac{||g||_{T'}}{||g||_T}$ (where we take the conventions $\frac{0}{0}=0$ and $\frac{1}{0}=+\infty$). Given a subset $\mathcal{C}\subseteq F_N$, we define ${\Lambda}_{\mathcal{C}}(T,T'):=\sup_{g\in\mathcal{C}}\frac{||g||_{T'}}{||g||_T}$ (in particular, we have ${\Lambda}_{F_N}(T,T')={\Lambda}(T,T')$). Given a map $f$ between $\mathbb{R}$-trees, we denote by $\text{Lip}(f)$ the Lipschitz constant of $f$. We start by recalling the proof of the following inequality, which shows in particular that strong domination implies weak domination.

\begin{prop} \label{easy-comparison}
For all $T,T'\in\overline{cv_N}$, we have ${\Lambda}(T,T')\le\text{Lip}(T,T')$. In particular, for all $\mathcal{C}\subseteq F_N$, we have ${\Lambda}_{\mathcal{C}}(T,T')\le\text{Lip}(T,T')$.
\end{prop}

\begin{proof}
Let $T, T'\in\overline{cv_N}$. Assume that $\text{Lip}(T,T')<+\infty$ (otherwise the claim is obvious), and let $f:T\to \overline{T'}$ be a Lipschitz $F_N$-equivariant map. Let $g\in F_N$, and let $x\in C_T(g)$. Then

\begin{displaymath}
\begin{array}{rl}
||g||_{T'}&\le d_{\overline{T'}}(f(x),gf(x))\\
&\le \text{Lip}(f) d_T(x,gx)\\
&= \text{Lip}(f) ||g||_T,
\end{array}
\end{displaymath}
 
\noindent so for all $g\in F_N$, we have $\frac{||g||_{T'}}{||g||_T}\le \text{Lip}(f)$. The claim follows.
\end{proof}

\subsection{Extending White's theorem to trees with dense orbits}

Given $\mathcal{C}\subseteq F_N$, we say that a tree $T\in\overline{cv_N}$ \emph{satisfies White's theorem relatively to $\mathcal{C}$} if for all $T'\in\overline{cv_N}$, we have $\text{Lip}(T,T')={\Lambda}_{\mathcal{C}}(T,T')$.

\begin{prop}\label{blueprint-1}
Let $\mathcal{C}\subseteq F_N$, and $T\in\overline{cv_N}$. If $T$ admits a Lipschitz approximation by a sequence of trees $T_n\in\overline{cv_N}$ which all satisfy White's theorem relatively to $\mathcal{C}$, then $T$ satisfies White's theorem relatively to $\mathcal{C}$.
\end{prop}

\begin{proof}
Proposition \ref{easy-comparison} shows that ${\Lambda}_{\mathcal{C}}(T,T')\le\text{Lip}(T,T')$, and if ${\Lambda}_{\mathcal{C}}(T,T')=+\infty$, then the reverse inequality is obvious. We can thus assume that ${\Lambda}_{\mathcal{C}}(T,T')<+\infty$. As the trees $T_n$ satisfy White's theorem relatively to $\mathcal{C}$, for all $n\in\mathbb{N}$, we have 
\begin{displaymath}
\begin{array}{rl}
\text{Lip}(T_n,T')&= {\Lambda}_{\mathcal{C}}(T_n,T')\\
&\le {\Lambda}_{\mathcal{C}}(T_n,T){\Lambda}_{\mathcal{C}}(T,T')\\
&\le {\Lambda}_{\mathcal{C}}(T,T'),
\end{array}
\end{displaymath}

\noindent since it follows from Proposition \ref{easy-comparison} and the definition of a Lipschitz approximation that $\Lambda_{\mathcal{C}}(T_n,T)\le 1$. As ${\Lambda}_{\mathcal{C}}(T,T')<+\infty$, Theorem \ref{limits} thus shows the existence of a ${\Lambda}_{\mathcal{C}}(T,T')$-Lipschitz $F_N$-equivariant map from $T$ to $\overline{T'}$, hence $\text{Lip}(T,T')\le{\Lambda}_{\mathcal{C}}(T,T')$.
\end{proof}

\begin{cor}\label{White-dense}
Let $T,T'\in\overline{cv_N}$, and assume that $T$ has dense orbits. Then 

\begin{displaymath}
\text{Lip}(T,T')=\Lambda(T,T')={\Lambda}_{\mathcal{P}_N}(T,T').
\end{displaymath}
\end{cor}

\begin{proof}
Theorem \ref{approximation-by-simplicial} and Lemma \ref{dense-arcs} show that $T$ admits a Lipschitz approximation by trees in $cv_N$, and Theorem \ref{Algom-Kfir} shows that trees in $cv_N$ satisfy White's theorem relatively to $\mathcal{P}_N$ (this actually only uses the particular case of Theorem \ref{Algom-Kfir} where the simplicial tree belongs to $CV_N$). Corollary \ref{White-dense} thus follows from Proposition \ref{blueprint-1}.
\end{proof}

\subsection{Simple-equivalent trees with dense orbits are equal.}

\begin{prop}\label{blueprint-2}
Let $T,T'\in\overline{cv_N}$ be two trees that both satisfy White's theorem relatively to $\mathcal{P}_N$. If $T$ and $T'$ are simple-equivalent, then $T=T'$.
\end{prop}

\begin{proof}
The hypotheses ensure the existence of $1$-Lipschitz $F_N$-equivariant maps from $T$ to $\overline{T'}$ and from $T'$ to $\overline{T}$, so Proposition \ref{easy-comparison} implies that $||g||_T=||g||_{T'}$ for all $g\in F_N$. Theorem \ref{rigidity} thus implies that $T=T'$.
\end{proof}

From Corollary \ref{White-dense} and Proposition \ref{blueprint-2}, we deduce the following corollary.

\begin{cor}\label{rigidity-for-dense}
Let $T,T'\in\overline{cv_N}$ be two trees with dense orbits. If $T$ and $T'$ are simple-equivalent, then $T=T'$. \qed
\end{cor}

\subsection{Computing stretching factors between trees with dense orbits}

We give a formula for $\text{Lip}(T,T')$ for trees $T,T'\in\overline{cv_N}$ having dense orbits, in terms of length measures on $T$ and $T'$. This notion was introduced by Paulin in \cite{Pau95}, and developed by Guirardel in \cite[Section 5]{Gui00}. An \emph{invariant length measure} $\mu$ on $T$ is a collection of finite Borel measures $\mu_I$ for all segments $I\subseteq T$ such that
\begin{itemize}
\item for all segments $J\subseteq I$, we have $\mu_J=(\mu_{I})_{|J}$, and 
\item for all segments $I\subseteq T$ and all $g\in F_N$, we have $\mu_{gI}=(g_{|I})_{\ast}\mu_I$.
\end{itemize}

Given a segment $I\subseteq T$, we will simply write $\mu(I)$ to denote $\mu_I(I)$. We denote by $\mu^T$ the \emph{Lebesgue measure} on $T$ given by $\mu^T([x,y]):=d_T(x,y)$ for all $x,y\in T$. A subset $E\subseteq T$ is \emph{measurable} if each intersection of $E$ with an arc of $T$ is measurable. A measurable subset $E$ has \emph{$\mu$-measure $0$} if for every arc $I\subseteq T$, we have $\mu_I(E\cap I)=0$. It has \emph{full $\mu$-measure} if $T\smallsetminus E$ has $\mu$-measure $0$. A measure $\mu$ on an $F_N$-tree $T$ is \emph{ergodic} if every $F_N$-invariant measurable subset of $T$ has either zero or full $\mu$-measure. We denote by $\mathcal{M}_0(T)$ the space of nonatomic measures on $T$. The following theorem, due to Guirardel, states that any tree in $\overline{cv_N}$ with dense orbits is finite-dimensional from the measure-theoretic viewpoint.

\begin{theo} \label{decomposition-measure} (Guirardel \cite[Corollary 5.4]{Gui00})
For all very small $F_N$-trees $T$ with dense orbits, the set $\mathcal{M}_0(T)$ is a finite-dimensional convex set. Furthermore, the tree $T$ has at most $3N-4$ nonatomic ergodic measures up to homothety, and every measure in $\mathcal{M}_0(T)$ is a sum of these ergodic measures.
\end{theo}

Let $T,T'\in\overline{cv_N}$ be two trees with dense orbits. Denote by $k$ the dimension of $\mathcal{M}_0(T)$, and let  $\{\mu_i\}_{i=1,\dots,k}$ be a collection of ergodic measures on $T$ given by Theorem \ref{decomposition-measure}. The measures $\mu_i$ are pairwise mutually singular, and there exist disjoint measurable sets $E_1,\dots,E_k$ that cover $T$ such that for all $i\in\{1,\dots,k\}$, the set $E_i$ has full $\mu_i$-measure. As the Lebesgue measure $\mu^T$ is nonatomic, it decomposes as $\mu^T=\sum_{i=1}^{k}\lambda_i\mu_i$, with $\lambda_i\ge 0$ for all $i\in\{1,\dots,k\}$. The measures $\mu_i$ for which $\lambda_i>0$ are absolutely continuous with respect to $\mu^T$. In particular, they are \emph{regular}, i.e. for all Borel subsets $X\subseteq T$ and all $\epsilon>0$, there exists an open subset $U\subseteq T$ containing $X$ such that for all segments $I\subseteq T$, we have $\mu_i(X\cap I)\ge\mu_i(U\cap I)-\epsilon$. If there exists a Lipschitz $F_N$-equivariant, alignment-preserving map $f:T\to T'$, then the measure $\mu^{f}$ defined on $T$ by $\mu^f([x,y]):=d_{T'}(f(x),f(y))$ is absolutely continuous with respect to $\mu^T$. In particular, it decomposes as $\mu^{f}=\sum_{i=1}^{k}\lambda'_i\mu_i$, where $\lambda'_i\ge 0$ for all $i\in\{1,\dots,k\}$ and $\lambda'_i=0$ whenever $\lambda_i= 0$. It follows from Proposition \ref{dense-rigidity} that there exists at most one Lipschitz, $F_N$-equivariant, alignment-preserving map $f:T\to T'$. If such a map exists, and if $\mu^T$ denotes the Lebesgue measure on $T$, we denote by $\mu^{T\to T'}$ the measure $\mu^f$ on $T$. 

\begin{prop} \label{stretching-dense}
Let $T,T'\in\overline{cv_N}$ be two $F_N$-trees with dense orbits. Then $\text{Lip}(T,T')<+\infty$ if and only if there exists a Lipschitz $F_N$-equivariant, alignment-preserving map from $T$ to $T'$. In this case, denote by $k$ the dimension of $\mathcal{M}_0(T)$, let $\mu^T=\sum_{i=1}^k \lambda_i \mu_i$ be the decomposition of the Lebesgue measure on $T$ on its ergodic components, and let $\mu^{T\to T'}=\sum_{i=1}^k \lambda'_i \mu_i$ be the decomposition of $\mu^{T\to T'}$. Then

\begin{displaymath}
\text{Lip}(T,T')=\max_{i\in\{1,\dots,k\}}\frac{\lambda'_i}{\lambda_i}.
\end{displaymath} 
\end{prop}

\begin{proof}
If $\text{Lip}(T,T')<+\infty$, then there exists a Lipschitz $F_N$-equivariant map $f:T\to\overline{T'}$, which is unique by Proposition \ref{dense-rigidity}, and preserves alignment by Proposition \ref{alignment-preserving}. This implies that $f(T)\subseteq T'$, otherwise we would find $x\in T$ with $f(x)\in\overline{T'}\smallsetminus T'$, and some component of $T\smallsetminus\{x\}$ would be entirely mapped to $f(x)$. However, as $T$ has dense orbits, this would imply that $f(T)\subseteq \overline{T'}\smallsetminus T'$, a contradiction. Let $C:=\max_{i\in\{1,\dots,k\}}\frac{\lambda'_i}{\lambda_i}$. By definition of $\mu^{T\to T'}$, for all $x,y\in T$, we have 

\begin{displaymath}
\begin{array}{rl}
d_{{T'}}(f(x),f(y))&=\mu^{T\to T'}([x,y])\\
&=\sum_{i=1}^k \lambda'_i \mu_i([x,y])\\
&\le C\mu^T([x,y])\\
&=Cd_T(x,y),
\end{array}
\end{displaymath}

\noindent so $\text{Lip}(f)\le C$. Let $i\in\{1,\dots,k\}$ be such that $\lambda_i>0$, and let $I\subseteq T$ be an arc such that $\mu_i(I)>0$. We denote by $\mu_i^c$ the measure $\mu^T-\lambda_i\mu_i$. The measures $\mu_1,\dots,\mu_k$ are pairwise mutually singular, so there exists a Borel subset $X$ of $I$ such that $\mu_i(X)=\mu_i(I)$ and $\mu_j(X)=0$ for all $j\neq i$. As $\mu_i$ is regular, for all $\epsilon>0$, there exists an open set $U\subseteq I$ that contains $X$, such that $\mu_i^c(U)<\epsilon\mu_i(I)$. Since $U\subseteq I$ is open, it is the disjoint union of a countable collection of open intervals. At least one of these intervals $I'$ must satisfy $\mu_i^c(I')<\epsilon\mu_i(I')$. By definition of $\mu^{T\to T'}$, we have 

\begin{displaymath}
\begin{array}{rl}
\text{Lip}(f)&\ge \frac{\mu^{T\to T'}(I')}{\mu^T(I')}\\
&=\frac{\lambda'_i\mu_i(I')+\sum_{j\neq i}\lambda'_j\mu_j(I')}{\lambda_i\mu_i(I')+\mu_i^c(I')}\\
&\ge\frac{\lambda'_i\mu_i(I')}{\lambda_i\mu_i(I')+\mu_i^c(I')}\\
&\ge \frac{\lambda'_i}{\lambda_i+\epsilon}.
\end{array}
\end{displaymath}

\noindent By choosing $\epsilon>0$ arbitrarily small, we thus get that $\frac{\lambda'_i}{\lambda_i}\le \text{Lip}(f)$. This holds for all $i\in\{1,\dots,k\}$ for which $\lambda_i>0$, and in addition we have $\lambda'_i=0$ whenever $\lambda_i=0$. This shows that $\text{Lip}(f)\ge C$, and hence $\text{Lip}(f)=\text{Lip}(T,T')=C$.
\end{proof}

\section{Generalizing White's theorem to arbitrary trees in the boundary of outer space} \label{sec-White}

We now generalize Theorems \ref{White} and \ref{Algom-Kfir} to arbitrary actions in the boundary of outer space. This answers a question by Algom-Kfir \cite[Question 4.6]{AK12}.

\begin{theo} \label{strong-White}
For all $T,T'\in\overline{cv_N}$, we have $\text{Lip}(T,T')={\Lambda}(T,T')$.
\end{theo}

The proof of Theorem \ref{strong-White} will be carried out in Sections \ref{sec-optimal} to \ref{sec-weak-strong}. White's theorem for trees in $cv_N$ is actually a bit stronger, as it provides a finite set (depending on $T$ but not on $T'$) of (conjugacy classes of) primitive elements of $F_N$ called candidates, represented by loops in $T/F_N$ having a particular shape, on which the supremum in the definition of ${\Lambda}(T,T')$ is achieved for all $T'\in cv_N$. In particular, this gives an explicit procedure for computing $\text{Lip}(T,T')$ for all $T,T'\in cv_N$. In Section \ref{sec-candidates}, we will give a generalization of the notion of candidates in a tree $T\in\overline{cv_N}$, and show that for all $T'\in\overline{cv_N}$, the supremum in the definition of ${\Lambda}(T,T')$ can be taken over the set of candidates in $T$ (Theorem \ref{White-candidates}). In Section \ref{sec-separable-candidates}, we show that in the case of good trees, this supremum can be taken over the set of primitive elements of $F_N$. This will turn out to be a crucial tool for tackling the problem of spectral rigidity of the set of primitive elements of $F_N$ in $\overline{cv_N}$.

\subsection{A few examples}

When $T,T'\in cv_N$, both the infimum in the definition of $\text{Lip}(T,T')$ and the supremum in the definition of ${\Lambda}(T,T')$ are achieved. This remains true more generally when $T$ is simplicial, and in this case we can replace $\overline{T'}$ by $T'$ in the definition of $\text{Lip}(T,T')$, see \cite[Proposition 4.5]{AK12} and Remark \ref{rk-completion} of the present paper. When $T,T'\in\overline{cv_N}$ are arbitrary trees, the infimum in the definition of $\text{Lip}(T,T')$ is still realized as long as there exists a Lipschitz $F_N$-equivariant map $f:T\to \overline{T'}$ (Proposition \ref{optimal}). However the supremum in the definition of $\Lambda(T,T')$ may not be realized (even if it is finite), as shown in the following example.

\begin{ex} \label{sup-not-realized} (see Figure \ref{fig-not-realized}).
We provide an example of a pair of trees $T_1,T_2\in\overline{cv_N}$ for which the supremum in the definition of ${\Lambda}(T_1,T_2)$ is not achieved. For all $i\in\{1,2\}$, let $T_i$ be the tree associated to a graph of actions with a single edge of length $i$ having trivial stabilizer, and two vertices, one having cyclic stabilizer generated by an element $t\in F_N$, and the other being a nontrivial $G$-tree $T_0$ with dense orbits (where $G$ is a complementary free factor in $F_N$ of the cyclic group generated by $t$), whose attaching point $p$ is not fixed by any element of $F_N$ (the existence of such a point $p$ follows from \cite{Jia91}). 

\begin{figure} 
\begin{center}
\input{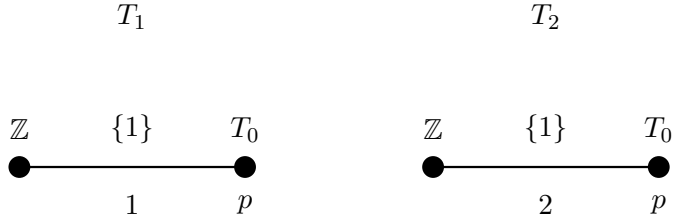}
\caption{The trees $T_i$ in Example \ref{sup-not-realized}.}
\label{fig-not-realized}
\end{center}
\end{figure}

There is an obvious $2$-Lipschitz $F_N$-equivariant map from $T_1$ to $T_2$ which stretches the edges in the simplicial part of $T$ by a factor of $2$, hence $\text{Lip}(T_1,T_2)\le 2$. As $T_0$ has dense orbits, for all $\epsilon>0$, there exists $g\in G$ such that $d_{T_0}(p,gp)<\epsilon$. So for all $i\in\{1,2\}$, we have $2i<||tg||_{T_i}<2i+\epsilon$, hence $\frac{||tg||_{T_2}}{||tg||_{T_1}}\ge\frac{4}{2+\epsilon}$, which becomes arbitrary close to $2$ as $\epsilon$ goes to $0$. So ${\Lambda}(T_1,T_2)\ge 2$, and hence by Proposition \ref{easy-comparison} we have $\text{Lip}(T_1,T_2)={\Lambda}(T_1,T_2)=2$. However, any element $g\in F_N$ either belongs to a conjugate of $G$, or of the cyclic group generated by $t$ (in which case $||g||_{T_1}=||g||_{T_2}$), or is conjugated to an element represented by a reduced word of the form $t^{\alpha_1}g_1t^{\alpha_2}\dots t^{\alpha_k}g_k$, with $\alpha_i\neq 0$ and $g_i\in G\smallsetminus\{e\}$ for all $i\in\{1,\dots,k\}$. In this last case, we have $l_i:=d_{T_0}(p,g_ip)>0$ for all $i\in\{1,\dots,k\}$, because $p$ is not fixed by any element of $G$, and

\begin{displaymath}
||g||_{T_1}=2k+\sum_{i=1}^kl_i,
\end{displaymath}

\noindent and similarly 

\begin{displaymath}
||g||_{T_2}=4k+\sum_{i=1}^kl_i<2||g||_{T_1}.
\end{displaymath} 

\noindent So no element in $F_N$ is stretched exactly by an amount of $2$ from $T_1$ to $T_2$. 
\end{ex}

\begin{ex} \label{completion-needed} (see Figure \ref{fig-completion-needed}).
We give an example of a pair of trees $T,T'\in\overline{cv_N}$ for which $\overline{T'}$ cannot be replaced by $T'$ in the definition of $\text{Lip}(T,T')$. More precisely, we give an example of a pair of trees $T,T'\in\overline{cv_N}$ for which there exists an $F_N$-equivariant Lipschitz map from $T$ to $\overline{T'}$, but no such map from $T$ to $T'$. 

Let $T_0$ be a minimal nontrivial $F_{N-1}$-tree with dense orbits, then $T_0$ is strictly contained in its metric completion (see \cite[Example II.6]{GL95}). Let $p\in\overline{T_0}\smallsetminus T_0$, and let $T$ be the tree associated to a graph of actions having
\begin{itemize}
\item two vertices $v_1$ and $v_2$, with $v_1$ having nontrivial cyclic vertex group, and $v_2$ having $\overline{T_0}$ as its vertex tree, and 
\item one single edge $e$ of length $2$ with trivial edge group, whose origin is $v_1$, and whose terminal vertex is $v_2$, with attaching point $p$.
\end{itemize}
\noindent Let $p'\in T_0$ be such that $d_{T}(p,p')=1$, and let $T'$ be the $F_N$-tree obtained by equivariantly folding half of the edge $e$ along the segment $[p,p']$, and passing to a minimal subtree.

\begin{figure}
\begin{center}
\input{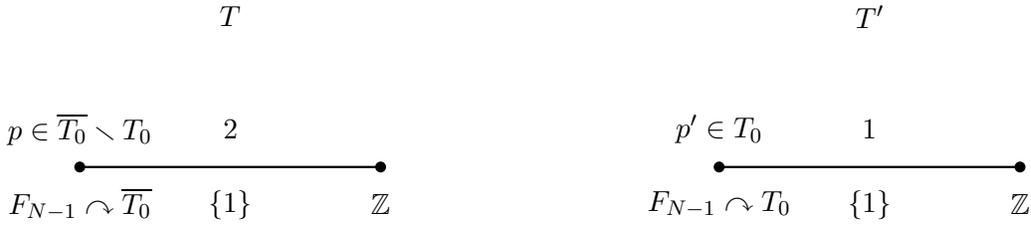}
\caption{The trees $T$ and $T'$ in Example \ref{completion-needed}.}
\label{fig-completion-needed}
\end{center}
\end{figure}

The definition of $T'$ provides a $1$-Lipschitz $F_N$-equivariant morphism from $T$ to $\overline{T'}$. However, we claim that for all $M\in\mathbb{R}$, there is no $M$-Lipschitz $F_N$-equivariant map from $T$ to $T'$. Indeed, suppose $f:T\to T'$ is $F_N$-equivariant and Lipschitz. Then $f(T_0)$ is an $F_{N-1}$-tree with dense orbits contained in $T'$, so $f(T_0)\subseteq T_0$, and by Proposition \ref{dense-rigidity}, the map $f$ restricts to the identity on $T_0$. As $p$ is the limit of a sequence $(p_n)_{n\in\mathbb{N}}$ of elements in $T_0$, its image $f(p)$ should be the limit of $f(p_n)=p_n$ in $T'$. However, the sequence $(p_n)_{n\in\mathbb{N}}$ does not converge in $T'$.
\end{ex}

\subsection{Optimal maps and legal turns}\label{sec-optimal}

Let $T,T'\in\overline{cv_N}$. A map $f:T\to\overline{T'}$ is \emph{piecewise-linear} if it is Lipschitz, and linear in restriction to the edges in the simplicial part of $T$. Let $X$ denote the underlying graph of the canonical decomposition of $T$ as a graph of actions with vertex trees having dense orbits (Proposition \ref{graph-of-actions}). The length of a segment $\gamma\subset T$ is denoted by $l_T(\gamma)$, and similarly the length of a path $\gamma$ in $X$ is denoted by $l_{X}(\gamma)$. We define $T_f$ to be the (possibly empty) subset of $T$ consisting of the edges $e$ in the simplicial part of $T$ for which $\frac{l_{\overline{T'}}(f(e))}{l_T(e)}=\text{Lip}(f)$. We denote by $T_f^{simpl}$ the projection of $T_f$ to $T^{simpl}$, and by $X_f$ its projection to $X$. An \emph{optimal map} $f:T\to \overline{T'}$ is an $F_N$-equivariant, piecewise-linear map such that $\text{Lip}(f)=\text{Lip}(T,T')$, and $X_f$ is minimal for the inclusion among all $F_N$-equivariant, piecewise-linear $\text{Lip}(f)$-Lipschitz maps. Note that in the case where $T$ has dense orbits, this last condition is empty.

\begin{prop} \label{optimal}
Let $T,T'\in \overline{cv_N}$. If $\text{Lip}(T,T')<+\infty$, then there exists an optimal map $f:T\to \overline{T'}$.
\end{prop}

\begin{rk}
Again, this extends to more general contexts than $F_N$-actions. For example, this is true if $T$ and $T'$ are two trees belonging to an irreducible deformation space, see \cite[Theorem 20]{Mei13}.
\end{rk}

\begin{proof}
Applying Theorem \ref{limits} to a sequence of $F_N$-equivariant maps $f_n:T\to\overline{T'}$ with $\text{Lip}(f_n)\le\text{Lip}(T,T')+\frac{1}{n}$ gives the existence of a $\text{Lip}(T,T')$-Lipschitz $F_N$-equivariant map $f:T\to \overline{T'}$. Tightening $f$ on the edges in the simplicial part of $T$ cannot increase its Lipschitz constant, hence we may choose $f:T\to \overline{T'}$ to be piecewise-linear. As $X$ is a finite graph, we can also choose $f$ so that $X_f$ is minimal. 
\end{proof}

We fix once and for all two trees $T,T'\in\overline{cv_N}$ such that $\text{Lip}(T,T')<+\infty$, together with an optimal map $f:T\to\overline{T'}$. A \emph{turn} at a vertex $v$ of $T^{simpl}$ is a pair $(e=[a,b],e'=[c,d])$ of distinct edges in the simplicial part of $T$ such that $[b,c]$ projects to $v$ (in other terms, the projections of $e$ and $e'$ to $T^{simpl}$ share a common vertex). Let $v$ be a vertex of $T_f^{simpl}$ such that $T_v$ is reduced to a point. A turn $(e,e')$ at $v$ is \emph{legal} for $f$ if $e,e'\subseteq T_f$, and $f(e)\cap f(e')=\{f(b)\}$, and \emph{illegal} otherwise. It is \emph{legal up to $G_v$} for $f$ if there exists $g\in G_v$ such that $(e,ge')$ is legal for $f$. The following proposition, already used by Algom-Kfir in her proof of \cite[Proposition 4.5]{AK12}, gives control over legal turns at a vertex $v$ of $T_f^{simpl}$ for which $T_v$ is reduced to a point. We provide a proof for completeness.

\begin{prop}\label{gates-1}
Let $f:T\to \overline{T'}$ be an optimal map, and assume that $T_f\neq\emptyset$. Let $v$ be a vertex in $T_f^{simpl}$ such that $T_v$ is reduced to a point. Then there exists a turn at $v$ which is legal for $f$, and if $G_v$ has rank at least $2$, then all turns at $v$ are legal up to $G_v$ for $f$. In addition, for all edges $e,e',e''$ in $T_f$ adjacent to $v$, if $(e,e')$ and $(e',e'')$ are both illegal for $f$, then $(e,e'')$ is also illegal for $f$.
\end{prop}

\begin{proof}
If the $f$-images of all edges in $T_f$ adjacent to $v$ have a common initial germ, and have a common initial segment with their $g$-translate for all $g\in G_v$, then all $f$-images of edges in $T_f$ adjacent to $v$ in a single $F_N$-orbit have a common initial segment, which is an arc fixed by $G_v$ in $\overline{T'}$. As there are finitely many such orbits, this implies that all $f$-images of edges in $T_f$ adjacent to $v$ have a common initial segment, which is an arc fixed by $G_v$ in $\overline{T'}$. One can then slightly homotope $f$ to either decrease $\text{Lip}(f)$ or $X_f$, contradicting optimality of $f$ (see the proof of \cite[Proposition 3.15]{FM11} or \cite[Proposition 2.3]{AK11} for details). So we can find a turn at $v$ which is legal for $f$. Assume in addition that $G_v$ has rank at least $2$. For all edges $e,e'$ in $T_f$ adjacent to $v$ (possibly with $e=e'$), the $f$-images of $e$ and $e'$ are not reduced to points. As $T'$ is very small, the subgroup $G'_v$ of elements $g\in G_v$ such that $f(e)$ and $gf(e')$ share a nondegenerate initial segment is at most cyclic, and for all $g\in G_v\smallsetminus G'_v$, the turn $(e,ge')$ is legal for $f$. The assertion stating that illegality at $v$ is a transitive relation follows from the definition of illegal turns. 
\end{proof}

One has to be slightly more careful when defining legality of turns at vertices $v$ of $T_f^{simpl}$ for which $T_v$ is not reduced to a point. Let $v$ be such a vertex. For $\epsilon> 0$, a turn $(e=[a,b],e'=[c,d])$ at $v$ is said to be \emph{$\epsilon$-legal} for $f$ if $e,e'\subseteq T_f$, and $d_T(b,c)<\epsilon$, and $l_{\overline{T'}}(f(e)\cap f(e'))<\epsilon$. It is \emph{legal up to $G_v$} for $f$ if for all $\epsilon>0$, there exists $g\in G_v$ such that $(e,ge')$ is $\epsilon$-legal for $f$. We aim at giving an analogue of Proposition \ref{gates-1} in this situation. The following lemma, illustrated in Figure \ref{fig-overlap}, will turn out to be useful.

\begin{lemma}\label{overlap}
Let $T$ be an $\mathbb{R}$-tree, let $l,\epsilon\in\mathbb{R}$ with $\epsilon<\frac{l}{10}$, and let $a,a',a'',b,b',b''\in T$. Assume that $l_T([a,b]\cap [a',b']), l_T([a,b]\cap [a'',b''])\ge l$ and $l_T([a,a']),l_T([a,a''])\le\epsilon$. Then $l_T([a',b']\cap [a'',b''])\ge l-\epsilon$.
\end{lemma}

\begin{proof}
As $\epsilon<\frac{l}{10}$, one can check that the tripods $\{a,a',a''\}$ and $\{b,b',b''\}$ do not intersect. One then argue depending on whether the intersection of the bridge between them with the tripod $\{a,a',a''\}$ (resp. $\{b,b',b''\}$) is in the direction of $a$, $a'$ or $a''$ (resp. $b$, $b'$ or $b''$). The various possibilities are displayed on Figure \ref{fig-overlap}.
\end{proof}

\begin{figure}
\begin{center}
\input{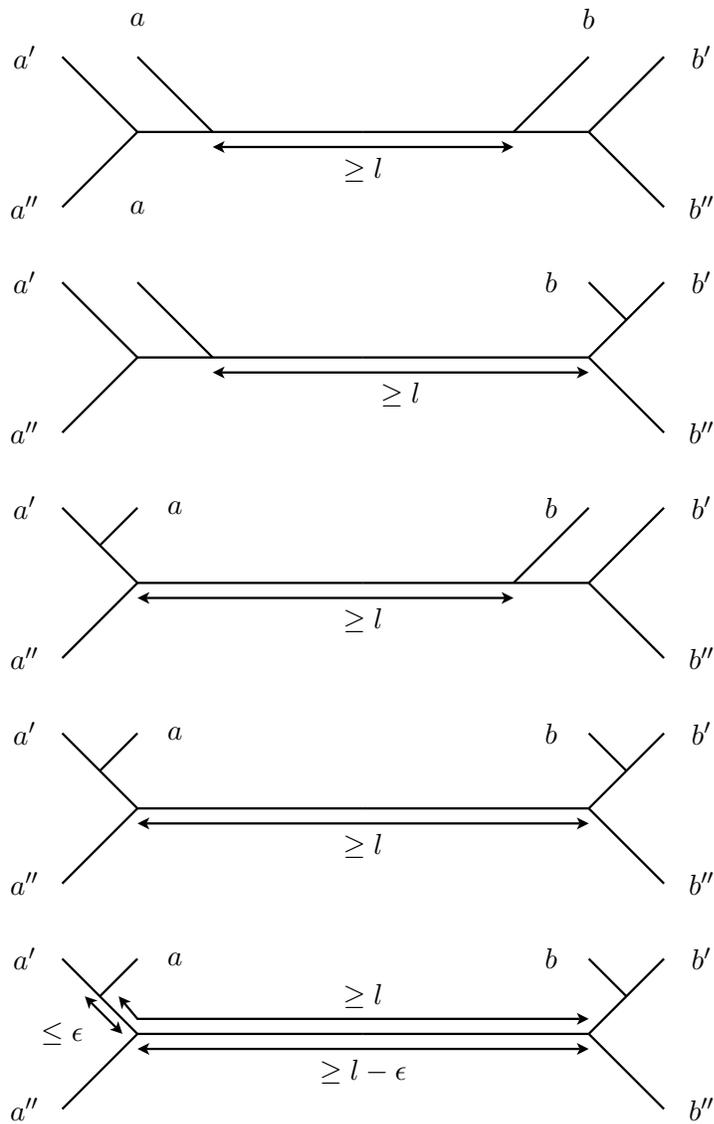}
\caption{The situation in Lemma \ref{overlap}.}
\label{fig-overlap}
\end{center}
\end{figure}

\begin{prop} \label{gates-2}
Let $f:T\to \overline{T'}$ be an optimal map, and assume that $T_f\neq\emptyset$. Let $v$ be a vertex in $T_f^{simpl}$ such that $T_v$ is not reduced to a point. Then all turns $(e,e')$ at $v$ with $e,e'\subseteq T_f$ are legal up to $G_v$ for $f$.
\end{prop}

\begin{proof}
We denote by $b$ (resp. $c$) the attaching point of $e$ (resp. $e'$) to $T_v$. Assume towards a contradiction that there exists $\epsilon>0$ such that for all $g\in G_v$, the turn $(e,ge')$ is not $\epsilon$-legal. Let $\epsilon':=\frac{\epsilon}{100\max\{M,1\}}$, where $M:=\text{Lip}(f)$. Let $g\in G_v$ be a hyperbolic element in $T_v$ such that $d_T(b,gb)<\epsilon'$. As $T_v$ has dense orbits, there exists $g_0\in G_v$ such that $d_T(b,g_0c)<\epsilon'$, and hyperbolic elements $g_1,g_2\in G_v$ which generate a rank $2$ subgroup of $G_v$, such that for all $i\in\{1,2\}$, we have $d_T(c,g_ic)<\epsilon'$ (in particular $||g_i||_{T}<\epsilon'$ and $||g_i||_{T'}<\frac{\epsilon}{100}$). By the triangle inequality and the fact that the $F_N$-action is isometric, we also have $d_T(b,g_0g_ic)<2\epsilon'$ for all $i\in\{1,2\}$. The hypothesis thus implies that $l_{\overline{T'}}(f(e)\cap g_0f(e'))\ge\epsilon$ and $l_{\overline{T'}}(f(e)\cap g_0g_if(e'))\ge\epsilon$ for all $i\in\{1,2\}$. In addition, as $f$ is $M$-Lipschitz, we have $d_{\overline{T'}}(f(b),f(g_0c))<\frac{\epsilon}{100}$, and $d_{\overline{T'}}(f(b), f(g_0g_ic))<\frac{\epsilon}{50}$, so Lemma \ref{overlap} implies that the segments $f(e'), f(g_1e')$ and $f(g_2e')$ pairwise intersect along a subsegment of length greater than $\frac{49\epsilon}{50}$. However, as $||g_i||_{T'}<\frac{\epsilon}{100}$ for all $i\in\{1,2\}$, this implies that the axes of $g_1$ and $g_2$ in $T'$ have a nontrivial overlap, of length greater than $||g_1||_{T'}+||g_2||_{T'}$. Hence some nontrivial element in the rank $2$ subgroup generated by $g_1$ and $g_2$ fixes a nondegenerate subsegment of these axes, contradicting Lemma \ref{dense-arcs}. 
\end{proof}

\subsection{Case where $\text{Lip}(T,T')<+\infty$} \label{sec-strong-weak}

We first prove Theorem \ref{strong-White} in the case where $\text{Lip}(T,T')<+\infty$.

\begin{prop} \label{strong-implies-weak}
Let $T,T'\in\overline{cv_N}$. If $\text{Lip}(T,T')<+\infty$, then $\text{Lip}(T,T')\le{\Lambda}(T,T')$.
\end{prop} 

Let $T\in\overline{cv_N}$ and $g\in F_N$. The \emph{combinatorial length} of $g$ in $T$, denoted by $l^{comb}_T(g)$, is defined as the length of $g$ in the simplicial tree obtained from $T^{simpl}$ by making all edge lengths equal to $1$. As there are finitely many orbits of branch points in $T$ by \cite[Corollary III.3]{GL95}, the number of orbits of edges in $T^{simpl}$ is finite. An element $g\in F_N$ is \emph{$\epsilon$-legal} for $f$ if 
\begin{itemize}
\item its axis $C_T(g)$ crosses an edge in the simplicial part of $T$, and 
\item whenever $C_T(g)$ crosses a turn at a vertex $v$ of $T^{simpl}$ whose corresponding vertex tree in $T$ is reduced to a point, then this turn is legal, and
\item whenever $C_T(g)$ crosses a turn at a vertex $v$ of $T^{simpl}$ whose corresponding vertex tree in $T$ is not reduced to a point, then this turn is $\epsilon$-legal.
\end{itemize}

\begin{lemma} \label{legal-loop}
Let $T,T'\in\overline{cv_N}$ be such that $\text{Lip}(T,T')<+\infty$, and let $f:T\to\overline{T'}$ be an optimal map. If $T_f\neq\emptyset$, then there exists $K>0$ such that for all $\epsilon>0$, there exists an element $g\in F_N$ with $l^{comb}_T(g)\le K$, which is $\epsilon$-legal for $f$.
\end{lemma}

\begin{proof}
Let $K$ be the (finite) number of orbits of oriented edges in $T^{simpl}$, let $\epsilon>0$, and let $x\in T_f$. Starting from $x$ and using Propositions \ref{gates-1} and \ref{gates-2}, we construct a path in $T$ by only crossing $\epsilon$-legal turns for $f$ (legal turns at vertices whose corresponding vertex tree is reduced to a point). After crossing at most $K$ turns, we have necessarily crossed the same orbit of oriented edges twice, so we have constructed a segment of the form $[v,gv]$ for some $g\in F_N$. In particular, we have $l_T^{comb}(g)\le K$, and $g$ is $\epsilon$-legal for $f$. 
\end{proof}

\noindent An element $g\in F_N$ is called an \emph{$\epsilon$-witness} for the pair $(T,T')$ if $\frac{||g||_{T'}}{||g||_T}\ge \text{Lip}(T,T')-\epsilon$.

\begin{lemma} \label{witness}
Let $T,T'\in\overline{cv_N}$ be such that $\text{Lip}(T,T')<+\infty$, and let $f:T\to\overline{T'}$ be an optimal map. Assume that $T_f\neq\emptyset$. For all $\epsilon>0$ and all $K\in\mathbb{N}$, there exists $\epsilon'>0$ such that any element $g\in F_N$ with $l^{comb}_T(g)\le K$ and which is $\epsilon'$-legal for $f$, is an $\epsilon$-witness for the pair $(T,T')$.
\end{lemma}

\begin{proof}
Let $\lambda>0$ be the smallest length of an edge in $T^{simpl}$, and let $\epsilon'>0$ be smaller than $\frac{M}{3}\lambda$, where $M:=\text{Lip}(f)$. Let $g\in F_N$ be $\epsilon'$-legal for $f$ and such that $l^{comb}_T(g)\le K$. Let $v\in C_T(g)$, and let $\gamma$ be the projection of $[v,gv]$ to $X$. Then $||g||_T\le l_X(\gamma)+K\epsilon'$. In addition, every edge in $T_f$ is mapped by $f$ to a segment of length at least $M\lambda$. As $\epsilon'\le\frac{M}{3}\lambda$, the control we have over cancellation for $\epsilon'$-legal turns ensures that after tightening, the length of any fundamental domain of the $f$-image of $C_T(g)$ is at least $\text{Lip}(T,T')l_X(\gamma)-2K\epsilon'$, and we have $||g||_{T'}\ge \text{Lip}(T,T')l_X(\gamma)-2K\epsilon'$, see Figure \ref{fig-cancellation}. Hence

\begin{displaymath}
\frac{||g||_{T'}}{||g||_T}\ge\frac{\text{Lip}(T,T')l_X(\gamma)-2K\epsilon'}{l_X(\gamma)+K\epsilon'}.
\end{displaymath}

\noindent By making $\epsilon'$ arbitrarily small, we can make $\frac{||g||_{T'}}{||g||_T}$ arbitrarily close to $\text{Lip}(T,T')$ (we can assume that $l_X(\gamma)$ is bounded below because $X$ is a finite graph).
\end{proof}

\begin{figure}
\begin{center}
\input{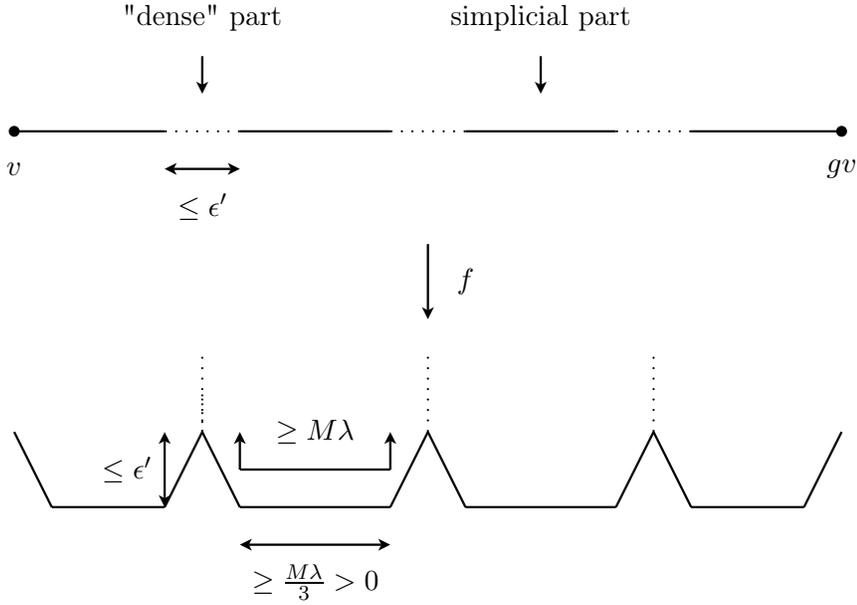}
\caption{The control over cancellation in an $\epsilon'$-legal path.}
\label{fig-cancellation}
\end{center}
\end{figure}

\begin{prop} \label{witness2}
Let $T,T'\in\overline{cv_N}$ be such that $\text{Lip}(T,T')<+\infty$. For all $\epsilon>0$, there exists an $\epsilon$-witness $g$ for the pair $(T,T')$. If $f:T\to\overline{T'}$ is an optimal map, and $T_f=\emptyset$, then we can choose $g$ to be contained in a vertex stabilizer of $T^{simpl}$.
\end{prop}

\begin{proof}
Let $f:T\to\overline{T'}$ be an optimal map (whose existence is provided by Proposition \ref{optimal}). If $T_f\neq\emptyset$, the claim follows from Lemmas \ref{legal-loop} and \ref{witness}. Otherwise, as $X$ is a finite graph, there exists a vertex $v$ in $T^{simpl}$ corresponding to a tree $T_v$ with dense orbits such that $f_{|T_v}:T_v\to f(T_v)$ has Lipschitz constant $\text{Lip}(T,T')$. Denoting by $T_v^{\min}$ the minimal $G_v$-subtree of $T_v$, we get a $G_v$-equivariant map $f_{|T_v^{\min}}:T_v^{\min}\to f(T_v^{\min})$, and $\text{Lip}(f_{|T_v^{\min}})=\text{Lip}(f_{|T_v})$. But $T_v^{\min}$ is a tree with dense orbits, and so is $f(T_v^{\min})$. Hence by Proposition \ref{dense-rigidity}, there exists a unique $\text{Lip}(T,T')$-Lipschitz $G_v$-equivariant map from $T_v^{\min}$ to $f(T_v^{\min})$, and this map is equal to $f_{|T_v^{\min}}$. Hence $\text{Lip}(T_v^{\min},f(T_v^{\min}))=\text{Lip}(f_{|T_v})=\text{Lip}(T,T')$. By Corollary \ref{White-dense}, we have $\text{Lip}(T_v^{\min},f(T_v^{\min}))=\sup_{g\in G_v}\frac{||g||_{T'}}{||g||_T}\le\sup_{g\in F_N}\frac{||g||_{T'}}{||g||_T}$, whence $\text{Lip}(T,T')\le{\Lambda}(T,T')$, and the claim follows.
\end{proof}

\begin{proof}[Proof of Proposition \ref{strong-implies-weak}]
Let $T,T'\in\overline{cv_N}$ be such that $\text{Lip}(T,T')<+\infty$. By Proposition \ref{witness2}, for all $\epsilon>0$, there exists an $\epsilon$-witness for the pair $(T,T')$, so ${\Lambda}(T,T')\ge\text{Lip}(T,T')$. The reverse inequality follows from Proposition \ref{easy-comparison}.
\end{proof}

\subsection{End of the proof of Theorem \ref{strong-White}} \label{sec-weak-strong}

In this section, we finish the proof of Theorem \ref{strong-White}. Proposition \ref{White-blueprint} will be used in the following sections in various contexts to get refinements of Theorem \ref{strong-White} (see Theorem \ref{White-candidates} and Corollaries \ref{White-separable} and \ref{White-primitive}). 

\begin{prop} \label{weak-strong-C}
Let $T\in\overline{cv_N}$, and let $\mathcal{C}(T)\subseteq F_N$ be a subset that contains all vertex stabilizers of $T^{simpl}$. For all $T'\in\overline{cv_N}$, if ${\Lambda}_{\mathcal{C}(T)}(T,T')<+\infty$, then $\text{Lip}(T,T')<+\infty$.
\end{prop}

\begin{proof} 
Let $T_v$ be a vertex tree of $T^{simpl}$ whose stabilizer $G_v$ is nontrivial, and let $T_v^{\min}$ be the minimal $G_v$-invariant subtree of $T_v$. Let $T'_v$ be the minimal $G_v$-invariant subtree of $T'$. As $G_v\subseteq\mathcal{C}(T)$, we have ${\Lambda}_{G_v}(T_v^{\min},T'_v)<+\infty$, so by Corollary \ref{White-dense}, there exists a Lipschitz $F_N$-equivariant map from $T_v^{\min}$ to $\overline{T'_v}$, and hence from $T_v$ to $\overline{T'_v}$. Notice that if an attaching point $p\in T_v$ is fixed by $g\in G_v$, then $f(p)$ is also fixed by $g$. Hence we can define a Lipschitz $F_N$-equivariant map from $T$ to $\overline{T'}$ by sending every vertex tree $T_v$ with dense orbits into the corresponding tree $\overline{T_v'}\subseteq \overline{T'}$ (in particular, every vertex tree which is reduced to a point with nontrivial vertex group $G_v$ is sent to a point fixed by $G_v$), sending the points in $T$ projecting to vertices in $T^{simpl}$ with trivial stabilizer arbitrarily in an $F_N$-equivariant way, and extending linearly on edges. The map we get is Lipschitz because there is a finite number of orbits of vertices and of orbits of edges in $T^{simpl}$.    
\end{proof}

In particular, Proposition \ref{weak-strong-C} applied to $\mathcal{C}(T):=F_N$, together with Proposition \ref{easy-comparison}, implies the following corollary.

\begin{cor}
For all $T,T'\in\overline{cv_N}$, we have $\text{Lip}(T,T')<+\infty$ if and only $\Lambda(T,T')<+\infty$. \qed
\end{cor}

The following proposition will be applied in the sequel to various choices of the set $\mathcal{C}(T)$ to get refinements of Theorem \ref{strong-White}. 

\begin{prop}\label{White-blueprint}
Let $T,T'\in\overline{cv_N}$. Let $\mathcal{C}(T)$ be a subset of $F_N$ that contains all vertex stabilizers of $T^{simpl}$. Assume in addition that either $\text{Lip}(T,T')=+\infty$, or that for all $\epsilon>0$, there exists $g\in\mathcal{C}(T)$ which is an $\epsilon$-witness for the pair $(T,T')$. Then $\text{Lip}(T,T')={\Lambda}_{\mathcal{C}(T)}(T,T')$.
\end{prop}

\begin{proof}
Let $T,T'\in\overline{cv_N}$. Proposition \ref{easy-comparison} shows that ${\Lambda}_{\mathcal{C}(T)}(T,T')\le\text{Lip}(T,T')$, and if ${\Lambda}_{\mathcal{C}(T)}(T,T')=+\infty$, then the reverse inequality is obvious. So we may assume that ${\Lambda}_{\mathcal{C}(T)}(T,T')<+\infty$. Proposition \ref{weak-strong-C} then shows that $\text{Lip}(T,T')<+\infty$, and the conclusion follows from the assumption made on $\mathcal{C}(T)$.
\end{proof}

\begin{proof}[Proof of Theorem \ref{strong-White}]
Theorem \ref{strong-White} follows from Proposition \ref{White-blueprint} applied to $\mathcal{C}(T):=F_N$ and Proposition \ref{witness2}.
\end{proof}

\subsection{Candidates} \label{sec-candidates}

We extend the notion of candidates from Section \ref{sec-metric} to arbitrary trees in $\overline{cv_N}$, compare with \cite[Definition 4.4]{AK12}. An element $g\in F_N$ is a \emph{candidate} in $T$ if there exists $v\in C_T(g)$ such that the segment $[v,gv]$ projects to a loop $\gamma$ in $X$ which is either

\begin{itemize}
\item an embedded loop, or
\item an embedded bouquet of two circles, or
\item a barbell graph, or
\item a \emph{simply-degenerate barbell}, i.e. $\gamma$ is of the form $u\eta\overline{\eta}$, where $u$ is an embedded loop in $X$ and $\eta$ is an embedded path in $X$ with two distinct endpoints which meets $u$ only at its origin, and whose terminal endpoint is a vertex in $X$ with nontrivial stabilizer, or
\item a \emph{doubly-degenerate barbell}, i.e. $\gamma$ is of the form $\eta\overline{\eta}$, where $\eta$ is an embedded path in $X$ whose two distinct endpoints have nontrivial stabilizers, or
\item a vertex in $X$.
\end{itemize} 

\noindent We display the possible shapes of the loop $\gamma$ on Figure \ref{fig-candidates}. In the case of (possibly simply- or doubly-degenerate) barbells, we call $\eta$ the \emph{central path} of $\gamma$. By a more careful analysis of the path built in the proof of Lemma \ref{legal-loop}, we show the following result, which was already noticed by Algom-Kfir \cite[Proposition 4.5]{AK12}. Our strategy of proof follows \cite[Proposition 2.3]{AK11}.

\begin{figure}
\begin{center}
\input{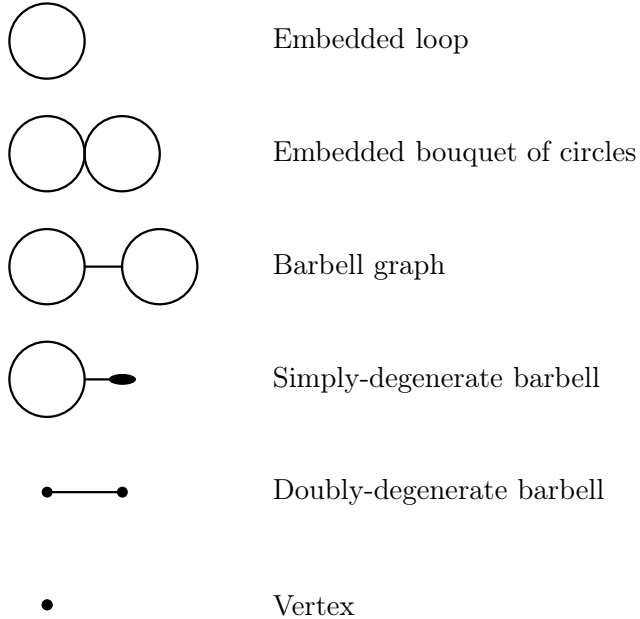}
\caption{The shape of loops in $X$ that represent candidates in $T$.}
\label{fig-candidates}
\end{center}
\end{figure}

\begin{prop} \label{candidate}
Let $T,T'\in\overline{cv_N}$ be such that $\text{Lip}(T,T')<+\infty$. Then for all $\epsilon>0$, there exists an element $g\in F_N$ which is a candidate in $T$ and is an $\epsilon$-witness for the pair $(T,T')$. More precisely, let $f:T\to\overline{T'}$ be an optimal map. 
\begin{itemize}
\item If $X_f=\emptyset$, then there exists $g\in F_N$ whose characteristic set in $T$ projects to a point in $X$, and which is an $\epsilon$-witness for the pair $(T,T')$.
\item If $X_f\neq\emptyset$, then there exists $g\in F_N$ which is a candidate in $T$, and which is $\epsilon$-legal for $f$. 
\end{itemize}
\end{prop}

\begin{proof}
Let $\epsilon>0$. If $X_f=\emptyset$, the claim follows from Proposition \ref{witness2}, so we assume that $X_f\neq\emptyset$. Choose a vertex $v_0\in T_f^{simpl}$, and an edge $e_0$ in $T_f$ whose projection to $T^{simpl}$ is adjacent to $v_0$. Propositions \ref{gates-1} and \ref{gates-2} enable us to construct a path in $T$ of the form $e_0\gamma_0e_1\gamma_1e_2\dots$, where for all integers $i$, 

\begin{itemize}
\item the subpath $e_i$ is an edge in $T_f$, with origin $x_i$ and terminal endpoint $x'_i$, and
\item the subpath $\gamma_i=[x'_i,x_{i+1}]$ lies in a vertex tree of $T$ (it projects to a vertex $v_{i+1}\in T^{simpl}$), and
\item the turn $(e_i,e_{i+1})$ is $\epsilon$-legal for $f$ (and legal for $f$ when $T_{v_{i+1}}$ is reduced to a point).
\end{itemize}

\noindent As the number of orbits of vertices in the simplicial part of $T$ is finite, there exist integers $i,k\in\mathbb{N}$ and an element $g\in F_N$, such that $v_{i+k}=gv_i$. After possibly renumbering the edges, we get a path in $T$ of the form $e_0\gamma_0e_1\dots e_{k-1}$, such that all the turns $(e_i,e_{i+1})$ are $\epsilon$-legal for $f$ (legal at vertices whose vertex tree is reduced to a point), and $v_i\neq v_j$ for all $i\neq j\in\{1,\dots,k\}$, but $(e_{k-1},ge_0)$ might not be $\epsilon$-legal (or legal) for $f$. This path projects to a loop $\gamma$ in $X_f$ which is either embedded, or consists of a single edge crossed successively in both directions, in which case we say it is \emph{degenerate}, see Figure \ref{case-0} (the degenerate case occurs when $k=2$ and the edges $e_0$ and $e_1$ belong to the same orbit of edges).

\begin{figure}
\begin{center}
\input{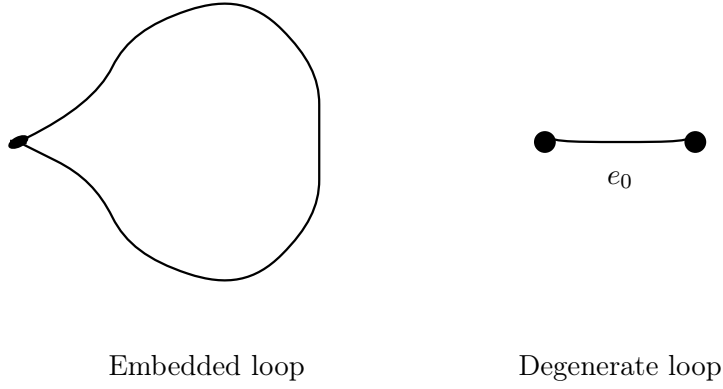}
\caption{The projection to $X$ of the path $e_0\gamma_1e_1\dots e_{k-1}$.}
\label{case-0}
\end{center}
\end{figure}

If there exists $g_k\in G_{v_k}$ so that the turn $(e_{k-1},g_kge_0)$ is $\epsilon$-legal (or legal) for $f$ (which happens for instance as soon as $G_{v_k}$ has rank at least $2$ by Propositions \ref{gates-1} and \ref{gates-2}), then $g_kg$ is a candidate in $T$ which is $\epsilon$-legal for $f$. From now on, we assume that for all $g_k\in G_{v_k}$, the turn $(e_{k-1},g_kge_0)$ is not $\epsilon$-legal for $f$, so in particular the vertex group $G_{v_k}$ is at most cyclic. Proposition \ref{gates-1} shows that for all $g_k\in G_{v_k}$, the turn $(e_{k-1},g_ke_{k-1})$ is not legal for $f$, but ensures the existence of an edge $e_{k}$ in $T_f$ adjacent to $v_{k}$ (not in the same $F_N$-orbit as $e_{k-1}$), such that the turn $(e_{k-1},e_k)$ is legal for $f$. Take this direction, and continue crossing turns which are $\epsilon$-legal for $f$ (legal for $f$ at vertices of $T^{simpl}$ with trivial vertex trees) till you reach a vertex $v_l$ whose orbit has already been visited (i.e. $v_l=g'v_j$ for some $j\in\{0,\dots, l-1\}$ and some $g'\in F_N$). Discussing on the rank of $G_{v_l}$, Propositions \ref{gates-1} and \ref{gates-2} ensure the existence of $g_l\in G_{v_l}$ such that one of the turns $(e_{l-1},g_lg'e_{j-1})$ or $(e_{l-1},g_lg'e_{j})$ is $\epsilon$-legal for $f$ (and legal for $f$ if $T_{v_l}$ is reduced to a point). As above, the path $e_j\dots e_{l-1}$ projects to a loop in $X_f$ which is either embedded or degenerate. Also notice that for all $g\in G_{v_k}$, the turn $(e_0,ge_k)$ is legal for $f$, otherwise Proposition \ref{gates-1} would imply that $(e_{k-1},e_k)$ is not legal for $f$, a contradiction. We give a description of all possible situations, see Figure \ref{cases} where we display the projection to $X$ of the path we have constructed. For simplicity of notations, we will denote a path in $T$ by the sequence of the simplicial edges it crosses.
\\
\\
\textit{Case 1} : The turn $(e_{l-1},g_lg'e_{j})$ is $\epsilon$-legal for some $g_l\in G_{v_l}$, and the path $e_j\dots e_{l-1}$ projects to an embedded loop.
\\
Then the path $e_j\dots e_{l-1}$ is a fundamental domain for the axis of an element $g\in F_N$ which is $\epsilon$-legal in $T$, and it projects to an embedded loop.
\\
\\
\textit{Case 2} : The turn $(e_{l-1},g_lg'e_{j})$ is $\epsilon$-legal for some $g_l\in G_{v_l}$, and the path $e_j\dots e_{l-1}$ projects to a degenerate loop.
\\
Then the path $e_j\dots e_{l-1}$ is a fundamental domain for the axis of an element $g\in F_N$ which is $\epsilon$-legal in $T$, and it projects to a doubly-degenerate barbell.
\\
\\
\noindent\textit{Case 3} : We have $j\in\{1,\dots,k-1\}$, and the turn $(e_{l-1},g_lg'e_{j-1})$ is $\epsilon$-legal for some $g_l\in G_{v_l}$.
\\
Then the path $e_0\dots e_{j-1}(g_lg')^{-1}(\overline{e_{l-1}}\dots \overline{e_{k}})$ is a fundamental domain for the axis of an element $g\in F_N$ which is $\epsilon$-legal in $T$, and it projects to an embedded loop.
\\
\\
\noindent\textit{Case 4} : We have $j=k$, the path $e_0\dots e_{k-1}$ projects to an embedded loop, the path $e_j\dots e_{l-1}$ projects to an embedded loop, and the turn $(e_{l-1},g_lg'e_{j-1})$ is $\epsilon$-legal for some $g_l\in G_{v_l}$.
\\
Then the path $e_0\dots e_{j-1}(g_lg')^{-1}(\overline{e_{l-1}}\dots \overline{e_{k}})$ is a fundamental domain for the axis of an element $g\in F_N$ which is $\epsilon$-legal in $T$, and it projects to a bouquet of two circles.
\\
\\
\noindent\textit{Case 5} : We have $j\in\{k,\dots,l-1\}$, the path $e_0\dots e_{k-1}$ projects to a degenerate loop, the path $e_j\dots e_{l-1}$ projects to an embedded loop, and the turn $(e_{l-1},g_lg'e_{j-1})$ is $\epsilon$-legal for some $g_l\in G_{v_l}$.
\\
Then the path $e_0\dots e_{j-1}(g_lg')^{-1}(\overline{e_{l-1}}\dots \overline{e_{k}})$ is a fundamental domain for the axis of an element $g\in F_N$ which is $\epsilon$-legal in $T$, and it projects to a simply-degenerate barbell.
\\
\\
\noindent\textit{Case 6} : We have $j\in\{k+1,\dots,l-1\}$, the path $e_0\dots e_{k-1}$ projects to an embedded loop, the path $e_j\dots e_{l-1}$ projects to an embedded loop, and the turn $(e_{l-1},g_lg'e_{j-1})$ is $\epsilon$-legal for some $g_l\in G_{v_l}$.
\\
Then the path $e_0\dots e_{k-1}e_k\dots e_{j-1}e_{j}\dots e_{l-1}g_lg'(\overline{e_{j-1}}\dots\overline{e_k})$ is a fundamental domain for the axis of an element $g\in F_N$ which is $\epsilon$-legal in $T$, and it projects to a barbell.
\\
\\
\noindent\textit{Case 7} : We have $j\in\{k,\dots,l-1\}$, the path $e_0\dots e_{k-1}$ projects to an embedded loop, the path $e_j\dots e_{l-1}$ projects to a degenerate loop, and the turn $(e_{l-1},g_lg'e_{j-1})$ is $\epsilon$-legal for some $g_l\in G_{v_l}$.
\\
Then the path $e_0\dots e_{k-1}e_k\dots e_{j-1}e_{j}\dots e_{l-1}g_lg'(\overline{e_{j-1}}\dots\overline{e_k})$ is a fundamental domain for the axis of an element $g\in F_N$ which is $\epsilon$-legal in $T$, and it projects to a simply-degenerate barbell.
\\
\\
\noindent\textit{Case 8} : We have $j\in\{k,\dots,l-1\}$, the path $e_0\dots e_{k-1}$ projects to a degenerate loop, the path $e_j\dots e_{l-1}$ projects to a degenerate loop, and the turn $(e_{l-1},g_lg'e_{j-1})$ is $\epsilon$-legal for some $g_l\in G_{v_l}$.
\\
Then the path $e_0\dots e_{k-1}e_k\dots e_{j-1}e_{j}\dots e_{l-1}g_lg'(\overline{e_{j-1}}\dots\overline{e_k})$ is a fundamental domain for the axis of an element $g\in F_N$ which is $\epsilon$-legal in $T$, and it projects to a doubly-degenerate barbell.
\\
\\
\indent In all cases, we have found an element $g\in F_N$ which is a candidate in $T$, and which is $\epsilon$-legal for $f$. In addition, there exists $K\in\mathbb{N}$ such that for all $g\in F_N$, if $g$ is a candidate in $T$, then $l_T^{comb}(g)\le K$ (we recall the notation $l_T^{comb}$ from Section \ref{sec-strong-weak}). The conclusion thus follows from Lemma \ref{witness}.
\end{proof}

\begin{figure}
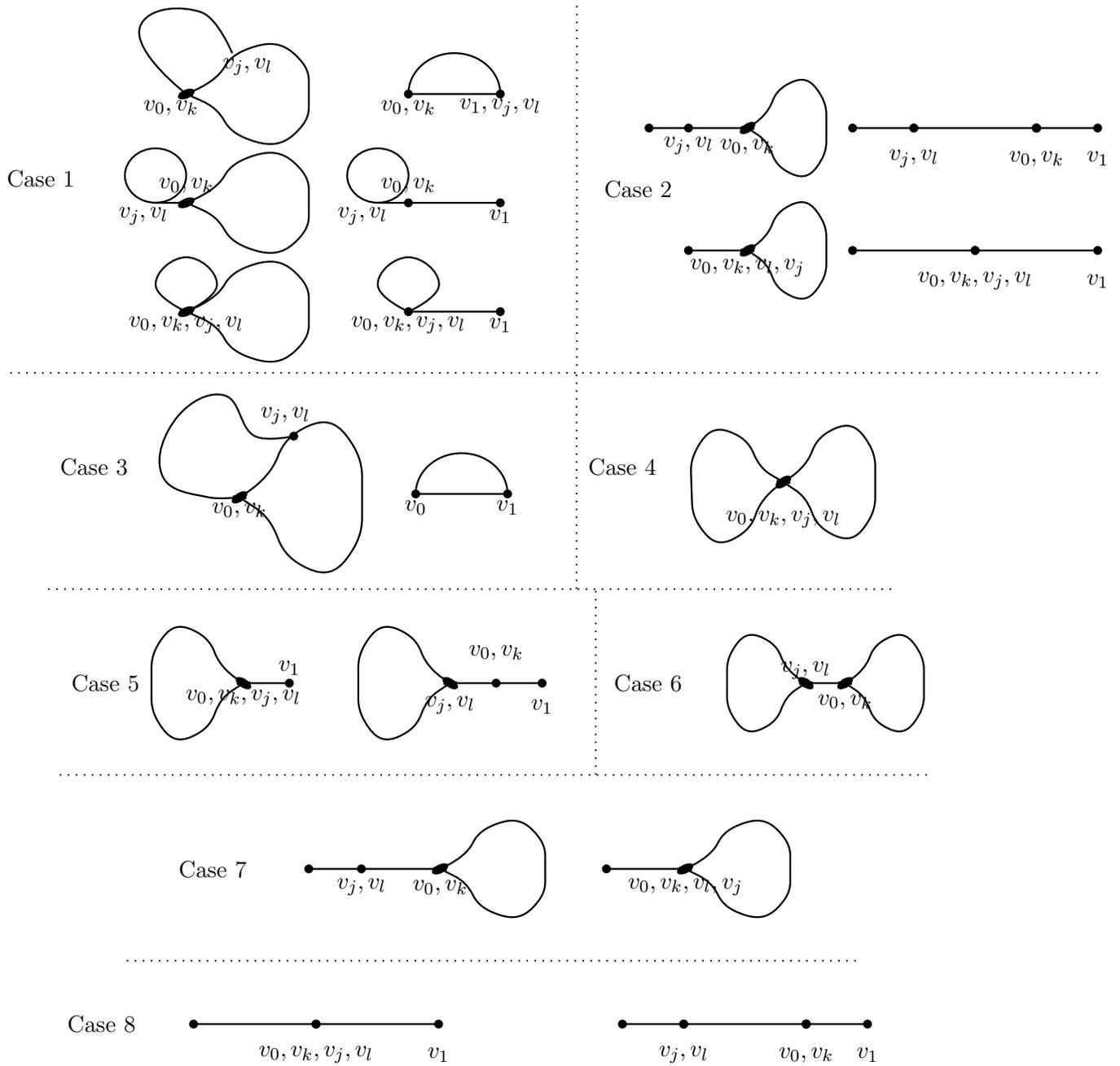

\begin{center}
\def\JPicScale{1}
\input{Figure-15-1,2.pst}
\\
\input{Figure-15-3,4.pst}
\\
\input{Figure-15-5,6.pst}
\\
\input{Figure-15-7.pst}
\\
\input{Figure-15-8.pst}
\caption{The projection to $X$ of the path constructed in the different cases of the proof of Proposition \ref{candidate}.}
\label{cases}
\end{center}
\end{figure}

\begin{theo}\label{White-candidates}
For all $T,T'\in\overline{cv_N}$, we have

\begin{displaymath}
\text{Lip}(T,T')=\sup_{g \text{~candidate in $T$}}\frac{||g||_{T'}}{||g||_T}.
\end{displaymath}
\end{theo}

\begin{proof}
Let $T,T'\in\overline{cv_N}$, and let $\mathcal{C}(T)$ be the set of elements of $F_N$ which are candidates in $T$. By definition, the set $\mathcal{C}(T)$ contains all vertex groups of $T^{simpl}$, and Proposition \ref{candidate} shows that $\mathcal{C}(T)$ satisfies the assumption of Proposition \ref{White-blueprint}. The conclusion thus follows from Proposition \ref{White-blueprint}.
\end{proof}

\subsection{The case of good trees} \label{sec-separable-candidates}

Let $T\in\overline{cv_N}$. We now carry on some further analysis on the set of candidates to show that when $T^{simpl}$ contains at least two $F_N$-orbits of edges with trivial stabilizers, for all $\epsilon>0$, we can find an element of $F_N$ which is simple, is a candidate in $T$, and is an $\epsilon$-witness for the pair $(T,T')$. 

\begin{prop} \label{strong-witness}
Let $T,T'\in\overline{cv_N}$ be such that $\text{Lip}(T,T')<+\infty$. Assume that $T^{simpl}$ contains at least two orbits of edges with trivial stabilizers. Then for all $\epsilon>0$, there exists $g\in F_N$ which is simple, is a candidate in $T$, and is an $\epsilon$-witness for the pair $(T,T')$.
\end{prop}

\begin{proof}
Let $f:T\to\overline{T'}$ be an optimal map (which exists by Proposition \ref{optimal}), and let $\epsilon>0$. If $X_f=\emptyset$, the claim follows from Proposition \ref{candidate} and Lemma \ref{detect-separability}, so we assume that $X_f\neq\emptyset$. By Proposition \ref{candidate}, there exists $g\in F_N$ which is a candidate in $T$ and is $\epsilon$-legal for $f$. Let $v\in C_T(g)$ be such that the projection $\gamma$ of $[v,gv]$ to $X$ has one of the forms prescribed by the definition of a candidate. If $\gamma$ is either an embedded loop or an embedded bouquet of two circles, then $\gamma$ crosses each edge of $X$ at most once. As $T$ contains an edge with trivial stabilizer in its simplicial part, Lemma \ref{detect-separability} ensures that $g$ is simple. The same argument also shows that $g$ is simple in the case where $\gamma$ does not cross some edge with trivial stabilizer of $X$, or when $\gamma$ is a (possibly simply-degenerate) barbell, one of whose loops crosses an edge of $X$ with trivial stabilizer. Hence we can assume that $\gamma$ is a (possibly simply- or doubly-degenerate) barbell, and that all edges in $X$ with trivial stabilizer belong to the central path $\gamma'$ of $\gamma$.

Assume that $\gamma'$ contains a vertex $v$ whose stabilizer has rank at least $2$, and such that if $v$ is an endpoint of $\gamma'$, then $v$ is adjacent to a loop of the barbell. Then $v$ separates $\gamma$ into two shorter simply- or doubly-degenerate barbells or embedded loops, at least one of which, which we denote by $\gamma''$, avoids an edge of $X$ with trivial stabilizer. Propositions \ref{gates-1} and \ref{gates-2} show that there exists $g\in F_N$ whose axis in $T$ projects to $\gamma''$ and which is $\epsilon$-legal for $f$, and $g$ is simple by Lemma \ref{detect-separability}.

We now restrict to the case where all nonextremal vertices of $\gamma'$ have vertex group at most cyclic. Assume that two edges in $\gamma'$ with nontrivial cyclic stabilizers have a common vertex $v$ (whose vertex group is cyclic). Then there exists an edge $e'$ in $X\smallsetminus\gamma$ adjacent to $v$, and $e'$ has trivial stabilizer because $T$ is very small, and the vertex group $G_v$ is cyclic. Lemma \ref{detect-separability} ensures that $g$ is simple. From now on, we assume that $\gamma$ does not contain two consecutive edges with nontrivial cyclic stabilizers.

Now assume that $\gamma'$ contains two edges with trivial edge groups having a common vertex $v$ (whose vertex group is at most cyclic). If $G_v$ is trivial, then any third edge coming out of $v$ has trivial stabilizer, and Lemma \ref{detect-separability} shows that $g$ is simple. We now assume that $G_v$ is infinite cyclic. Let $e,e'$ be two consecutive edges in $C_T(g)$ (not in the same $F_N$-orbit) adjacent to a vertex $\widetilde{v}\in T$ that projects to $v$. Denote by $t$ a generator of the cyclic group $G_{\widetilde{v}}$. If $(e,te)$ is legal for $f$, then again we can replace $g$ by another candidate $g'$ which is $\epsilon$-legal for $f$, and is represented by a loop which does not cross the orbit of $e'$. We now assume that for all $t\in G_{\widetilde{v}}$, the turn $(e,te)$ is not legal for $f$. If for some $t'\in G_{\widetilde{v}}$, the turn $(e,t'e')$ were not legal for $f$, then by Proposition \ref{gates-1}, the turn $(t'e,t'e')$ would not be legal for $f$, contradicting the fact that $g$ is $\epsilon$-legal for $f$. So for all $t\in G_{\widetilde{v}}$, the turn $(e,te')$ is legal for $f$. Let $\widetilde{\gamma}$ be a fundamental domain of the axis of $g$ that projects to $\gamma$ and crosses twice a turn at a vertex in the orbit of $\widetilde{v}$. The previous argument shows that up to replacing $g$ by another candidate $g'$ which is also $\epsilon$-legal for $f$, we can assume that these two turns belong to the same $F_N$-orbit (of the form $(e,t^ke')$ for some $k\in\mathbb{Z}$). We claim that $g'$ is simple. Indeed, by equivariantly folding small initial segments of the edges $e$ and $t^ke'$, one constructs a new $F_N$-tree that projects to a graph of groups in which $g'$ is represented by a loop that avoids an edge with trivial stabilizer, see Figure \ref{slide}. By Lemma \ref{detect-separability}, this shows that $g'$ is simple.  

\begin{figure}
\begin{center}
\input{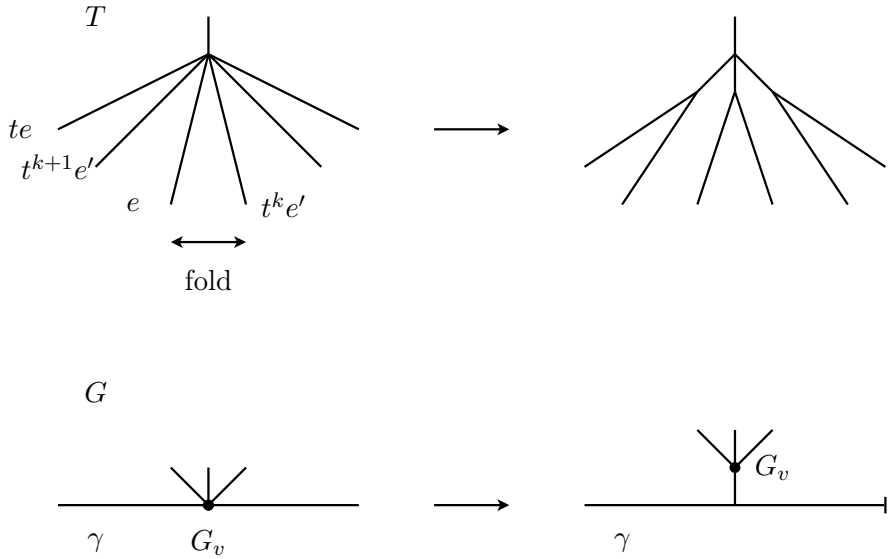}
\caption{Sliding the vertex $v$ to detect simple elements.}
\label{slide}
\end{center}
\end{figure}

We are thus left with the case where $\gamma'$ contains an edge $e$ with nontrivial cyclic edge group, which is surrounded in $\gamma'$ by two edges with trivial edge groups. Denote by $e_1$ and $e_2$ the two edges in a lift of $\gamma$ to $T$ that are adjacent to a lift of $e$. By the same argument as above, we can assume that for all $t\in G_e$, the turns $(e_1,te_1)$ and $(e_2,te_2)$ are not legal for $f$. So for all $t\in G_e$, the turns $(te_1,e)$ and $(e,te_2)$ are legal for $f$, otherwise Proposition \ref{gates-1} would imply that $(e,e_1)$ or $(e,e_2)$ is not legal for $f$, contradicting the fact that $g$ is $\epsilon$-legal for $f$. As above we can construct a candidate $g'$ which is $\epsilon$-legal for $f$, and is simple. Indeed, we can choose $g'$ such that when equivariantly collapsing $e$ to a vertex, and applying the same folding argument as above, we get a new graph of groups in which $g'$ is represented by a loop that avoids an edge with trivial edge group. This again shows that $g'$ is simple.

Hence we have found an element $g\in F_N$ which is simple, is a candidate in $T$, and is $\epsilon$-legal for $f$. The conclusion thus follows from Lemma \ref{witness} since there is a bound on the combinatorial length of a candidate in $T$.
\end{proof}

\begin{cor} \label{White-separable}
Let $T,T'\in\overline{cv_N}$. If $T^{simpl}$ contains two distinct orbits of edges with trivial stabilizers, then 
\begin{displaymath}
\text{Lip}(T,T')=\sup_{g \text{~simple}}\frac{||g||_{T'}}{||g||_T}.
\end{displaymath}
\end{cor}

\begin{proof}
Let $\mathcal{C}$ be the set of simple elements of $F_N$. Lemma \ref{detect-separability} shows that $\mathcal{C}$ contains all the vertex groups of $T^{simpl}$, and Proposition \ref{strong-witness} shows that $\mathcal{C}$ satisfies the assumption of Proposition \ref{White-blueprint}. Hence Corollary \ref{White-separable} follows from Proposition \ref{White-blueprint}.
\end{proof}

\begin{cor} \label{White-good}
Let $T,T'\in\overline{cv_N}$. If $T$ is good, then 
\begin{displaymath}
\text{Lip}(T,T')=\sup_{g \text{~simple}}\frac{||g||_{T'}}{||g||_T}.
\end{displaymath}
\end{cor}

\begin{proof}
This follows from Corollary \ref{White-separable} and Proposition \ref{blueprint-1} applied to the set of simple elements of $F_N$.
\end{proof}

\begin{cor} \label{White-primitive}
Let $T,T'\in\overline{cv_N}$. If $T$ is good, then 
\begin{displaymath}
\text{Lip}(T,T')=\sup_{g\in\mathcal{P}_N}\frac{||g||_{T'}}{||g||_T}.
\end{displaymath}
\end{cor}

\begin{proof}
Let $w\in F_N$ be a simple element of $F_N$ (contained in some proper free factor of $F_N$), and let $w'$ be a primitive element contained in a complementary free factor. Then $w'w^k$ is primitive for all $k\in\mathbb{N}$, and $||w||_T=\lim_{k\to +\infty}\frac{||w'w^k||_T}{k}$. If $||w||_T\neq 0$, we thus get that 

\begin{displaymath}
\frac{||w||_{T'}}{||w||_T}=\lim_{k\to +\infty}\frac{||w'w^k||_{T'}}{||w'w^k||_T},
\end{displaymath}

\noindent so 

\begin{displaymath}
\frac{||w||_{T'}}{||w||_T}\le\sup_{g \in\mathcal{P}_N}\frac{||g||_{T'}}{||g||_T}.
\end{displaymath}

\noindent If $||w||_T=0$ and $||w||_{T'}>0$, then $\frac{||w'w^k||_{T'}}{||w'w^k||_T}$ tends to $+\infty$ as $k$ tends to $+\infty$, so the above inequality still holds. It also holds when $||w||_T=||w||_{T'}=0$, because in this case we have $\frac{||w||_{T'}}{||w||_T}=0$ by convention. Hence

\begin{displaymath}
\sup_{w \text{~simple}}\frac{||w||_{T'}}{||w||_T}=\sup_{g \in\mathcal{P}_N}\frac{||g||_{T'}}{||g||_T},
\end{displaymath}

\noindent and the claim follows from Corollary \ref{White-good}.
\end{proof}

From Proposition \ref{blueprint-2} and Corollary \ref{White-primitive}, we deduce the following statement.

\begin{prop} \label{case-good-approx}
Let $T,T'\in\overline{cv_N}$ be two good $F_N$-trees. If $T$ and $T'$ are primitive-equivalent, then $T=T'$. \qed
\end{prop}

\begin{rk} \label{separable-needed}
The condition on $T$ cannot be removed in Corollary \ref{White-primitive}, otherwise the simple-equivalence relation would be trivial on $\overline{cv_N}$, contradicting our analysis in Section \ref{sec-beginning}. If the translation length functions of two distinct trees $T,T'\in\overline{cv_N}$ are equal in restriction to $\mathcal{P}_N$, then either $T$ or $T'$ is not good. Applying Theorem \ref{approximation-by-separable} to the $F_{N-1}$-tree appearing in the definition of pull-equivalence classes, we see that the standard element of the class is thus the only good tree in its class. It follows from Remark \ref{rk-surface} that trees whose pull-equivalence class is nontrivial are geometric and contain no exotic components.
\end{rk}

\section{End of the proof of the main theorem} \label{sec-rigidity}

We now finish the proof of the main theorem of the paper.

\begin{theo} \label{equivalences}
For all $T,T'\in\overline{cv_N}$, the following assertions are equivalent.
\begin{itemize}
\item For all $g\in\mathcal{P}_N$, we have $||g||_T=||g||_{T'}$.
\item For all simple elements $g\in F_N$, we have $||g||_T=||g||_{T'}$.
\item The trees $T$ and $T'$ are special-pull-equivalent.
\end{itemize}
\end{theo}

In view of Propositions \ref{primitive-separable} and \ref{slide-separable}, we are left showing that simple-equivalent trees are special-pull-equivalent. Given an $\mathbb{R}$-tree $T$ and $x\in T$, a \emph{direction} based at $x$ is a germ of nondegenerate segments $[x,y]$ with $y\neq x$. In particular, any U-turn in an $F_N$-tree $T$ is defined by a pair of directions based at some point in $T$. Any direction based at $x$ defines an open half-tree of $T$, which is the set of all $y\in T\smallsetminus\{x\}$ such that $[x,y]$ contains $d$. The axis of an element $g\in F_N$ which is hyperbolic in $T$ \emph{crosses} a direction $d$ based at a point $x\in T$ if $x\in C_T(g)$ and there exists $y\in C_T(g)$ such that $d$ is the germ of $[x,y]$ (we also say that $C_T(g)$ crosses a pair of distinct directions $\{d,d'\}$ based at a point in $T$ if it crosses both $d$ and $d'$). A \emph{full family of U-turns} over $e$ is the collection of all turns of the form $(e,g^ke)$ for $k\in\mathbb{Z}$, where $(e,ge)$ is some given U-turn over $e$ (in particular $g$ fixes an endpoint of $e$ and is not a proper power). A \emph{compatible set of U-turns} over $e$ is a set consisting of at most one full family of U-turns at each extremity of $e$.  

\begin{prop} \label{separable-in-graph}
Let $T$ be an $F_N$-tree which contains an orbit of edges with trivial stabilizer $e^0$. Let $Y$ be a compatible set of U-turns over $e^0$. Let $\{d,d'\}$ be a pair of distinct directions based at a point $x_0\in T$, which do not define a U-turn over $e^0$, and do not define a U-turn over an edge in $T$ with nontrivial stabilizer. Then there exists a simple element $g\in F_N$ which is hyperbolic in $T$ and whose axis in $T$ crosses $\{d,d'\}$ but does not cross any of the orbits of the turns in $Y$.  
\end{prop}

\begin{proof}
Let $v_1$ and $v_2$ denote the extremities of $e^0$. We think of $e^0$ as an open subset of $T$.
\\
\\
\textit{Case 1} : The edge $e^0$ projects to a nonseparating edge in the quotient graph of actions.\\
Let $T'$ be a connected component of $T\smallsetminus F_N.e^0$. One can always find a path $\gamma$ in $T'$ which joins a point in the orbit of $v_1$ to a point in the orbit of $v_2$, and let $\gamma'$ be the concatenation of $\gamma$ and of an edge in the orbit of $e^0$. We will show that if $x_0\in T'$, then $\gamma$ can be chosen so that $\gamma'$ crosses a turn in the orbit of $\{d,d'\}$ (this property is automatic if $x_0$ belongs to the interior of $e^0$). Then $\gamma'$ is a fundamental domain for the axis of an element $g\in F_N$ which has the desired properties (in particular, it is primitive because $\gamma'$ crosses the orbit of $e^0$ exactly once).

Assume that $x_0\in T'$. Let $A$ be the stabilizer of $T'$. If $A$ stabilizes an edge in $T$, then $A$ is cyclic, and $T'$ consists of a single edge $e$ with cyclic stabilizer. In this case, we choose $\gamma$ to be equal to this edge. Otherwise, the $A$-minimal subtree of $T'$ is well-defined, we denote it by $T'_{min}$. 

If $T'\smallsetminus T'_{min}$ contains a simplicial edge with trivial stabilizer, then $e^0$ projects to a loop-edge in the quotient graph of actions, which has a valence $3$ vertex with trivial stabilizer, attached to a separating edge. In this case, each of the directions $d$ and $d'$ is either contained in $e^0$, or determines a half-tree of $T'$ that contains a point in the orbit of $v_1$, which is equal to the orbit of $v_2$. Therefore, we can find a path $\gamma$ in $T'$ which joins two points in the orbit of $v_1$, so that the concatenation $\gamma'$ of $\gamma$ and an edge in the orbit of $e^0$ crosses a turn in the orbit of $\{d,d'\}$.

Otherwise, all edges in $T'\smallsetminus T'_{min}$ (if any) have nontrivial stabilizer, equal to the stabilizer of the extremity of the edge in the orbit of $e^0$ to which they are attached. If none of the open half-trees of $T$ determined by the directions $d$ and $d'$ intersects $T'_{min}$, then the assumption made on the turn $\{d,d'\}$ implies that either 

\begin{itemize}
\item the vertices $v_1$ and $v_2$ belong to the same $F_N$-orbit, and $d$ and $d'$ are contained in translates of $e^0$, in which case we can choose $\gamma$ to be reduced to a point, or
\item one of the directions, say $d$, is contained in an edge $e$ with nontrivial stabilizer, and $d'$ is contained in $e^0$, in which case we choose $\gamma$ to be equal to $e$, or
\item the directions $d$ and $d'$ belong to two distinct orbits of edges with nontrivial stabilizer whose concatenation forms the desired path $\gamma$. 
\end{itemize}

\noindent Otherwise, up to exchanging $d$ and $d'$, the open half-tree of $T'$ determined by $d$ contains both a point in the orbit of $v_1$ and a point in the orbit of $v_2$. In addition, either $x_0$ belongs to the orbit of $v_1$ or $v_2$, or the open half-tree determined by $d'$ contains a point in the orbit of $v_1$ or $v_2$. So we can find a path $\gamma$ in $T'$ joining a point in the orbit of $v_1$ to a point in the orbit of $v_2$, so that $\gamma'$ crosses a turn in the orbit of $\{d,d'\}$. 
\\
\\
\textit{Case 2} : The edge $e^0$ projects to a separating edge in the quotient graph of actions.\\ 
Let $T_1$ (resp. $T_2$) be the connected component of $T\smallsetminus F_N.e^0$ that contains $v_1$ (resp. $v_2$). For all $i\in\{1,2\}$, let $A_i$ be the stabilizer of $T_i$, and let $T_i^{min}$ be the $A_i$-minimal subtree of $T_i$. Up to interchanging the roles of $T_1$ and $T_2$, we can assume that $x_0\in T_1$, or $x_0\in e^0$.

If $x_0\in T_1$, then the open half-tree determined by one of the directions $d$ or $d'$ intersects $T_1^{min}$, and in both cases we can choose a reduced path $\gamma_1$ in $T_1$ that crosses a turn in the orbit of $\{d,d'\}$ and joins $v_1$ to a translate $g_1v_1$. If $A_2$ is not elliptic in $T_2$, then we can choose a primitive element $g_2\in A_2$ that is hyperbolic in $T_2$. If $A_2$ is elliptic in $T_2$ and has rank at least $2$, then we can choose a primitive element $g_2\in A_2$ such that $(e^0,g_2e^0)\notin Y$. If $A_2$ is cyclic and elliptic in $T_2$, generated by an element $g_2$, then our definition of U-turns implies that $(e^0,g_2e^0)\notin Y$. Then $g_1^{-1}g_2$ satisfies the required conditions (in particular it is primitive, because if $\{a_1,\dots,a_k\}$ is a free basis of $A_1$, and $\{a_{k+1},\dots,a_{N-1},g_2\}$ is a free basis of $A_2$, then $\{a_1,\dots,a_{N-1},g_1^{-1}g_2\}$ is a free basis of $F_N$).

If $x_0\in e^0$, then we can find two primitive elements $g_1\in A_1$ and $g_2\in A_2$ as above and let $g:=g_1^{-1}g_2$. 
\end{proof}

Proposition \ref{separable-in-graph} can also be restated in the following way. 

\begin{cor} \label{cor-separable-in-graph}
Let $T,\widehat{T}\in\overline{cv_N}$. Assume that $\widehat{T}$ contains exactly one orbit of edges $e^0$ with trivial stabilizer, and that $T$ is a pull of $\widehat{T}$. Let $\{d,d'\}$ be a pair of distinct directions based at the same point in $\widehat{T}$, which does not define a U-turn over $e^0$, and does not define a U-turn over an edge with nontrivial stabilizer. Then there exists a simple element $g\in F_N$, which is hyperbolic in $\widehat{T}$, whose axis in $\widehat{T}$ crosses $\{d,d'\}$, and such that $||g||_T=||g||_{\widehat{T}}$. 
\end{cor}

\begin{proof}
As $T$ is obtained from $\widehat{T}$ by equivariantly folding a collection $Y$ of U-turns, we have $||g||_T\le ||g||_{\widehat{T}}$ for all $g\in F_N$, with equality as long as $C_{\widehat{T}}(g)$ does not cross any turn in the orbit of a turn in $Y$. Corollary \ref{cor-separable-in-graph} thus follows from Proposition \ref{separable-in-graph}.
\end{proof}

\begin{prop} \label{separable-slide}
Let $T,T'\in\overline{cv_N}$. If $T$ and $T'$ are simple-equivalent, then there exists a good tree $\widehat{T}\in\overline{cv_N}$ such that either $T=T'=\widehat{T}$, or there exists an edge $e^0$ in $\widehat{T}$ with trivial stabilizer such that $T$ and $T'$ are both obtained from $\widehat{T}$ by pulling $e^0$.
\end{prop}

\begin{proof}
If both trees $T$ and $T'$ are good, then by Proposition \ref{case-good-approx} we have $T=T'$. We can thus assume that $T$ is not good, hence there exists a good tree $\widehat{T}$ having exactly one orbit of edges with trivial stabilizer such that $T$ is a pull of $\widehat{T}$ (Theorem \ref{approximation-by-separable}).

We first show the existence of an $F_N$-equivariant morphism from $\widehat{T}$ to $\overline{T'}$, which is isometric on edges. As $T$ is a pull of $\widehat{T}$, we have $||g||_{\widehat{T}}\ge ||g||_T$ for all $g\in F_N$. As $T$ and $T'$ are simple-equivalent, we thus have $||g||_{\widehat{T}}\ge ||g||_{T'}$ for all simple elements $g\in F_N$. As $\widehat{T}$ is good, Corollary \ref{White-good} provides a $1$-Lipschitz $F_N$-equivariant map $f:\widehat{T}\to \overline{T'}$, which we may choose to be linear on edges in the simplicial part of $\widehat{T}$. As $\widehat{T}$ contains an edge with trivial stabilizer, Lemma \ref{detect-separability} ensures that all vertex stabilizers of $\widehat{T}^{simpl}$ lie in some proper free factor of $F_N$, so for all $g\in F_N$ belonging to one of these vertex stabilizers, we have $||g||_{\widehat{T}}=||g||_T=||g||_{T'}$. Let $v$ be a vertex in $\widehat{T}^{simpl}$ whose stabilizer $G_v$ has rank at least $2$. Then the $G_v$-minimal subtree of $T'$ has the same translation length function as the $G_v$-minimal subtree $T_v^{min}$ of $\widehat{T}$, so by Theorem \ref{rigidity}, these two trees are $F_N$-equivariantly isometric. So $T_v^{min}$ isometrically embeds as an $F_N$-invariant subtree in $T'$, and by Proposition \ref{dense-rigidity}, the map $f$ restricts to a $G_v$-equivariant isometry on $T_v^{min}$, and hence on $T_v$. Hence we can write $f=f_2\circ f_1$, where $f_1$ reduces the length of some edges in the simplicial part of $\widehat{T}$, and $f_2$ is a morphism which is isometric on edges. If $f_1$ is not equal to the identity map, then $f_1$ strictly reduces the length of an edge $e'$ in the simplicial part of $\widehat{T}$. Corollary \ref{cor-separable-in-graph}, applied to a pair of opposite directions in the edge $e'$, gives the existence of a simple element $g\in F_N$, whose axis in $T$ crosses $e'$ (so that $||g||_{T'}<||g||_{\widehat{T}}$), and such that $||g||_{T}=||g||_{\widehat{T}}$. This is impossible as $T$ and $T'$ are simple-equivalent. So $f_1$ is equal to the identity map, and hence $f$ is a morphism.

Assume that $f$ identifies a pair $\{d,d'\}$ of directions in $T$. Since U-turns over edges with nontrivial stabilizers cannot be folded by $f$ (Lemma \ref{morphism-2}), Corollary \ref{cor-separable-in-graph} ensures that the pair $\{d,d'\}$ defines a U-turn over $e^0$, otherwise we would find a simple element $g\in F_N$ with $||g||_{T'}<||g||_T$. In other words, all turns in $T$, except possibly U-turns over $e^0$, are legal for the morphism $f$. So $f$ factors through a tree $T_1$ obtained by equivariantly identifying maximal subsegments of the unique edge of $\widehat{T}$ with trivial stabilizer along some translate at each of its extremities. If $T_1$ also contains an edge with trivial stabilizer, then the maximality condition in the definition of $T_1$ ensures that all turns in $T_1$ are legal, so $T'=T_1$. If all edges in the simplicial part of the tree $T_1$ have nontrivial stabilizer, then no more folding can occur (Lemmas \ref{morphism-1}, \ref{morphism-2} and \ref{morphism-3}), so again $T'=T_1$. The claim follows.
\end{proof}

In order to complete the proof of Theorem \ref{equivalences}, we are thus left showing the following.

\begin{prop}\label{weak-horo-strong}
Let $T,T'\in\overline{cv_N}$ be simple-equivalent. Assume that there exists a good tree $\widehat{T}$ and an edge $e^0$ in $\widehat{T}$ with trivial stabilizer, such that $T$ and $T'$ are obtained from $\widehat{T}$ by pulling $e^0$. Then $T$ and $T'$ are special-pull-equivalent.
\end{prop}

\begin{rk}
The content of Proposition \ref{weak-horo-strong} is to show that the pull operation performed when passing from $T$ to $T'$ is of the particular form prescribed by the definition of special-pull-equivalent trees (Definition \ref{de-NS-pull}).
\end{rk}

\begin{proof}
Let $X$ denote the underlying graph of the canonical graph of actions associated to $\widehat{T}$ provided by Proposition \ref{graph-of-actions}. We again denote by $e^0$ the image of $e^0$ in $X$.
\\
\\
\textit{Case 1} : The edge $e^0$ is separating in $X$.\\
Then there exists a free splitting of $F_N$ of the form $F_N=A_1\ast A_2$ (for which we denote by $k_i$ the rank of $A_i$ for all $i\in\{1,2\}$, with $k_1+k_2=N$), an $A_i$-tree $T_i\in\overline{cv_{k_i}}$ together with an attaching point $p_i\in \overline{T_i}$ for all $i\in\{1,2\}$, nonnegative real numbers $l,l_1,l_2,l'_1,l'_2\in\mathbb{R}_{+}$ satisfying $l_1+l_2\le l$ and $l'_1+l'_2\le l$, and elements $g_i,g'_i\in A_i$ which are elliptic in $T_i$ and fix $p_i$ for all $i\in\{1,2\}$ such that $T$ and $T'$ are the trees dual to the graphs of actions displayed on Figure \ref{fig-the-graphs}. (Notice that up to changing the values of $l_i$ and $l'_i$, we can always assume the trees $T_i$ to be minimal). If $N=2$, then $\widehat{T}$ is dual to a (possibly simply- or doubly-degenerate) barbell graph, and no pull can be performed on $\widehat{T}$, whence $T=T'$. 

\begin{figure}
\begin{center}
\input{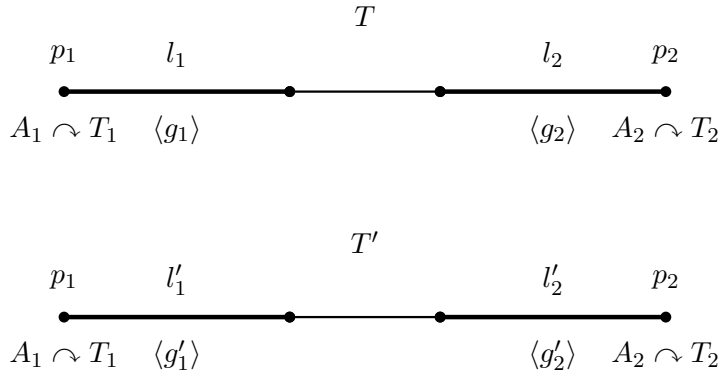}
\caption{The trees $T$ and $T'$ in Case 1 of the proof of Proposition \ref{weak-horo-strong}.}
\label{fig-the-graphs}
\end{center}
\end{figure}

We now assume that $N\ge 3$. Up to interchanging the roles of $A_1$ and $A_2$, we may assume that $k_2\ge 2$, so that there exists a primitive element $g''_2\in A_2\smallsetminus(\langle g_2\rangle \cup\langle g'_2\rangle)$. The element $g_1g''_2$ is primitive in $F_N$, because if $\{a_1,\dots, a_{k_1}\}$ is a free basis of $A_1$, and $\{g''_2,a'_2,\dots,a'_{k_2}\}$ is a free basis of $A_2$, then $\{a_1,\dots,a_{k_1},g_1g''_2,a'_2,\dots,a'_{k_2}\}$ is a free basis of $F_N$. We have

\begin{displaymath}
||g_1g''_2||_T=2(l-l_1)+d_{\overline{T_2}}(p_2,g''_2p_2),
\end{displaymath}

\noindent and as $d_{\overline{T_1}}(p_1,g_1p_1)=0$, we have

\begin{displaymath}
||g_1g''_2||_{T'}=\left\{
\begin{array}{rl}
2(l-l'_1)+d_{\overline{T_2}}(p_2,g''_2p_2)&\text{~if~} g'_1=g_1^{\pm 1}\\
2l+d_{\overline{T_2}}(p_2,g''_2p_2)&\text{~if~} g'_1\neq g_1^{\pm 1}
\end{array}.
\right.
\end{displaymath}

\noindent As $T_1$ and $T_2$ are simple-equivalent, as soon as $l_1>0$, we have $l_1=l'_1$ and $g_1={g'_1}^{\pm 1}$. The same argument shows that as soon as $l'_1>0$, we have $l_1=l'_1$ and $g_1={g'_1}^{\pm 1}$. If $k_1\ge 2$, we argue similarly to show that either $l_2=l'_2=0$, or $l_2=l'_2$ and $g_2={g'_2}^{\pm 1}$, hence $T=T'$. If $k_1=1$, then $l_1=l'_1=0$. Let $g''_1\in A_1$ be a generator of $A_1$. Comparing the translation lengths of $g''_1g_2$ in $T$ and in $T'$ also shows that either $l_2=l'_2=0$, or $l_2=l'_2$ and $g_2={g'_2}^{\pm 1}$, whence $T=T'$. 
\\
\\
\textit{Case 2} : The edge $e^0$ is nonseparating in $X$.\\ 
Then there exists a corank one free factor $A$ of $F_N$, with a choice of a stable letter $t$, a (not necessarily minimal) very small $A$-tree $T_{N-1}$, attaching points $p_1,p_2\in\overline{T_{N-1}}$, nonnegative real numbers $l,l_1,l_2,l'_1,l'_2\in\mathbb{R}_{+}$ satisfying $l_1+l_2\le l$ and $l'_1+l'_2\le l$, and elements $g_i,g'_i\in A$ which are elliptic in $T_{N-1}$ and fix $p_i$ for all $i\in\{1,2\}$, such that $T$ and $T'$ are the trees dual to the graphs of actions displayed on Figure \ref{fig-the-graphs-2}. If $N=2$, then $A$ contains no proper free factor, so $T$ and $T'$ are special-pull-equivalent. We now assume that $N\ge 3$. Assume that $l_1>0$ and that $g_1$ is contained in a corank one free factor $B$ of $A$. As $N\ge 3$, we can find $g'\in A\smallsetminus(\langle g_2\rangle\cup\langle g'_2\rangle)$ such that $A=B\ast\langle g'\rangle$. Then $g't^{-1}g_1t$ is primitive in $F_N$ (because if $\{b_1,\dots,b_{N-2}\}$ is a free basis of $B$, then $\{b_1,\dots,b_{N-2},t,g't^{-1}gt\}$ is a free basis of $F_N$). We have 

\begin{displaymath}
||g't^{-1}g_1t||_{T}=2(l-l_1)+d_{\overline{T_{N-1}}}(p_2,g'p_2), 
\end{displaymath}

\noindent and as $d_{\overline{T_{N-1}}}(p_1,g_1p_1)=0$, we also have 

\begin{displaymath}
||g't^{-1}g_1t||_{T'}=\left\{
\begin{array}{rl}
2(l-l'_1)+d_{\overline{T_{N-1}}}(p_2,g'p_2)&\text{~if~} g'_1=g_1^{\pm 1}\\
2l+d_{\overline{T_{N-1}}}(p_2,g'p_2)&\text{~if~} g'_1\neq g_1^{\pm 1}
\end{array}.
\right.
\end{displaymath}

As $T$ and $T'$ are simple-equivalent, as soon as $l_1>0$, we have $g_1={g'_1}^{\pm 1}$ and $l_1=l'_1$. Similarly, as soon as $l'_1>0$ and $g'_1$ is contained in a proper free factor of $A$, we have $g_1={g'_1}^{\pm 1}$ and $l_1=l'_1$. A similar argument also shows that if $l_2>0$ (resp. $l'_2>0$) and if $g_2$ (resp. $g'_2$) is contained in a proper free factor of $A$, then $g_2={g'_2}^{\pm 1}$ and $l_2=l'_2$. Hence $T$ and $T'$ are NS-pull-equivalent.
\end{proof}

\begin{figure}
\begin{center}
\input{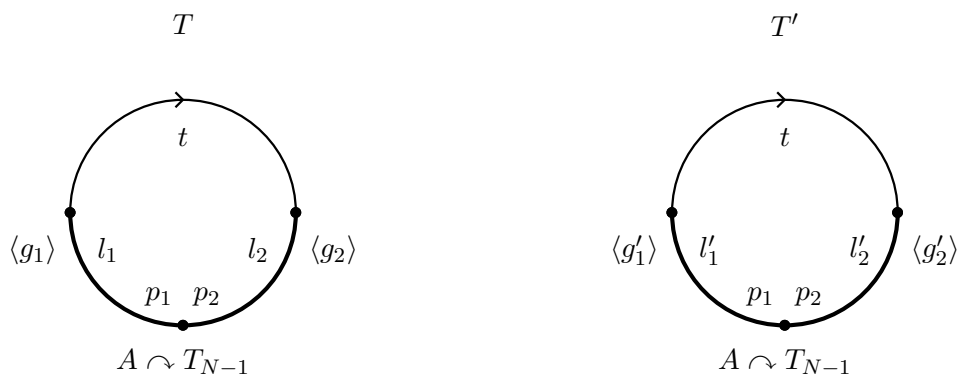}
\caption{The trees $T$ and $T'$ in Case 2 of the proof of Proposition \ref{weak-horo-strong}.}
\label{fig-the-graphs-2}
\end{center}
\end{figure}

\begin{proof} [Proof of Theorem \ref{equivalences}]
Theorem \ref{equivalences} follows from Propositions \ref{primitive-separable}, \ref{slide-separable}, \ref{separable-slide} and \ref{weak-horo-strong}.
\end{proof}

\bibliographystyle{amsplain}
\bibliography{/Users/Camille/Documents/Bibliographie}

\end{document}